\documentclass{amsart}
\usepackage[pagebackref=false,colorlinks=false, pdfstartview=FitV, linkcolor=blue,citecolor=red, urlcolor=blue]{hyperref}

\usepackage{amsmath,amsthm,amssymb,verbatim,mathrsfs}%

\newtheorem{thm}{Theorem}[section]
\newtheorem{cor}[thm]{Corollary}
\newtheorem{lem}[thm]{Lemma}
\newtheorem{prop}[thm]{Proposition}
\newtheorem{con}[thm]{Conjecture}
\newtheorem{theorema}{Theorem}

\theoremstyle{definition}

\theoremstyle{definition}

\newtheorem{rem}[thm]{Remark}

\newcommand\be{\begin{equation}}
\newcommand\ee{\end{equation}}
\newcommand\bee{\begin{equation*}}
\newcommand\eee{\end{equation*}}
\newcommand\ben{\begin{enumerate}}
\newcommand\een{\end{enumerate}}

\def\G{\ensuremath {{\bf G}}}
\def\CC{\ensuremath {{\mathbb C}}}
\def\H{\ensuremath {{\mathcal H}}}
\def\I{\ensuremath {{\mathcal I}}}

\def\R{\ensuremath {{\bf R}}}
\def\F{\ensuremath {{\mathcal F}}}

\def\A{\ensuremath {{\mathbb A}}}

\def\P{\ensuremath {{\mathscr P}}}
\def\E{\ensuremath {\mathscr E}}
\def\L{\ensuremath {\mathscr L}}
\def\Z{\ensuremath {{\mathbb Z}}}
\def\D{\ensuremath {\mathcal D}}
\def\C{\ensuremath {\mathcal C}}

\def\K{\ensuremath {\mathcal K}}
\def\G{\ensuremath {\mathscr G}}
\def\M{\ensuremath {\mathscr M}}

\def\SI{\ensuremath {S\mathcal{I}}}
\def\SS{\ensuremath {S\mathcal I}}
\def\g{\ensuremath {\mathfrak g}}

\def\o{\ensuremath {\mathfrak o}}
\def\O{\ensuremath {\mathscr O}}

\def\t{\ensuremath {\mathfrak t}}

\def\a{\ensuremath {\mathfrak a}}
\def\tr{\ensuremath {\textnormal{tr}}}
\def\vol{\ensuremath {\textnormal{vol}}}

\def\bs{\ensuremath {\backslash}}
\def\ac{\ensuremath {\mathrm{ac}}}
\def\temp{\ensuremath {\mathrm{temp}}}
\def\el{\ensuremath {\mathrm{ell}}}
\def\orb{\ensuremath {\mathrm{orb}}}
\def\disc{\ensuremath {\mathrm{disc}}}
\def\unit{\ensuremath {\mathrm{unit}}}
\def\ram{\ensuremath {\mathrm{ram}}}
\def\reg{\ensuremath {\mathrm{reg}}}
\def\unip{\ensuremath {\mathrm{unip}}}
\def\ss{\ensuremath {\mathrm{ss}}}
\def\der{\ensuremath {\mathrm{der}}}
\def\adm{\ensuremath {\mathrm{adm}}}
\def\pa{\ensuremath {\mathrm{par}}}
\def\uns{\ensuremath {\mathrm{uns}}}

\numberwithin{equation}{section}

\title[A weighted stable trace formula I]{A weighted stable trace formula I:\\ Basic functions}
\author{Tian An Wong}
\email{{\href{tiananw@umich.edu}{tiananw@umich.edu}}}
\address{University of Michigan, 4901 Evergreen Rd, Dearborn, MI 48128, US}
\subjclass[2010]{22E55  (primary), 11R39 (secondary). }
\keywords{Stable trace formula, basic functions, Langlands-Shelstad transfer}

\begin{document}
\begin{abstract}
We establish endoscopic and stable trace formulas whose discrete spectral terms are weighted by automorphic $L$-functions, by the use of basic functions that are incorporated into the global spectral and geometric coefficients. This is a continuation of a previous work of the author which established the corresponding weighted invariant trace formula for noncompactly supported test functions. The meromorphic continuation of these weighted trace formulas would yield $r$-trace formulas, and can therefore be seen as precursors to them. Along the way, we formulate a weighted form of the Langlands-Shelstad transfer conjecture, generalizing the weighted fundamental lemma of Arthur, that is suggested by our analysis but not central to the main result.  
 \end{abstract}

\maketitle

\tableofcontents \addtocontents{toc}{\protect\setcounter{tocdepth}{1}}

\section{Introduction}

\subsection{Background}

Given a reductive group $G$ over a number field $F$, the stable trace formula is an identity of stable distributions on $G$, expressed as 
\begin{align*}
S(f) & = \sum_{M\in\L}|W^M_0||W^G_0|^{-1}\int_{\Phi(M,V,\zeta)}b^{M}(\phi)S_M(\phi,f)d\phi \\
&= \sum_{M\in\L}|W^M_0||W^G_0|^{-1}\sum_{\delta\in\Delta(M,V,\zeta)}b^{M}(\delta)S_M(\delta,f) ,
\end{align*}
which has been used, for example, for the endoscopic classification of automorphic representations of various classical groups. To obtain it one first makes the trace formula invariant, and its stabilization in turn depends crucially upon the weighted fundamental lemma. This paper is a continuation of \cite{witf}, the goal of which is to weight the spectral side of the stable trace formula with automorphic $L$-functions by the use of certain basic functions. In \cite{witf} we established the extension of the invariant trace formula to allow for basic functions, and in this paper we shall stabilize it.

Fixing a central induced torus $Z$ of $G$ with an automorphic character $\zeta$, we let $V$ be a large finite set of valuations $v$ of $F$ outside of which $G,Z,$ and $\zeta$ are unramified. Let 
\[
G_V = \prod_{v\in V} G(F_v), \quad G^V = \prod_{v\not\in V} G(F_v).
\]
The stable linear form $S(f)$ established by Arthur is valid only for test functions of the form 
\[
\dot{f} = f \times u^V
\]
where, $f$ is a $\zeta^{-1}$-equivariant, compactly-supported function on $G_V$ and $u^V$ is the unit element of the $\zeta^{-1}$-equivariant Hecke algebra of $G^V$. In order to weight the trace formula with the associated $L$-functions, we shall instead use test functions of the form
\be
\label{bv}
f^r_s = f\times b^V
\ee
where $f$ is a $\zeta^{-1}$-equivariant, noncompactly-supported function on $G_V$ as in \cite[\S2]{witf}, the space of which we denoted by $\C^\circ(G,V,\zeta)$, and $b^V$ is constructed from the basic function that depends on an irreducible complex finite-dimensional representation $r$ of the $L$-group $^LG$ of $G$, and a complex number $s$ with Re$(s)$ large enough. We shall assume both of these to be fixed. The main result of \cite{witf} is then the existence of an invariant linear form 
\[
I^r_s(f) = I({f}^r_s), \qquad f\in \C^\circ(G,V,\zeta)
\]
valid for $\mathrm{Re}(s)$ large enough, which comes with the parallel expansions
\begin{align*}
\label{theorema}
&  \sum_{M\in\L}|W^M_0||W^G_0|^{-1}\int_{\Pi(M,V,\zeta)}a^{M}_{r,s}(\pi)I_M(\pi,f)d\pi\\
&= \sum_{M\in\L}|W^M_0||W^G_0|^{-1}\sum_{\gamma\in\Gamma(M,V,\zeta)}a_{r,s}^{M}(\gamma)I_M(\gamma,f).\notag
\end{align*}
The global coefficients $a^{M}_{r,s}(\pi)$ and $a_{r,s}^{M}(\gamma)$ that occur here are weighted forms of the coefficients $a^M(\pi)$ and $a^M(\gamma)$ that occur in Arthur's trace formula \cite{ITF1}. In particular, the spectral coefficient $a^{M}_{r,s}(\pi)$ is related to the unramified automorphic $L$-function $L^V(s,\pi,r)$. This can be seen as a step towards an $r$-invariant trace formula in the sense of Arthur \cite{problems}. If the weighted coefficients $a_{r,s}^M(\pi)$ and $a_{r,s}^M(\gamma)$ can be shown to have meromorphic continuation just to the left of the line Re$(s)=1$, then in principle the residue theorem can be applied to establish an $r$-trace formula {\em without} removing the contribution of the nontempered spectrum, what is generally seen to be the key obstruction in obtaining an $r$-trace formula \cite{FLN}. Indeed, the $L$-functions of nontempered automorphic representations are expected still to have meromorphic continuation, and thus the method that we propose here will establish a trace formula that carry information about these $L$-functions, which one might hope will be treated in the comparison of trace formulas.

\subsection{Main result} In this paper, we proceed to stabilize the linear form $I^r_s(f)$. Along the way, we formulate a weighted form of the Langlands-Shelstad transfer conjecture, namely the endoscopic transfer of weighted orbital integrals of the basic function (Conjecture \ref{conj}).  We shall use this transfer to streamline the stabilization of the unramified geometric terms in the trace formula, where before what was needed was the transfer of weighted orbital integrals of the unit element of the spherical Hecke algebra. The latter led to the weighted Fundamental Lemma conjectured by Arthur \cite[Conjecture 5.1]{STF1}, and which is now proved for split groups by the combined works of Waldspurger, Chaudouard, and Laumon \cite{charchange,CL}. Given that the ordinary Langlands-Shelstad transfer follows from the unweighted Fundamental Lemma, it seems natural to pose such an analogue. Indeed, Waldspurger has suggested such a possibility in the Lie algebra case \cite{traces}. The weighted fundamental lemma itself is of course the strongest evidence for this conjecture. By some stretch of the imagination, we might also take as evidence the transfer of weighted orbital integrals recently obtained by Li in the context of the Guo-Jacquet relative trace formula \cite{huajie}.

While our results will be stated in these terms, we emphasise that we shall assume Conjecture \ref{conj} only for convenience; it can be dropped by instead using the function $u^V_S\times b^S$ in place of $b^V$ in \eqref{bv}, where $S$ is large finite set of places depending on the support of $f$, and $u^V_S$ is the $\zeta^{-1}$-equivariant characteristic function of $K^V_S = \prod_{v \in S-V} K_v$, where $K_v$ is a fixed hyperspecial maximal compact subgroup of $G_v$. This would  be sufficient for capturing the poles of unramified automorphic $L$-functions on the spectral side, but comes at the expense of a slightly more complicated formula than the one obtained below. 

In any case, we establish a decomposition generalizing the stable trace formula \cite{STF1}, 
\[
I^r_s(f) = \sum_{G'}\iota(G,G')\hat{S}_{r,s}'(f'),
\]
for stable distributions $\hat{S}_{r,s}' = \hat{S}_{r,s}^{G'}$ on the endoscopic groups $G'$ for $G$, where $f\to f'$ is the ordinary Langlands-Shelstad transfer and $\iota(G,G')$ is the coefficient defined in \cite[Theorem 8.3.1]{Kcusp}. The transfer $f'$ is determined only up its stable orbital integrals and the distribution $f\to \hat{S}_{r,s}^{*}(f^*)$ that is indexed by $G'=G^*,$ the quasisplit inner form of $G$, can be regarded as the stable part of $I^r_s(f)$. Assume inductively that $S_{r,s}'$ exists and is stable.  In the case that $G$ is quasisplit, in which case $G^*=G$, we define
\[
S^G_{r,s}(f) = I^r_s(f) - \sum_{G'\neq G^*}\iota(G,G')\hat{S}_{r,s}'(f').
\]
The problem then is to show that the right-hand side is stable. If $G$ is not quasisplit, all the terms on the right-hand side are defined inductively. The decomposition of $I^r_s(f)$ above then represents an identity to be proved. 

Let $\C^\circ(G,V,\zeta) = \C^\circ(G_V^Z,\zeta_V)$ be the space of $\zeta_V^{-1}$-equivariant nonconpactly supported test functions on $G_V^Z$ in \cite[\S2.1]{witf}. It is a subspace of the $\zeta^{-1}$-equivariant Harish-Chandra Schwartz functions $\C(G,V,\zeta)$, and is the class of functions to which Finis-Lapid-M\"uller extended the trace formula to. The main result of this paper is the following.

\begin{theorema}
\textnormal{(a)}
 If $G$ is arbitrary, there is a linear form
\[
I^
\E_{r,s}(f) = I^\E({f}^r_s), \qquad f\in \C^\circ(G,V,\zeta)
\]
converging absolutely for $\mathrm{Re}(s)$ large enough. It comes with the parallel expansions
\begin{align*}
&  \sum_{M\in\L}|W^M_0||W^G_0|^{-1}\int_{\Pi(M,V,\zeta)}a^{M,\E}_{r,s}(\pi)I^\E_M(\pi,f)d\pi\\
&= \sum_{M\in\L}|W^M_0||W^G_0|^{-1}\sum_{\gamma\in\Gamma(M,V,\zeta)}a^{M,\E}_{r,s}(\delta)I^\E_M(\gamma,f).
\end{align*}
\textnormal{(b)}
Assume Corollary \ref{conk}. If $G$ is quasisplit, there is a stable linear form
\[
S^G_{r,s}(f) = S({f}^r_s), \qquad f\in \C^\circ(G,V,\zeta)
\]
converging absolutely for $\mathrm{Re}(s)$ large enough. It comes with the parallel expansions
\begin{align*}
&  \sum_{M\in\L}|W^M_0||W^G_0|^{-1}\int_{\Phi(M,V,\zeta)}b^{M}_{r,s}(\phi)S_M(\phi,f)d\phi\\
&= \sum_{M\in\L}|W^M_0||W^G_0|^{-1}\sum_{\delta\in\Delta(M,V,\zeta)}b^{M}_{r,s}(\delta)S_M(\delta,f).
\end{align*}
\end{theorema}

\noindent The global coefficients $b^{M}_{r,s}(\delta)$ and $b^{M}_{r,s}(\phi)$ are now weighted forms of the stable coefficients $b^M(\delta)$ and $b^M(\phi)$ that occur in the usual stable trace formula, which are naturally constructed from the `unstable' coefficients, and similarly for the endoscopic coefficients $a^{M,\E}_{r,s}(\gamma)$ and $a^{M,\E}_{r,s}(\pi)$.  Corollary \ref{conk} here is the weighted transfer of stable orbital integrals of basic functions, which we recall is possible to circumvent at the cost of using less elegant test functions. Moreover, the assumption that Re$(s)$ is sufficiently large is a matter of the choice of representation $r$ of $^LG$. In the cases where meromorphic continuation is known, for example the $L$-functions arising from the Langlands-Shahidi method, this restriction can also be relaxed. Along with the expectation of meromorphic continuation of automorphic $L$-functions, we also expect these weighted coefficients to also have meromorphic continuation to the entire complex plane.

Our proof generally follows  Arthur's method in the stabilization of the trace formula. In \cite[\S6,7]{STF1}, this takes the form of a pair of main global theorems and main local theorems, regarding nature of the global coefficients and local distributions that appear on either side of the expansions. It is perhaps worthy of note that the local distributions that occur in the weighted trace formula remain the same, and what is needed is simply an extension to $\C^\circ(G,V,\zeta)$. The analogues of the main local theorems will be proved in the appendix, whereas the main body of the paper is dedicated to stabilizing the new weighted global coefficients. (This may seem a little disjointed to the reader. The simple reason for this is because the results of Appendix \ref{conts} were proved in an earlier version of \cite{witf}, but have been moved to the current paper by recommendation of the referee.)

Theorem 1 is a step towards the $r$-stable trace formula, whereby the cuspidal spectral terms are nonzero only if $L(s,\phi,r)$ has a pole at $s=1$. Establishing it would amount to showing that the distribution $S^G_{r,s}(f)$ has meromorphic continuation to the point $s=1$, and obtaining an explicit formula for its residue there. From this it will follow that the residue distribution
\[
S^G_r(f) = 
-\frac{1}{2\pi i}\int_{\text{Re}(s)=1}S^G_{r,s}(f) x^s\frac{ds}{s}, \quad x>0,
\]
can be seen to be an $r$-trace formula, in the sense that the contribution of the cuspidal tempered terms for which the associated $L(s,\pi,r)$ is holomorphic in the half plane Re$(s)\ge1$ will be zero. Here the contour integral should be taken with a small semi-circle to the left of $s=1$ to detect the possible pole there. Crucially, this approach circumvents the difficulties of presented by the nontempered terms by allowing the contribution of the nontempered discrete spectrum to remain. Thus the distribution $S^G_r(f)$ will retain information about nontempered representations whose $L$-functions have poles in the half plane Re$(s)\ge1$ for a given $r$. Indeed, by the principle of functoriality one also expects such $L$-functions to have meromorphic continuation, so the distribution $S^G_r(f)$ will be well-defined, and thus justifies calling $S^G_r(f)$ an $r$-trace formula. An approximation of this was constructed by Getz for GL$(n)$ using certain nonabelian Fourier transforms following ideas of Braverman and Kazhdan \cite[\S5]{getz}.

The construction of the endoscopic and stable global coefficients are built from the global coefficients in $I_{s}^r(f)$, and their meromorphic continuation follows from that of the original coefficients. We recall that the spectral coefficients $b^M_{r,s}(\phi)$ contain information related to unramified automorphic $L$-functions, hence showing their meromorphic continuation to $s=1$ amounts to doing the same for the $L$-functions.  As is typical with the trace formula, rather than doing this directly we can consider the geometric coefficients $b^M_{r,s}(\delta)$ instead, which we shall study in the sequel.

We conclude with a brief summary of the paper. Section \ref{sectionconj} formulates the weighted transfer conjecture, which we use in Section \ref{sectionunrg} to stabilize the unramified terms on the geometric side of the trace formula. In Section \ref{sectionunrs}, we stabilize the unramified terms on the spectral side of the trace formula. These unramified terms occur as the new weighted terms in the global coefficients on either side, and represent the main technical innovations of this paper. This rest of the argument is in fact largely formal, following the method of Arthur's stabilization quite faithfully, as long as we pay attention to the the spaces of functions involved. We formulate the analogues of Arthur's main global theorems in Theorems \ref{globge} and \ref{globsp}, whose proofs will take up the rest of the paper. In Section \ref{sectiondesc} we prove a descent formula for the global geometric coefficients along the lines of \cite{STF2}, which reduces the study of the geometric coefficients to unipotent elements. Finally, in Section \ref{sectionstab} we complete the proofs of the global theorems and thereby establish the endoscopic and stable trace formulas. The analogous main local theorems needed for the stabilisation are proved in the appendix. 

\section{A weighted transfer conjecture}
\label{sectionconj}

The stable trace formula established by Arthur depends on a weighted Fundamental Lemma, which is a generalization of the Fundamental Lemma originally conjectured by Langlands and Shelstad \cite{LS}. The unweighted version was proved by Ng\^{o} \cite{Ngo} for Lie algebras in positive characteristic, from which the required result follows. It implies the Langlands-Shelstad transfer of orbital integrals \cite{transfert}, which is required in the various endoscopic constructions necessary for the stabilization of the trace formula. On the other hand, the weighted Fundamental Lemma for Lie algebras has been proved by Chaudouard and Laumon in the split case and in positive characteristic \cite{CL}, and by Waldspurger this yields the weighted Fundamental Lemma in characteristic zero \cite{charchange}. Similar methods are expected apply to the general case of quasisplit groups, which the full stabilization of the trace formula is conditional upon.  In our case, the stabilization of the unramified terms in the trace formula will require a transfer of weighted orbital integrals instead. 

In order to formulate our statements precisely, we first fix some necessary notation. Most of the notation and constructions that we employ in this paper are chosen to be consistent with those of \cite{STF1,STF3}, while the main variations occur where the basic functions are involved. 

\subsection{$K$-groups}

We will work in the setting of a multiple group $G$, which is an algebraic variety whose components $G_\alpha$ are reductive groups over a field $F$,
\[
G =\coprod_{\alpha\in\pi_0(G)}G_\alpha,
\]
with an equivalence class of frames 
\[
(\psi,u) = \{(\psi_{\alpha\beta},u_{\alpha\beta}):\alpha,\beta\in \pi_0(G)\},
\]
where $\psi_{\alpha\beta}:G_\beta\to G_\alpha$ is an $\bar{F}$-isomorphism and $u_{\alpha\beta}:\Gamma_F\to G_{\alpha,\text{sc}}$ is a locally constant function from the absolute Galois group $\Gamma_F$ of $F$ to the simply connected cover $G_{\alpha,\text{sc}}$ of the derived group of $G_\alpha$. The class of $(\psi,u)$ satisfies compatibility conditions described in \cite[\S4]{STF1}. If $F$ is a local field, we call a multiple group $G$ a $K$-group if the $u_{\alpha\beta}$ are 1-cocycles and for each $\alpha$, the image of the map from $\{u_{\alpha\beta}:\beta\in\pi_0(G)\}$ to $H^1(F,G_\alpha)$ is in bijection with the image of $H^1(F,G_{\alpha,\text{sc}})$ in $H^1(F,G_\alpha)$. If $F$ is a global field, we call a multiple group $G$ a $K$-group if it satisfies the preceding properties, and it has a local product structure, i.e., there is a family of local $K$-groups $(G_v,F_v)$ indexed by the valuations of $F$, and a family of homomorphisms $G\to G_v$ over each completion $F_v$, whose restricted direct product $G\to \prod_v G_v$ is an isomorphism of schemes over the adele ring $\A$ of $F$. Such a structure determines a surjective map 
\[
\alpha\to \alpha_V=\prod_{v\in V}\alpha_v, \quad \alpha\in\pi_0(G),\alpha_v\in\pi_0(G_v)
\]
of components, for any finite set of valuations $V$ of $F$, and is bijective if $V$ contains the infinite places $V_\infty$ of $F$. Suppose that $G$ is a $K$-group, and that $G^*$ is a quasisplit inner twist of $G$, which determines a quasisplit inner twist $G_v^*$ of each of the local $K$-groups $G_v$. We call $G$ an inner $K$-form of $G^*$, and we call $G$ quasisplit if at least one of its components is quasisplit over the global field $F$.

We shall fix a central induced torus $Z$ of $G$, and a character $\zeta$ of $Z(\A)/Z(F)$. It comes with central embeddings $Z\stackrel{\sim}{\rightarrow} Z_\alpha$ over $F$ for each $\alpha\in \pi_0(G)$ that are compatible with the isomorphisms $\psi_{\alpha\beta}$. Locally, we have pairs $Z_v$ and $\zeta_v$ where $\zeta_v$ is a character on $Z_v$. We then write $G_V^Z$ for the set of  $x \in G_V= \prod_{v\in V}G_v$ such that $H_G(x)$ lies in the image of the canonical map from $\a_Z$ to $\a_G$. We shall generally assume that the finite set $V$ contains the set places $V_\ram(G,\zeta)$ over which $G,Z,$ and $\zeta$ are ramified. We also have natural notions of Levi subgroups and parabolic subgroups of a $K$-group $G$, whence we denote by $\L(M)$ to be the collection of Levi subgroups of $G$ containing $M$, $\L^0(M)$ the subset of proper Levi subgroups in $\L(M)$, and $\P(M)$ the collection of parabolic subgroups of $G$ containing $M$. We shall fix a minimal Levi $M_0$ of $G$, and write $\L=\L(M_0)$ and $\L^0=\L^0(M_0)$. For any $M\in \L$, we have the real vector space $\a_M= \text{Hom}(X(M)_F,\R)$, and the set 
\[
\a_{M,V}=\{H_M(m) : m\in M(F_V)\}
\]
is a subgroup of $\a_M$, and $F_V = \prod_{v\in V} F_v$. It is equal to $\a_M$ if $V$ contains an archimedean place, and is a lattice in $\a_M$ otherwise.

\subsection{Hecke algebras}

Now, suppose we have fixed a maximal compact subgroup $K_v$ of $G_v$ at each place $v$. Following \cite[\S1]{STF1}, we define 
\[
\H(G^Z_V,\zeta_V)= \H(G,V,\zeta)
\]
to be the $\zeta^{-1}_V$-equivariant Hecke algebra of $G_V^Z$, consisting of smooth, compactly-supported $K_V$-finite functions on $G_V$. We similarly define $\H(G_v,\zeta_v)$ to be the local $\zeta_v^{-1}$-equivariant Hecke algebra. If $\H(G_v)$ is the usual Hecke algebra of $G_v$, then one has a natural projection onto $\H(G_v,\zeta_v)$ given by sending any $f_v\in \H(G_v)$ to the function
\be
\label{zetav}
f_v^\zeta(x) =  \int_{Z_v}f_v(zx)\zeta_v(z) dz,\quad x\in G_v
\ee
in $\H(G_v,\zeta_v)$. 

We also require the larger space of almost compactly-supported functions $\H_\ac(G_v)$ introduced in \cite[\S11]{res1}, defined in the following manner. Let $\Gamma$ be a finite subset of irreducible representations of $K_v$, and let $\H(G_v)_\Gamma$ be the subspace of functions in $\H(G_v)$ that transform on each side under $K_v$ according to representations in $\Gamma$.  We then define $\H_\ac(G_v)_\Gamma$ to be the space of functions on $G_v$ such that for every $g\in C_c^\infty(\a_{G,v})$ the function 
\[
f^b(x) = f(x)g(H_G(x))
\]
belongs to $\H(G_v)_\Gamma$. We then form the direct limit over finite subsets $\Gamma$,
\[
\H_\ac(G_v) = \varinjlim_\Gamma \H_\ac(G_v)_\Gamma.
\]
We may also view $\H_\ac(G_v)$ as the space of uniformly smooth functions $f$ in $\H(G_v)$ such that for any $X\in \a_{G,v}$, the restriction of $f$ to the preimage of $X$ in $G_v$ is compactly supported. By uniformly smooth, we mean that the function $f$ is bi-invariant under an open compact subgroup of $G_v$. We can then define 
\[
\H_\ac(G_v,\zeta_v) = \{ f^\zeta : f \in \H(G_v,\zeta_v)\}
\]
to be the image of $\H_\ac(G_v)$ under the mapping $f\mapsto f^\zeta$ given by \eqref{zetav}, containing $\H(G_v,\zeta_v)$. Similarly, we define
\[
\H_\ac(G_V^Z,\zeta_V) = \H_\ac(G,V,\zeta) = \{f^\zeta : f \in \H(G,V,\zeta)\}.
\]We topologise these spaces in the manner described in \cite[\S11]{res1}. As remarked in \cite[\S23]{art}, the invariant linear forms $I_M(\gamma,f)$ and $I_M(\pi,f)$ that we shall study are determined by their restriction to $\H(G_V)$ of $\H_\ac(G_V)$, and moreover that $I_M(\gamma,f)$ can be extended continuous to a linear form on the Schwartz space $\C(G_V)$.

\subsection{Basic functions}

We can now introduce the basic functions that enter into the unramified terms in the trace formula. Let $r$ be an irreducible finite-dimensional complex representation of the $L$-group $^L{G}$ of $G$. Suppose that $V$ contains $V_\infty$ and $V_\ram(G,\zeta)$, and that $r$ is unramified outside of $V$. Let 
\[
c=\{c_v: v\not\in V\}
\]
be a family of semisimple conjugacy classes in the local $L$-group $^LG_v = \hat{G}\rtimes W_{F_v}$ whose image in the Weil group $W_{F_v}$ is a Frobenius element. We also assume that the image of $c_v$ under the projection $^LG_v \to {^LZ_v}$ is the conjugacy class in $^LZ_v$ associated to $\zeta_v$, and that for every $\hat{G}$-invariant polynomial $A$ on $^LG$, we have $|A(c_v)|\le q^{r_A}_v$ for each $v\not\in V$ and for some constant $r_A>0$ depending only on $A$. To such an element we may associate the unramified local $L$-function
\[
L_v(s,c,r) = \det(1-r(c_v)q^{-s}_v)^{-1}, \quad s\in \CC
\]
which is analytic in some right half plane. It can be expressed as the formal series
\[
\sum_{k=0}^\infty \tr((\text{Sym}^k r)(c_v))q^{-ks}_v,
\]
where $q_v$ is the cardinality of the residue field of $F_v$. Let us write $\H(G_v,K_v)$ for the spherical Hecke algebra of $G_v$ with respect to $K_v$, and define $\H_\ac(G_v,K_v)$ as the space of bi-$K_v$-invariant functions in $\H_\ac(G_v)$. Suppose we have fixed a Borel pair $(B_v,T_v)$ of $G_v$ defined over $F_v$ such that the torus $T_v$ splits over a finite unramified extension of $F_v$. We shall also assume that $K_v$ is a hyperspecial maximal compact subgroup that lies in the apartment of $T_v$, and write $K_{T_v} = K_v\cap T_v$. Let $W_v = N_{G_v}(T_v)/T_v$ be the Weyl group of $T_v$ in $G_v$. We recall that the Satake isomorphism gives a bijection
\[
\H(G_v,K_v) \to \H(T_v,K_{T_v})^{W_v},
\]
where the right hand side can be identified with the coordinate algebra of $W_v$-invariant regular functions on the complex torus $\hat{T}_v = X(T_v)\otimes \CC$. The isomorphism extends to the slightly larger space as a consequence of \cite[Proposition 2.7]{Li}, 
\be
\label{satake}
\H_\ac(G_v,K_v) \to \H_\ac(T_v,K_{T_v})^{W_v}
\ee
where the left hand side now consists of $W_v$-invariant formal series on the same torus. 

The basic function $b_v$ belongs to $\H_\ac(G_v,K_v)$ by \cite[Lemma 2.1]{witf}, which maps to an element in $\H_\ac(G_v,\zeta_v)$ by \eqref{zetav}. It also belongs to the Fr\'echet algebra
\[
\C^\circ(G,V,\zeta)=\C^\circ(G_V^Z,\zeta_V) 
\]
defined in \cite[\S2.1]{witf}, modeled after the space introduced by Finis, Lapid, and M\"uller to extend the trace formula. It is topologised using a family of seminorms on right $K_V$-invariant functions and then taking an inductive limit. The Hecke algebra $\H(G,V,\zeta)$ is dense in $\C^\circ(G,V,\zeta)$, which contains $\C(G,V,\zeta)$ but not $\H_\ac(G,V,\zeta)$ as a subalgebra.  If $G$ is a connected reductive group, then $b_v=b^r_{v,s}$ is the unramified spherical function in $\H_\ac(G_v,\zeta_v)$ whose character at the representation $\pi_v=\pi(c_v)$ associated to $c_v$ satisfies
\[
b_{G,v}(c_v) = \tr(\pi_v(b_v))=L_v(s,c,r)
\]
for any $s\in\CC$ with real part large enough. If $G$ is a $K$-group, we note that each of the $G_\alpha$ are related by inner twists and therefore share a common dual group $\hat{G}$. We can thus write $b_v = \oplus_{\alpha_v} b_{v,\alpha_v}$ and to each $c_v$ we have the identity
\be
\label{bgl}
b_{G,v}(c_v) = |\pi_0(G_v)| L_v(s,c_{v},r).
\ee
Throughout this paper, we shall assume $r$ and $s$ to be fixed with Re$(s)$ large unless otherwise specified.

\begin{rem}As an aside, we note that if $G$ is semisimple,  one can also embed $\hat{G}$ into a reductive group $\hat{G}_r$ whose center $Z(\hat{G})$ is isomorphic to $\CC^\times$ \cite{cass}, and extend $r$ to a 	representation of $\hat{G}_r$ such that $r(z) = zI$, namely,
\[
\hat{G}_r = \frac{\CC^\times\times \hat{G}}{\{(r(z),z^{-1}):z\in Z(\hat{G})\}}.
\]
This provides a homomorphism $\det_v:G_{r,v}\to F^\times_v$ for $v\not\in V$, in which case $b^r_{v,s}$ is equal to the product of $b^r_{v,0}$ and $|\det_v|^s$. We caution that while the local $L$-factor, and hence $b^r_{v,s}$ may be defined at $s=0$, it is by no means clear that the product over all $v\not\in V$ converges, as it amounts to the meromorphic continuation of the unramified global $L$-function.\end{rem}

\subsection{Endoscopic data}

The endoscopic datum for $G$ consists of a connected quasisplit group $G'$ over $F$, embedded in a larger datum $(G',\G',s',\xi')$, where we recall that $G'$ is a quasisplit group over $F$, $\G'$ is a split extension of $W_F$ by a dual group $\hat{G}'$ of $G'$, $s'$ is a semisimple element in $\hat{G}$, and $\xi'$ is an $L$-embedding of $\G'$ into $^LG$ \cite[(1.3)]{LS}. It is required that $\xi'(\hat{G}')$ be equal to the connected centralizer of $s'$ in $\hat{G}$, and that $\xi'(u')s'=s'\xi'(u')a(u')$ for any $u'\in\G'$ and locally trivial cocycle $a\in H^1(W_F,Z(\hat{G})$. We say that an endoscopic datum $G'$ is elliptic if the connected component of the identity in the centralizer of $\xi'(\G')$ in $\hat{G}$ is trivial, which is the same as saying that the image of $\xi'$ in $^LG$ is not contained in $^LM$ for any proper Levi subgroup $M$ of $G$ over $F$.

Let $G_{v,\alpha_v}$ be a connected component of $G_v$. We write $\Gamma_\reg(G_{v,\alpha_v})$ for the set of strongly regular conjugacy classes in $G_{v,\alpha_v}$, and also
\[
\Gamma_\reg(G_v) = \coprod_{\alpha_v\in\pi_0(G_v)}\Gamma_\reg(G_{\alpha_v,v}).
\]
We shall write $\E(G)$ for the set of isomorphism classes of endoscopic data for $G$ over $F$ that are locally relevant to $G$, in the sense that for every $v$, $G'_v$ has a strongly $G$-regular element that is an image of some class in $\Gamma_\reg(G_v)$. If $V$ is a finite set of valuations of $F$ that contains $V_\ram(G)$, we write $\E(G,V)$ for the subset of elements $G'\in\E(G)$ that are unramified outside of $V$. We also write $\E_\el(G)$ and $\E_\el(G, V )$ for the subset of elements in $\E(G)$ and $\E(G, V)$, respectively, that are elliptic over $F$. Given any $G'\in\E_\el(G)$, we shall fix auxiliary data $(\tilde{G}',\tilde\xi')$ where $\tilde{G}'\to G'$ is a central extension of $G'$ by an induced torus $\tilde{C}'$ and $\tilde\xi':\G'\to {^L\tilde{G}'}$ is an $L$-embedding satisfying the conditions of \cite[Lemma 2.1]{LCR}. 

\subsection{Transfer factors}

To discuss the endoscopic transfer of functions, we have to define the transfer factors that we shall use, following \cite[\S4]{STF1}. Let $\gamma_v$ be a strongly regular element of $G_v$ and $\delta'_v$ a strongly $G$-regular element in $\tilde{G}'_v$. Suppose that $V$ contains $V_\ram(G)$. We shall use the local transfer factors for $K$-groups 
\[
\Delta(\delta'_v,\gamma_v)= \Delta_{K_v}(\delta'_v,\gamma_v),
\]
normalized according to \cite[\S4]{STF1}. In particular, they depend a priori on a choice of local base points $\bar\delta'_v\in\tilde{G}'(F)$ and $\bar\gamma_v\in G_\alpha(F)$ for some $\alpha\in\pi_0(G)$, but can in fact be normalized to depend only on the choice of hyperspecial maximal compact subgroup $K_v$ of $G_v$. 
We recall that two elements $\gamma\in G_{\alpha,V}$ and $\delta\in G_{\beta,V}$ are called stably conjugate if $\psi_{\alpha\beta}(\delta)$ is stably conjugate to $\gamma$ in $G_{\alpha,V}$. Let $\Delta_{G\text{-reg}}(\tilde{G}_V')$ be the the set of $G$-regular stable conjugacy classes in $\tilde{G}'_V$. The transfer factor for $\gamma_V\in\Gamma_\reg(G_V)$ and $\delta_V'\in \Delta_{G\text{-reg}}(\tilde{G}_V')$ is then defined as the product
\[
\Delta(\delta'_V,\gamma_V) = \prod_{v\in V}\Delta(\delta'_v,\gamma_v)
\]
of local transfer factors, which depends only the hyperspecial maximal compact subgroup $K^V = \prod_{v \not\in V}K_v$.  We shall also consider local endoscopic data $G'_v$ of the local $K$-groups $G_v$, and write $\E(G_V)$ for the set of products $G'_V = \prod_{v\in V} G'_v$ of $G'_v\in \E(G_v)$. For such a $G'_V$, the transfer factor 
\[
\Delta(\delta'_V,\gamma_V),\qquad \delta'_V\in \prod_{v\in V}\Delta_{G\text{-reg}}(G_v),\  \gamma_V\in \prod_{v\in V}\Gamma_\reg(G_v)
\]
is a local object that does not depend on the local base point or the corresponding auxiliary data $(\tilde{G}'_V,\tilde\xi'_V)$.

Let us also define $\tilde\Delta_\reg^\E(G_V)$ to be the quotient of the family of elements 
\[
\{(G'_V,\tilde\xi'_V,\delta'_V):G'_V\in \E(G_V), \tilde\xi'_V:\G'_V\to {^L\tilde{G}'_V},\delta'_V\in\Delta_{G\text{-reg}}(\tilde{G}'_V)\}
\]
 that are $G_V$-images, up to the equivalence defined in \cite[\S2]{LCR} for the set that was denoted as $\tilde\Gamma^\E(G_v)$ therein. We similarly define $\Delta^\E_\reg(G_V)$ for the quotient of the family of elements in $(G'_V,\delta'_V)$ that are $G_V$-images, where $G'_V\in \E(G_V)$ and $\delta'_V\in \Delta_{G\text{-reg}}(G_V')$. The transfer factor depends only on the image $\delta_V$ of $\delta'_V$ in $\tilde\Delta_\reg^\E(G_V)$. We can therefore regard the transfer factor
\[
\Delta(\delta_V,\gamma_V) = \Delta(\delta_V',\gamma_V)
\]
as a function on $\tilde\Delta_\reg^\E(G_V)\times \Gamma_\reg(G_V).$

\subsection{Orbital integrals}

Suppose for the moment that $G$ is a connected reductive group over $F$. Given an element $\gamma_V\in G^Z_V$, let $G_{\gamma_V}=\prod_{v}G_{\gamma_v}$ be its centralizer and $D(\gamma_V)=\prod_vD(\gamma_v)$ the Weyl discriminant for some choice of invariant measure on the quotient $G_{\gamma_V}\cap G^Z_V\bs G^Z_V$. Then for any $f\in \C(G^Z_V)$, we define the orbital integral of $f$ at $\gamma_V$ as
\[
f_G(\gamma_V) =|D(\gamma_V)|^\frac12\int_{G_{\gamma_V}\cap G^Z_V\bs G^Z_V}f(x^{-1}\gamma_V x)dx
\]
and the $\zeta_V$-equivariant analogue
\be
\label{orb}
\int_{Z_V}\zeta_V(z)f_G(z\gamma_V)dz.
\ee
The latter belongs to the space of $\zeta_V$-equivariant distributions $\D(G^Z_V,\zeta_V)$ that are invariant under $G^Z_V$-conjugation and supported on the preimage in $G^Z_V$ of a finite union of conjugacy classes in $\bar G^Z_V=G^Z_V/Z_V$. We write $\D_\orb(G^Z_V,\zeta_V)$ for the subspace of $\D(G^Z_V,\zeta_V)$ spanned by \eqref{orb}. The spaces are equal if $V$ contains only nonarchimedean places, and contain more general distributions otherwise. We shall fix a basis $\Gamma(G^Z_V,\zeta_V)$ of $\D(G^Z_V,\zeta_V)$, satisfying natural compatibility conditions described in \cite[p.186]{STF1}. For example, the intersection $\Gamma_\orb(G^Z_V,\zeta_V)=\Gamma(G^Z_V,\zeta_V)\cap \D_\orb(G^Z_V,\zeta_V)$ is required to be again a basis of $\D_\orb(G^Z_V,\zeta_V)$.  Given an element $\gamma_V$ in $G_V^Z$, we can identify it with an element in $\D_\orb(G^Z_V,\zeta_V)$ by sending $f \in \C(G^Z_V)$ to the $\zeta_V$-equivariant orbital integral of $f$ at $\gamma_V$.

We can further define the weighted orbital integral at a strongly regular conjugacy class $\gamma_V$ of $M_V^Z$ as
\[
J_M(\gamma_V,f) =|D(\gamma_V)|^\frac12\int_{G_{\gamma_V}\cap G^Z_V\bs G^Z_V}f(x^{-1}\gamma_V x)v_M(x)dx
\]
where $v_M(x)$ is the volume of a certain convex hull depending on $x$ and $M$. The weighted orbital integrals at singular elements are more complicated to define, and are described in \cite[\S1]{realgerms} for real groups and \cite[(6.2)]{local} for $p$-adic groups. The $\zeta_V$-equivariant analogues are then defined in a similar fashion as \eqref{orb}, and the extension of these definitions to a $K$-group $G$ is largely formal (cf. \cite[\S3]{ArtTW}).

Let $S$ be a large finite set of valuations of $F$ containing $V\supset V_\infty$. The unramified weighted orbital integrals will be defined at the places $v$ in $S-V$. Let us write 
\[
G^V_S = \prod_{v\in S-V}G_v, \quad  Z^V_S = \prod_{v\in S-V}Z_v, \quad K^V_S = \prod_{v\in S-V}K_v,
\]
and define $\K(\bar{G}^V_S)$ to be the set of conjugacy classes in $\bar{G}^V_S=G^V_S/Z^V_S$ that are bounded, meaning that for each $v\in S-V$, the image of each element lies in a compact subgroup of $G_v$. Any element $k\in \K(\bar{G}^V_S)$ determines a distribution $\gamma^V_S(k)$ in $\Gamma_\orb(\bar{G}^V_S,\zeta^V_S)$. Given $\gamma\in \Gamma(G^Z_V,\zeta_V)$ and $k\in\K(\bar{G}^V_S)$, we write $\gamma\times k = \gamma\times \gamma^V_S(k)$ for the associated element in $\Gamma(G^Z_S,\zeta_S)$. Let $u_v$ be the function on $G_v$ supported on $K_vZ_v$ such that $u_v(kz) = \zeta^{-1}(z)$ for any $k\in K_v$ and $z\in Z_v$. We define the unramified weighted orbital integral
\[
r^G_M(k) = J_M(\gamma^V_S(k),u^V_S)
\]
on $\K(\bar{M}^V_S)$. 

\subsection{Stable distributions}

We are almost ready to state the weighted transfer conjecture. Suppose that $c$ belongs to the set of semisimple conjugacy class $\Gamma_\ss(G^Z_V)$ of $G^Z_V$. We write $\D_c(G^V_Z,\zeta_V)$ for the subspace of distributions in $\D(G^V_Z,\zeta_V)$ whose preimages in $G^Z_V$ are supported on conjugacy classes with semisimple part equal to $c$. We then write $\Gamma_\ss(\bar{G}^Z_V,\zeta_V)$ for the classes $c$ such that $\D_c(G^Z_V,\zeta_V)$ is nonzero. As with \cite[\S5]{STF1}, we shall assume for simplicity that $\Gamma_\ss(\bar{G}^V_S,\zeta^V_S)$ is equal to $\Gamma_\ss(G^S_V)$, so that there is an injection from $\K(\bar{G}^V_S)$ into $\Gamma(\bar{G}^V_S,\zeta^V_S)$. 

We call a distribution on $G^Z_V$ stable if it lies in the closed linear span of the strongly regular, stable
orbital integrals
\[
f^G(\delta_V)=\sum_{\gamma_V} f_G(\gamma_V)
\]
where $\delta_V$ is any strongly regular, stable conjugacy class in $G^Z_V$ and the sum runs over the finite set of conjugacy classes in $\delta_V$. Let $S\D(G^V_Z,\zeta_V)$ be the subspace of stable distributions in $\D(G^V_Z,\zeta_V)$. We can identify any strongly regular element $\delta\in\Delta_{\reg}(G^Z_V)$ with a subset $\Delta_\reg(G^Z_V,\zeta_V)$ of $S\D(G^Z_V,\zeta_V)$ generated by $f^G(\delta_V)$ as in \cite[\S1]{STF1}. Similarly we fix a subset $\Delta_{G\text{-reg}}(\tilde{G}'_V,\tilde\zeta')$ of $G$-regular elements in $S\D(\tilde{G}'_V,\tilde\zeta_V')$. The transfer factor can then be converted to a function on $\Delta_{G\text{-reg}}(\tilde{G}'_V,\tilde{\zeta}'_V)\times \Gamma_\reg(G_V,\zeta_V)$ by \cite[\S4]{STF1}. 

\subsection{Fundamental Lemma}

For this section, let $F$ be a nonarchimedean local field, and assume that $G,Z$, and $\zeta$ are unramified over $F$. Following \cite[\S5]{STF1}, we write $\K_\reg(\bar{G})$ for the set of strongly regular conjugacy classes of $\bar{G}(F)=G(F)/Z(F)$ that are bounded, and $k\to \gamma(k)$ for the canonical injection from $\K_\reg(\bar{G})$ to $\Gamma_\reg(G,\zeta)=\Gamma_\reg(G_v,\zeta_v)$. We also write $\L_\reg(\bar{G})$ for the set of strongly regular stable conjugacy classes in $\bar{G}(F)$ that are bounded, and $\ell\to\delta(\ell)$ for the corresponding injection from $\L_\reg(\bar{G})$ to $\Delta_\reg(G,\zeta)=\Delta_\reg(G_v,\zeta_v)$. If $G'$ is any local endoscopic datum for $G$, the normalized transfer factor $\Delta_K(\delta',\gamma)$ attached to the hyperspecial maximal compact subgroup $K$ of $G(F)$ is a canonical function on $\Delta_{G\text{-reg}}(\tilde{G}',\tilde{\zeta}')\times \Gamma_\reg(G,\zeta)$. It depends on the choice of auxiliary data $(\tilde{G}',\tilde\xi')$ attached to $G'$. If $G'$ is unramified, there is a canonical class of admissible embeddings of $^LG'$ into $^LG$, so that we may set $\tilde{G}'=G'$. We set
\[
\Delta_K(\ell',k) = \Delta_K(\delta'(\ell'),\gamma(k)), \quad \ell'\in\L_{G\text{-reg}}(\bar{G}'), k \in \K_\reg(\bar{G})
\]
for the unramified endoscopic datum $G'$. It is independent of the character $\zeta$ and the choice of $\tilde\xi'$.

Now suppose that $M$ is a Levi subgroup of $G$ in good position relative to $K$. We write $r^G_M(k) = J_M(k,u)$ as before for $k\in\K_{G\text{-reg}}(\bar{M})$, and note that it is also independent of the central datum $(Z,\zeta)$. If $M'$ is an unramified endoscopic datum of $M$, we then have the transfer factor
\[
\Delta_{K\cap M}(\ell',k), \quad \ell'\in\L_{G\text{-reg}}(\bar{M}'), k \in \K_\reg(\bar{M}).
\]
We define $\E_{M'}(G)$ to be the set of endoscopic data $G'$ for $G$, taken up to translation of $s'$ by $Z(\hat{G})^\Gamma$ with $\Gamma=\Gamma_F$ the absolute Galois group of $F$, in which $s'$ lies in $s'_MZ(\hat{M})^\Gamma$, $\hat{G}'$ is the connected centralizer of $s'$ in $\hat G$, $\G'$ equals $\M'\hat{G}'$, and $\xi'$ is the identity embedding of $\G'$ in $^LG$. For each $G'\in\E_{M'}(G)$, we fix an embedding $M'\subset G'$ for which $\hat{M}'\subset \hat{G}'$ is a dual Levi subgroup. We also fix auxiliary data $(\tilde{G}',\tilde\xi')$ for $G'$, where $\tilde{G}'\to G'$ and $\tilde\xi':\G'\to{^L\tilde{G}'}$, that restrict to auxiliary data $(\tilde{M}',\tilde\xi'_M)$ for $M'$, where $\tilde{M}'\to M'$ and $\tilde\xi'_M:\G'\to{^L\tilde{M}'}$, and whose central character data $(\tilde{Z}',\tilde\zeta')$ coincide. Note that $G^*$ belongs to $\E_{M'}(G)$ if and only if $M'=M^*$. Finally, we define the coefficient 
\[
\iota_{M'}(G,G') = |Z(\hat{M}')^\Gamma/Z(\hat{M})^\Gamma||Z(\hat{G}')^\Gamma/Z(\hat{G})^\Gamma|^{-1}
\]
for any $G'\in \E_{M'}(G)$.

We can then state the weighted Fundamental Lemma \cite[Conjecture 5.1]{STF1}, which we recall is now a theorem in the case of split groups $G$ due to \cite[Th\'eor\`eme 8.7]{charchange} and \cite[Th\'eor\`eme 1.4.1]{CL}. We may assume that the central datum $(Z,\zeta)$ is trivial, since the functions involved are independent of it. 

\begin{thm}\label{WFL}
For each $G$ quasiplit and $M$, there is a function 
\[
s^G_M(\ell),\qquad \ell\in\L_{G\textnormal{-reg}}(M)
\]
with the property that for any $K$, any unramified elliptic endoscopic datum $M'$ for $M$, and any element $\ell'\in\L_{G\textnormal{-reg}}(M')$, the transfer
\[
r^G_M(\ell') = \sum_{k\in\K_{G\textnormal{-reg}}(M)}\Delta_{K\cap M}(\ell',k)r^G_M(k)
\]
equals
\be
\label{flsum}
\sum_{G'\in \E_{M'}(G)}\iota_{M'}(G,G')s^{{G}'}_{{M}'}(\ell').
\ee
\end{thm}

\noindent The function $s^G_M(\ell)$ is uniquely determined by the required identity. If $M'=M$, the quasisplit group $G$ belongs to $\E_{M'}(G)$, and the identity becomes
\[
s^G_M(\ell) = \sum_{k\in\K_{G\text{-reg}}(M)}\Delta_{K\cap M}(\ell',k)r^G_M(k)  - \sum_{G'\neq G}\iota_{M'}(G,G')s^{{G}'}_{{M}'}(\ell'),
\]
which serves as an inductive definition of $s^G_M(\ell)$. When $M=G$, the expression \eqref{flsum} equals $s_{G'}(\ell') = s^{G'}_{G'}(\ell')$, and if $G'\neq G$ the required identity reduces to the standard Fundamental Lemma, whereas if $G'=G$ there is nothing to prove. On the other hand, if $M$ is the minimal Levi subgroup, then $M'=M$ is the only endoscopic datum and the statement is trivial.

\begin{rem}
We note that implicit in the statement of the weighted fundamental lemma is that the sum of \emph{noninvariant} orbital integrals is not only invariant but stably invariant. This expectation will continue to hold in our formulation of the weighted transfer below. 
\end{rem}

\subsection{Transfer}

Let us return to $F$ being a global field. Let us write 
\[
\I(G_V,\zeta_V),\quad \SS(G_V,\zeta_V)
\] for the space of functions generated by ordinary orbital integrals $f_G$ and stable orbital integrals $f^G$ respectively of functions $f\in\H(G_V,\zeta_V)$. We also define the larger spaces 
\[
\I_\ac(G_V,\zeta_V),\quad \SS_\ac(G_V,\zeta_V)
\] corresponding to $f\in\H_\ac(G_V,\zeta_V)$.  Similarly, if we define $\C(G_V,\zeta_V)$ to be the space of $\zeta_V^{-1}$-equivariant Schwartz functions on $G_V$, we write 
\[
\I\C(G_V,\zeta_V),\quad S\C(G_V,\zeta_V)\]  for the spaces generated by $f_G$ and $f^G$ respectively for all $f\in\C(G_V,\zeta_V)$. The topologies on these spaces are chosen such that the maps $f\to f_G$ and $f\to f^G$ respectively are open and continuous. We also note that the topology in $\C(G_V,\zeta_V)$ is taken to be the usual one induced by the family of seminorms used to the define the Schwartz space. 

The Langlands-Shelstad transfer then asserts that for any $G_V'\in\E(G_V)$, the map that sends $f\in \H(G_V,\zeta_V)$ to the function 
\be
\label{transfer}
f'(\delta') = f^{G'}(\delta') = \sum_{\gamma\in\Gamma_\text{reg}(G_V,\zeta_V)}\Delta(\delta',\gamma)f_G(\gamma)
\ee
where $\delta'\in \Delta_{G\text{-reg}}(\tilde{G}'_V,\tilde{\zeta}'_V)$, is a continuous map from $\H(G_V,\zeta_V)$ to $\SS(\tilde{G}'_V,\tilde\zeta'_V)$ \cite[\S5]{STF1}. As a consequence of Lemma \ref{FL} this extends to a map from $\C(G_V,\zeta_V)$ to $S\C(\tilde{G}'_V,\tilde\zeta'_V)$. 
(Note that Shelstad had already proved the archimedean transfer for Schwartz functions \cite{sheltrans}.) Moreover, by continuity with respect to the topology above, we see that the transfer also extends to a continuous map from $\H_\ac(G_V,\zeta_V)$ to $\SS_\ac(\tilde{G}'_V,\tilde\zeta'_V)$. That is, as described in  in \cite[p.224]{STF1}, the Langlands-Shelatad transfer $f\mapsto f'$ sends $\mathcal H(G_v,\zeta_v)$ continusously $SI\mathcal H(\tilde{G}'_v,\tilde\zeta'_v)$. Given the topologies on the associated spaces, the extension of the transfer to a map from $\mathcal C(G_v,\zeta_v) $ to $S\I(\tilde{G}'_v,\tilde\zeta'_v)$ then follows from the density properties above.

We are interested in the transfer of weighted orbital integrals. We recall the invariant linear forms on $\H_\ac(G_V^Z,\zeta_V)$ defined inductively by
\[
I_M(\gamma,f) = J_M(\gamma,f) - \sum_{L\in\L^0(M)}\hat{I}^L_M(\gamma,\phi_L(f))
\]
where $\phi_L$ is the map on $\H_\ac(G_V^Z,\zeta_V)$ defined in \cite[(2.3)]{STF1} and \cite[\S2]{ITF1}. We recall that they are determined by their restriction to the subspace $\H(G_V^Z,\zeta_V)$ of $\H_\ac(G_V^Z,\zeta_V)$, and extend to continuosly to $\C(G_V^Z,\zeta_V)$. Suppose that we have fixed so-called endoscopic and stable bases $\Delta^\E(G^Z_V,\zeta_V)$ and $\Delta(G^Z_V,\zeta_V)$ of the spaces $\D(G^Z_V,\zeta_V)$ and $S\D(G^Z_V,\zeta_V)$ respectively following the requirements of \cite[\S5]{STF1} (see also Section \ref{A2}); for example, they are chosen to be subsets of the corresponding bases $\Delta^\E(G_V,\zeta_V)$ and $\Delta(G_V,\zeta_V)$ of the spaces $\D(G_V,\zeta_V)$ and $S\D(G_V,\zeta_V)$ respectively. Assume inductively that for each $G'\in\E_{M'}(G)$ we have defined stable linear forms $\hat{S}^{\tilde{G}'}_{\tilde{M}'}(\delta',f')$ on $\SS_\ac((\tilde{G}'_V)^{\tilde{Z}'},\tilde\zeta'_V)$ with $\delta'\in\Delta(\tilde{M}'_V,\tilde\zeta'_V)$. Let 
\[
\E^0_{M'}(G) = \begin{cases} \E_{M'}(G)-\{G^*\}, &\text{ if $G$ is quasisplit}\\ \E_{M'}(G),& \text{otherwise},\end{cases}
\]
and let $\varepsilon(G)=1$ if $G$ is quasisplit and equal to 0 otherwise. We then define linear forms $I^\E_M(\delta',f)$ and $S^G_M(M',\delta',f)$ inductively by the formula
\[
I^\E_M(\delta',f) = \sum_{G'\in\E^0_{M'}(G)}\iota_{M'}(G,G')\hat{S}^{\tilde{G}'}_{\tilde{M}'}(\delta',f') + \varepsilon(G)S^G_M(M',\delta',f) 
\]
together with the supplementary requirement that 
\[
I^\E_M(\delta',f)= I_M(\delta,f)
\]
in the case that $G$ is quasisplit and $\delta'$ maps to the element $\delta\in\Delta^\E(M^Z_V,\zeta_V)$, where 
\[
I_M(\delta,f) = \sum_{\gamma\in\Gamma(M_V^Z,\zeta_V)}\Delta_M(\delta,\gamma) I_M(\gamma,f)
\]
and $\Delta_M(\delta,\gamma)$ is the generalized transfer factor for $M$ that is compatible with the generalized transfer factor for $G$. Namely, if $\mu\to \mu^G$ and $\gamma\to\gamma_M$ are the induction and restriction maps respectively between the spaces $\D(G_V,\zeta_V)$ and $\D(M_V,\zeta_V)$ such that  $f_G(\mu^G)\to f_M(\mu)$ and 
\[
\sum_{\gamma\in\Gamma(G_V,\zeta_V)}a_M(\gamma_M)b_G(\gamma) = \sum_{\mu\in\Gamma(M_V,\zeta_V)}a_M(\mu)b_G(\mu^G)
\]
respectively for any linear functions $a_M$ on $\D(M_V,\zeta_V)$ and $b_G$ on $\D(G_V,\zeta_V)$, then we have that 
\[
\Delta_G(\nu^G,\gamma)=\Delta_M(\nu,\gamma_M),\quad \nu\in\Delta^\E(M_V,\zeta_V),\gamma\in\Gamma(G_V,\zeta_V),
\]
and
\[
\Delta_G(\delta,\mu^G)=\Delta_M(\delta_M,\mu),\quad \delta\in\Delta^\E(G_V,\zeta_V),\mu\in\Gamma(M_V,\zeta_V)\]
respectively. In the case that $G$ is quasisplit and $M'=M^*$, we have that $\delta'=\delta^*$ belongs to $\Delta((M^*_V)^{Z^*},\zeta^*_V)$ and the image $\delta$ of $\delta'$ in $\Delta^\E(M^Z_V,\zeta_V)$ lies in the subset $\Delta(M^Z_V,\zeta_V)$. It follows from Corollary \ref{cora}(b), a mild extension of the main local theorem \cite{STF1}, that the form 
\[
S^G_M(\delta,f) = S^G_M(M',\delta',f)
\]
is stable, and vanishes unless $M'=M^*$. We thus have a linear form $\hat{S}^{G^*}_{M^*}(\delta^*,f^*) = S^G_M(\delta,f)$ on $\SS_\ac((G^*_V)^{Z^*},\zeta^*_V)$ that is the analogue for $(G^*,M^*)$ of the terms $\hat{S}^{\tilde{G}'}_{\tilde{M}'}(\delta',f')$. 

Let us now state our weighted form of the Langlands-Shelstad transfer conjecture, formulated over a local field $F$. The preceding objects that we have just defined have natural analogues in this case. 

\begin{con}\label{conj}
For each $G$ and $M$, the stable linear form
\[
\hat{S}^{{\tilde G}'}_{{\tilde M}'}(\delta',f'), \quad \delta\in\Delta_{G\textnormal{-reg}}(M,\zeta), 
\]
on $\SS_\ac (\tilde{G}',\tilde\zeta')$ has the property that for any unramified local elliptic endoscopic datum $M'$ for $M$, and any element $\delta'\in\Delta_{G\textnormal{-reg}}(M',\zeta')$, the transfer
\[
J_{M}(\delta',f) = \sum_{\gamma\in\Gamma_{G\textnormal{-reg}}(M,\zeta)}\Delta_{M}(\delta',\gamma)J_{M}(\gamma,f)
\]
equals
\[
\sum_{G'\in \E_{M'}(G)}\iota_{M'}(G,G')\hat{S}^{{\tilde G}'}_{{\tilde M}'}(\delta',f').
\]
\end{con}

\noindent As with the case of the ordinary Langlands-Shelstad transfer, one may also conjecturally specify the transfer $f'$ in the case that $f$ belongs to the spherical Hecke algebra by means of the $L$-embedding $\xi'$. We remind here that $f'$ is unique only up to its weighted orbital integral, just as in the case of ordinary transfer where $f' = f^{G'}$ is identified with its stable orbital integral. We have formulated a more general statement than strictly necessary, however, as for our purposes we  only need the following special case, that is, only for nonarchimedean local fields $F_v$ for which $v\not\in V\supset V_\ram(G,\zeta)$ and where $f= b$. As in the case of Theorem \ref{WFL}, we may again assume that $(Z,\zeta)$ is trivial. 

\begin{cor}\label{conk}
Assume Conjecture \ref{conj}. For each $G,M$ and $b$, there is a function
\[
s^G_M(\ell,b),\qquad \ell\in\L_{G\text{-reg}}(M)
\]
with the property that for any $K$, any unramified elliptic endoscopic datum $M'$ for $M$, and any element $\ell'\in\L_{G\text{-reg}}(M')$, the transfer
\[
r^G_M(\ell',b) = \sum_{k\in\K_{G\text{-reg}}(M)}\Delta_{K\cap M}(\ell',k)r^G_M(k,b)
\]
equals
\[
\sum_{G'\in \E_{M'}(G)}\iota_{M'}(G,G')s^{{G}'}_{{M}'}(\ell',b').
\]
\end{cor}

\noindent In particular, the functions $s^{G'}_{M'}(\ell',b')$ here would be given by the linear form $\hat{S}^{{G}'}_{{M}'}(\delta',b')$ where $b'$ is the transfer of $b$. Evidence for the conjecture is of course the weighted Fundamental Lemma itself, which is the case $f = u$, and the unweighted transfer which is the case of $M=G$. Finally, we note that one can also formulate an analogue of Conjecture \ref{conj} for Lie algebras, as Waldspurger has pointed out that it should be possible to formulate an analogous Lie algebra statement for the transfer of weighted orbital integrals \cite[VIII.7]{traces}, which we can take as further support for the conjecture.

\begin{rem}
We remind the reader that this conjecture is only needed to streamline the stabilisation of the unramified terms at a finite number of places, and if we do not assume this conjecture, the main results hold unconditionaly, only for a slightly smaller class of test functions. 
\end{rem}

\section{The unramified geometric terms}

\label{sectionunrg}

\subsection{The unramified terms}

We can now turn to the unramified geometric terms that arise in the stable trace formula. We shall fix a suitably large finite set of places $S\supset V$. We shall assume that $V$ contains $V_\infty$, so the places in $S-V$ are nonarchimedean. We may thus assume that any distribution on $\Gamma(G_S^V,\zeta_S^V)$ is defined by a signed measure on the preimage in $G^V_S$ of a conjugacy class in $\bar{G}^V_S$. We also assume that $\Gamma_\ss(\bar{G}^V_S,\zeta^V_S)$ is equal to $\Gamma_\ss(\bar{G}^V_S)$, which holds if $V$ contains $V_\ram(G,\zeta)$, so that there is an injection $k\to \gamma_V^S(k)$ from $\K(\bar{G}^V_S)$ from $\K(\bar{G}^V_S)$ into $\Gamma({G}^V_S,\zeta^V_S)$.  Following \cite[(4.11)]{witf}, we define the unramified weighted orbital integrals
\[
r^G_M(k,b) = J_M(\gamma^V_S(k),b^V_S), \qquad k\in \K(\bar{M}^V_S).
\]
Furthermore, we define the subset $\L(\bar{G}^V_S)$ of $\Delta(G^V_S,\zeta^V_S)$ consisting of formal linear combinations of classes in $\K(\bar{G}^V_S)$ corresponding to distributions in $\Delta(G^V_S,\zeta^V_S)$ under the linear extension of the map $\gamma^V_S$. It can be identified with a subset of the corresponding family $\L((\bar{G}^*)^V_S)$ for $G^*$ by means of a canonical embedding $\ell\to \ell^*$. Similarly, we define the subset $\L^\E(\bar{G}^V_S)$ in $\Delta^\E(G^V_S,\zeta^V_S)$ as the quotient of $G$-relevant pairs in
\[
\{(G',\ell'): G'\in\E(G^V_S), \ell' \in \L((\bar{G}')^V_S)\}
\]
with an injection $\ell\to\delta^V_S(\ell)$ into $\Delta^\E(G^V_S,\zeta^V_S)$, sending the subset $\L(\bar{G}^V_S)$ into $\Delta(G^V_S,\zeta^V_S)$.
Suppose now that $V$ contains $V_\ram(G,\zeta)$. The unramified function $r^G_M(k,b)$ depends on a choice of hyperspecial maximal compact subgroup $K^V_S$ of $G^V_S$, which we shall assume to be in good position relative to $M^V_S$. The intersection $K^V_S\cap M^V_S$ is also a hyperspecial maximal compact subgroup of $M^V_S$. Following \cite[\S8]{STF1}, we define the normalized transfer factor
\[
\Delta_{K^V_S,M^V_S}(\ell,k) = \Delta_{K^V_S\cap M^V_S}(\delta^V_S(\ell),\gamma^V_S(k))=\prod_{v\in V-S}\Delta_{K_v\cap M_v}(\delta_v,\gamma_v)
\]
for $k\in\K(\bar{M}^V_S)$ and $\ell\in \L^\E(\bar{M}^V_S)$, and we use this to form the function
\be
\label{rgmell}
r^G_M(\ell,b) = \sum_{k\in\K(\bar{M}^V_S)}\Delta_{K^V_S,M^V_S}(\ell,k)r^G_M(k,b).
\ee
on $\L^\E(\bar{M}^V_S)$, which is independent of the choice of $K^V_S$. We note that the sets $\L(\bar{G}^V_S)$ and $\L^\E(\bar{G}^V_S)$ are independent of $\zeta^V_S$, and hence so is $\Delta_{K^V_S,M^V_S}(\ell,k)$. In the following we shall consider triples $(G,M,\zeta)$ where $G$ is a reductive $K$-group over $F$, $M$ a Levi subgroup of $G$, and $\zeta$ a character of $Z(\A)/Z(F)$ for a central induced torus $Z$ of $G$.

\begin{prop}
\label{sgmell}
Assume Conjecture \ref{conj}. For each triple $(G,M,\zeta)$ with $G$ quasisplit, there is a function
\[
s^G_M(\ell,b) = s^{G^*}_{M^*}(\ell^*,b^*)
\]
which vanishes unless $V$ contains $V_\mathrm{ram}(G)$, and such that for any elliptic endoscopic datum $M'$ of $M$ and any $\ell'\in \L((\bar{M}')^V_S)$ with image $\ell\in \L^\E(\bar{M}^V_S)$, we have
\[
r^G_M(\ell, b) = \sum_{G'\in \E_{M'}(G)}\iota_{M'}(G,G')s^{\tilde{G}'}_{\tilde{M}'}(\ell',b').
\]
\end{prop}

\begin{proof}
If $V$ does not contain $V_\ram(G)$, we set $s^G_M(\ell,b)=0$. Otherwise, we  define $s^G_M(\ell,b)$ inductively by setting
\[
s^G_M(\ell,b) = r^G_M(\ell,b) - \sum_{G'\in \E_{M^*}^0(G)}\iota_{M^*}(G,G')s^{\tilde{G}'}_{\tilde{M}^*}(\ell^*,b^*).
\]
The sum is finite since the coefficient $\iota_{M^*}(G,G')$ vanishes unless $G'$ is elliptic. We then have to show that $r^G_M(\ell,b)$ equals the endoscopic expression
\[
r^{G,\E}_M(\ell',b') =\sum_{G'\in \E_{M'}(G)}\iota_{M'}(G,G')s^{\tilde{G}'}_{\tilde{M}'}(\ell',b').
\]
We shall assume that $S$ is large enough to contain $V_\ram(M')$. If $V$ does not contain $V_\ram(M')$, then $M'$ ramifies at some places $v$ in $S-V$. In that case, the functions $s^{\tilde{G}'}_{\tilde{M}'}(\ell',b)$ vanish for all $G'\in \E_{M'}(G)$ by definition, so we have to show that $r^G_M(\ell,b)$ also vanishes.


By the usual splitting formulas for weighted orbital integrals (e.g., \cite[\S3]{stablegerms}), we can reduce to the case where $S-V=\{v\}$. Moreover, the germ expansions and descent formula allow us to reduce to the case where $\ell$ is elliptic and the groups $G,M,M'$ are replaced by the local objects $G_v,M_v,M'_v$. If $M'$ is ramified at $v$, by \cite[Proposition 8.1]{STF1} we see that $r^G_M(\ell,b)$ vanishes as an application of \cite[Proposition 7.5]{Kell}. We note that the latter result is stated for functions in the spherical Hecke algebra of $G$, but the argument holds identically for Schwartz functions. If $M'$ is unramified at $v$, we have to show that $r^G_M(\ell,b) = r^{G,\E}_M(\ell',b)$. In this case, the required identity will follow from an application of weighted transfer. Namely, for any $G,M,$ and $b$ there is a function
\[
s^G_M(\ell,b), \qquad \ell\in \L_{G\text{-reg}}(M_v)
\]
such that any $K$, any unramified elliptic endoscopic datum $M'_v$ of $M_v$, and any $\ell'\in \L_{G\text{-reg}}(M_v')$ the transfer
\[
\sum_{k\in\K_{G\text{-reg}}(M_v)}\Delta_{K_v\cap M_v} (\ell',k)r^G_M(k,b)
\]
is equal to
\[
\sum_{G'\in\E_{M'}(G)}\iota_{M'}(G,G')s^{G'}_{M'}(\ell',b).
\]
The function $s^G_M(\ell,b)$ is uniquely defined by this identity. The existence here then follows from Corollary \ref{conk}.
\end{proof}


Let $(b^{V,G}_S)'$ be the transfer of $b^V_S$ in $S\I_\ac(G^V_S,\zeta^V_S)$, and write 
\[
\dot{f}^b_S = f\times b^V_S
\]
for any $f\in\C(G_V,\zeta_V)$. We then have the following analogue of \cite[Corollary 8.2]{STF1}.

\begin{cor}
\label{commtran}
For any pair $(G,\zeta)$ such that $V$ contains $V_\infty$, any endoscopic datum $G'\in\E_\mathrm{ell}(G)$, and any function $f\in \C(G_V,\zeta_V)$, we have 
\[
(\dot{f}^b_S)' = 
\begin{cases}
f'\times (b^{V}_S)', & \text{if } V\supset V_\ram(G,\zeta),\\
0 ,& \text{otherwise}.
\end{cases}
\]
In particular, the function $(\dot{f}^b_S)'$ vanishes unless $G'$ belongs to $\E_\mathrm{ell}(G,V)$.
\end{cor}

\begin{proof}
This follows from the case $M=G$ in the preceding proposition. We note that this case relies only on the ordinary Langlands-Shelstad transfer. 
\end{proof}

\subsection{The geometric coefficients}

We shall first construct the global geometric coefficients, which will be the terms on the geometric side that depend on the basic function. Following \cite[\S2]{STF1}, we write $\Gamma_\el(G,S,\zeta)$ for the set of $\gamma$ in $\Gamma_\orb(G_S^Z,\zeta_S)$ such that there is a $\dot\gamma\in G(F)$ such that 
\ben
\item[(i)] the semisimple part of $\dot\gamma$ is $F$-elliptic in $G$, 
\item [(ii)] the conjugacy class of $\dot\gamma$ in $G_V$ maps to $\gamma$, 
\item [(iii)] and  $\dot\gamma$ is bounded at each $v\not\in S$.
\een
Then let $\K^V_\el(\bar{M},S)$ denote the elements in $\K(\bar{G}^V_S)$ such that $\gamma\times k$ belongs to $\Gamma_\el(G,S,\zeta)$ for some $\gamma$.  Let $A_M$ be the maximal split torus of a Levi subgroup $M$ of $G$. We then identify the Weyl group of $(G,A_M)$ with the quotient of the normaliser of $M$ by $M$, thus $W^G(M)= \text{Norm}_G(M)/M$, and set $W^G_0=W^G(M_0)$.  We then define for any $\gamma\in \Gamma(G^Z_V,\zeta_V)$, the geometric coefficient
\be
\label{arsgam}
a^G_{r,s}(\gamma) =\sum_{M\in\L}|W^M_0||W^G_0|^{-1}\sum_{k\in \K^V_\el(\bar{M},S)}a^M_{r,s,\el}(\gamma_M\times k)r^G_M(k,b)
\ee
where the elliptic coefficient $a^M_{r,s,\el}(\gamma_M\times k)$ is the one constructed in \cite[(4.10)]{witf}. The coefficient $a^G_{r,s}(\gamma)$ is supported on the set $\Gamma_\el(G,S,\zeta)$, where $S$ is any finite set of valuations of containing $V$ such that $\gamma\times K^V$ is $S$-admissible in the sense of \cite[\S1]{STF1}. It is supported on the discrete subset $\Gamma(G,V,\zeta)$ of $\Gamma(G^Z_V,\zeta_V)$ given by the union of induced distributions $\mu^G$ where $\mu$ runs over elements in $\Gamma_\el(M,V,\zeta)$ and $M$ runs over Levis in $\L$. We shall write $a^G_{r,s,\el}(\gamma,S)$ for the term $M=G$ in the expansion of $a^G_{r,s}(\gamma)$. That is,
\be
\label{agrsell}
a^G_{r,s,\el}(\gamma,S) = \sum_{k\in\K_\el^V(\bar{G},S)}a^G_{r,s,\el}(\gamma\times k)r_G(k,b).
\ee
We note that whereas $a^G_{r,s,\el}(\gamma,S)$ depends on $S$, by \cite[Corollary 4.7]{witf} the coefficient $a^G_{r,s}(\gamma)$ does not.  

We next construct parallel families of endoscopic and stable geometric coefficients on the domains $\Gamma(G^Z_V,\zeta_V)$ and $\Delta^\E(G^Z_V,\zeta_V)$. For any $\gamma\in \Gamma(G^Z_V,\zeta_V),$ we set 
\be
\label{agrs}
a^{G,\E}_{r,s}(\gamma) = \sum_{G'}\sum_{\delta'}\iota(G,G')b^{\tilde{G}'}_{r,s}(\delta')\Delta_G(\delta',\gamma)+\varepsilon(G)\sum_{\delta}b^G_{r,s}(\delta)\Delta_G(\delta,\gamma)
\ee
with $G',\delta',$ and $\delta$ summed over $\E^0_\el(G,V),$ $\Delta((\tilde{G}'_V)^{\tilde{Z}'},\tilde{\zeta}'_V)$ and $\Delta^\E(G^Z_V,\zeta_V)$ respectively, and the coefficients $b^{\tilde{G}'}_{r,s}(\delta)$ are defined inductively by the requirement that
\[
a^{G,\E}_{r,s}(\gamma) = a^G_{r,s}(\gamma)
\]
in the case that $G$ is quasisplit. Moreover, we set
\[
b^{G^*}_{r,s}(\delta^*) = b^G_{r,s}(\delta), \quad \delta\in\Delta(G^Z_V,\zeta_V)
\]
where $b^G_{r,s}(\delta)$ is obtained as a function on $\Delta^\E(G^Z_V,\zeta_V)$ by the local inversion formula \cite[(5.5)]{STF1}. The coefficients $a^{G,\E}_{r,s}(\gamma)$ and $b^G_{r,s}(\delta)$ are in fact supported on the discrete subsets $\Gamma^\E(G,V,\zeta)$ and $\Delta^\E(G,V,\zeta)$ respectively, which are constructed in a manner parallel to $\Gamma(G,V,\zeta)$ \cite[\S7]{STF1}. For example, we inductively define the set $\Delta^\E_\el(G,V,\zeta)$ to be the collection of $\delta\in\Delta^\E(G^Z_V,\zeta_V)$  such that either $\Delta(\gamma,\delta)\neq 0$ for some $\gamma\in\Gamma_\el(G,V,\zeta)$, or $\delta$ is the image in $\Delta^\E(G^Z_V,\zeta_V)$ of an element $\delta'$ in the subset $\Delta_\el(\tilde{G}',V,\tilde\zeta')$ of $\Delta((\tilde{G}'_V)^{\tilde{Z}'},\tilde\zeta_V')$ for some $G'\in\E^0_\el(G,V)$. We set 
\[
\Delta_\el(G,V,\zeta) = \Delta^\E_\el(G,V,\zeta)\cap \Delta(G^Z_V,\zeta_V),
\]
and define $\Delta(G,V,\zeta)$ again to be the union of induced classes $\mu^G$ where $\mu\in\Delta_\el(M,V,\zeta)$ for some $M\in\L$. Then the sums over $\delta'$ and $\delta$ in \eqref{agrs} can be taken over the smaller sets $\Delta(\tilde{G}',V,\tilde\zeta')$ and $\Delta^\E(G,V,\zeta)$ respectively. 

We can now state the main global theorem concerning the geometric coefficients. It is the analogue of the main Global Theorem 1$'$ of \cite[\S7]{STF1}, and will be proved by a series of reductions. We state it here in order to use the necessary induction hypotheses for the reduction.
\begin{thm}
\label{globge}
\textnormal{(a)} If $G$ is arbitrary, we have 
\[
a^{G,\E}_{r,s}(\gamma) = a^G_{r,s}(\gamma),\quad \gamma\in\Gamma^\E(G,V,\zeta).
\]
\textnormal{(b)} If $G$ is quasisplit, we have that
\[
b^G_{r,s}(\delta), \quad \delta\in\Delta^\E(G,V,\zeta),
\]
is supported on the subset $\Delta(G,V,\zeta)$ of $\Delta^\E(G,V,\zeta)$.
\end{thm}

We shall also define the endoscopic and stable analogues of the elliptic coefficients $a^G_{r,s,\el}(\dot\gamma_S)$. For any admissible elements $\dot\gamma_S\in\Gamma_\el^\E(G,S,\zeta)$ and $\dot\delta_S\in\Delta^\E_\el(G,S,\zeta)$ with $S\supset V_\ram(G,\zeta)$, we set
\be
\label{ageell}
a^{G,\E}_{r,s,\el}(\dot\gamma_S) = \sum_{G'}\sum_{\dot\delta_S'}\iota(G,G')b^{\tilde{G}'}_{r,s,\el}(\dot\delta_S')\Delta_G(\dot\delta_S',\dot\gamma_S)+\varepsilon(G)\sum_{\dot\delta_S}b^G_{r,s,\el}(\dot\delta_S)\Delta_G(\dot\delta_S,\dot\gamma_S),
\ee
with $G',\dot\delta_S',$ and $\dot\delta_S$ summed over $\E^0_\el(G,S),\Delta_\el(\tilde{G}',S,\tilde{\zeta}')$ and $\Delta^\E_\el(G,S,\zeta)$ respectively, and the coefficients $b^{\tilde{G}'}_{r,s}(\delta)$ are defined inductively by the requirement that
\[
a^{G,\E}_{r,s,\el}(\dot\gamma_S) = a^G_{r,s,\el}(\dot\gamma_S)
\]
and 
\[
b^{G^*}_{r,s,\el}(\dot\delta_S^*) = b^G_{r,s,\el}(\dot\delta_S)
\]
in the case that $G$ is quasisplit.

Finally, let us also define analogues of \eqref{agrsell}. We write $\delta\times\ell = \delta\times\delta^V_S(\ell)$ for the element in $\Delta^\E(G^Z_S, \zeta_S)$ associated to a pair $\delta\in \Delta^\E (G^Z_V, \zeta_V)$ and $\ell \in \L^\E(\bar{G}^V_S)$. We also write $\L^{V,\E}_\el(\bar{G},S)$ for the set of $\ell\in \L^\E (\bar{G}^V_S)$ such that $\delta\times\ell$ belongs to $\Delta^\E_\el(G, S,\zeta)$ for some $\delta\in \Delta^\E (G^Z_S, \zeta_S)$, and $\L^V_\el(\bar{G},S)$ for the intersection of $\L^{V,\E}_\el(\bar{G},S)$ with $\L(\bar{G}^V_S)$. We also write $\K^{V,\E}_\el(\bar{G}, S)$ for the set of $k$ in $\K(\bar{G}^V_S)$ such that $\gamma\times k$ belongs to $\Gamma^\E_\el(G, S,\zeta)$ for some $\gamma$. We then define the endoscopic and stable analogues of \eqref{agrsell},
\be
\label{agersell}
a^{G,\E}_{r,s,\el}(\gamma,S) = \sum_{k\in\K_\el^{V,\E}(\bar{G},S)}a^{G,\E}_{r,s,\el}(\gamma\times k)r_G(k,b)
\ee
for $G$ arbitrary and $\gamma\in\Gamma^\E_\el(G,S,\zeta)$, and
\be
\label{bgrsell}
b^G_{r,s,\el}(\delta,S) = \sum_{\ell\in\L_\el^V(\bar{G},S)}b^G_{r,s,\el}(\delta\times \ell)r_G(k,b)
\ee
for $G$ quasisplit and $\delta\in\Delta^\E_\el(G,V,\zeta)$. These definitions will allow us to define endoscopic and stable variants of the geometric expansion of the linear form $I^r_s(f)$.

\subsection{The elliptic and orbital parts}

Recall that the geometric expansion of $I^r_s(f)$ is given in \cite[Theorem 4.6]{witf} by
\be
\label{Irsg}
I^r_s(f) = \sum_{M\in\L}|W^M_0||W^G_0|^{-1}\sum_{\gamma\in\Gamma(M,V,\zeta)} a^M_{r,s}(\gamma)I_M(\gamma,f)
\ee
that is valid for any $f\in \C(G,V,\zeta)$.  We shall examine this more closely. Let us first define
 \be
 \label{Idotell}
 I_{r,s,\el}(\dot{f}_S) = \sum_{\dot\gamma\in\Gamma_\el(G,S,\zeta)}a^G_{r,s,\el}(\dot\gamma_S)\dot{f}_{S,G}(\dot\gamma_S), 
 \ee
 for $\dot{f}_S$ belonging to the subspace $\C_\adm(G,S,\zeta)$ of functions in $\C(G,S,\zeta)$ whose support is an admissible subset of $G_S$, and
 \be
 \label{iorb}
 I_{r,s,\orb}(f) = \sum_{\gamma\in\Gamma(G,V,\zeta)}a^G_{r,s}(\gamma)f_G(\gamma)
 \ee
 for $f \in \C(G,V,\zeta)$, corresponding to the term $M=G$ in \eqref{Irsg} and is a linear combination of invariant orbital integrals. If we restrict to the elliptic coefficients, we obtain the linear form
\be
\label{ielS}
I_{r,s,\el}(f,S)= \sum_{\gamma\in\Gamma_\el(G,S,\zeta)} a^G_{r,s,\el}(\gamma,S)f_G(\gamma)
\ee
which can be regarded as the elliptic part of $I^r_s(f)$. It follows from the definitions that  $I_{r,s,\el}(f,S) = I_{r,s,\el}(f^b_{S})$ for any $S$ large enough such that $f^b_{S}$ belongs to $\C_\adm(G,V,\zeta)$.

We define endoscopic and stable analogues of these by setting inductively 
 \be
 \label{ieell}
 I^{\E}_{r,s,\el}(\dot{f}_S)=\sum_{G' \in \E^0_{\el}(G,S)}\iota(G,G')\hat{S}^{\tilde{G}'}_{r,s,\el}(\dot{f}_S') + \varepsilon(G)S^G_{r,s,\el}(\dot{f}_S)
 \ee
 and
 \be
 \label{ieorb}
I^\E_{r,s,\orb}(f)=\sum_{G' \in \E^0_\el(G,V)}\iota(G,G')\hat{S}^{\tilde{G}'}_{r,s,\orb}(f') + \varepsilon(G)S^G_{r,s,\orb}(f)
 \ee
where we recall that the coefficient $\iota(G,G')$ is the one given in \cite[Theorem 8.3.1]{Kcusp}. Here $\E^0_\el(G,S)$ is the complement of $\{G\}$ in $\E_\el(G,S)$, and $\dot{f}_S\in\C(G,S,\zeta)$ and $f\in\C(G,V,\zeta)$ respectively. The term $\hat{S}^{\tilde{G}'}_{r,s,\el}$ is a linear form on the image $S\C_\adm(\tilde{G}',S,\tilde\zeta)$ of $\C_\adm(\tilde{G}',S,\tilde\zeta)$ in $S\C(\tilde{G}',S,\tilde\zeta)$, and $\hat{S}^{\tilde{G}'}_{r,s,\orb}$ is a linear form on $S\C(\tilde{G}',V,\tilde\zeta')$. We furthermore require that 
\[
I^{\E}_{r,s,\el}(\dot{f}_S) =  I_{r,s,\el}(\dot{f}_S)
\]
and
\[
I^\E_{r,s,\orb}(f)=I_{r,s,\orb}(f)
\]
in the case that $G$ is quasisplit, and the general induction hypothesis that $S^{\tilde G'}_\el$ and $S^{\tilde G'}_\orb$ are stable for $G'$ in $\E^0_\el(G,S)$ and $\E^0_\el(G,V)$ respectively. 

If $G$ is arbitrary, it follows by the same argument of \cite[Lemma 7.2]{STF1} that
\be
\label{ieelllem}
I^\E_{r,s,\el}(\dot{f}_S) = \sum_{\dot\gamma_S\in\Gamma^\E_\el(G,S,\zeta)}a^{G,\E}_{r,s,\el}(\dot\gamma_S)\dot{f}_{S,G}(\dot\gamma_S)
\ee
and
\be
\label{ieorblem}
I^\E_{r,s,\orb}(f) = \sum_{\gamma\in\Gamma^\E_\el(G,V,\zeta)}a^{G,\E}_{r,s}(\gamma)f_{G}(\gamma),
\ee
and if $G$ is quasisplit we have that
\be
\label{selllem}
S^G_{r,s,\el}(\dot{f}_S) = \sum_{\dot\delta_S\in\Delta^\E_\el(G,S,\zeta)}b^{G}_{r,s,\el}(\dot\delta_S)\dot{f}_{S,G}^{\E}(\dot\delta_S)
\ee
and
\be
\label{seorblem}
S^G_{r,s,\orb}(f) = \sum_{\delta\in\Delta^\E_\el(G,V,\zeta)}b^{G}_{r,s}(\delta)f_{G}^\E(\delta).
\ee
Here 
\[
f\to f^\E_G(\delta) = f'(\delta')
\]
is a linear form on $\C(G_V,\zeta_V)$ for any $\delta'\in\Delta_{G\text{-reg}}(\tilde{G}'_V,\tilde\zeta'_V)$ with image $\delta\in\Delta_\reg^\E(G_V,\zeta_V)$. Let $I\C^\E(G_V,\zeta_V)$ and $I\C(G_V,\zeta_V)$ be the spaces spanned by $f_G$ and $f^\E_G$ respectively for any $f\in \C(G_V,\zeta_V)$. The map $f_G\to f_G^\E$ provides an isomorphism between these two spaces. 
 
 We also set 
\be
\label{iells}
I^{\E}_{r,s,\el}(f,S) = \sum_{G'\in\E^0_\el(G)}\iota(G,G')\hat{S}^{\tilde{G}'}_{r,s,\el}(f',S) +\varepsilon(G)S^{G}_{r,s,\el}(f,S)
\ee
for linear forms $\hat{S}^{\tilde{G}'}_{r,s,\el}(f' ,S)$ on $S\C(\tilde{G}',S,\tilde\zeta')$ which are defined inductively by requiring that 
\[
I^{\E}_{r,s,\el}(f,S) = I_{r,s,\el}(f,S)
\]
in the case that $G$ is quasisplit. Now suppose that $S$ is large enough so that $\dot{f}_{S}$ belongs to $\C_\adm(G,S,\zeta)$. It then follows inductively from Corollary \ref{commtran} and  \eqref{ieell} that 
\[
I^{\E}_{r,s,\el}(f,S) = I^{\E}_{r,s,\el}(\dot{f}_{S})
\]
and
\[
S^{G}_{r,s,\el}(f,S) = S^{G}_{r,s,\el}(\dot{f}_{S}),
\]
moreover from the expansions \eqref{ieelllem} and \eqref{selllem} we conclude that 
\be
\label{iellS}
I^{\E}_{r,s,\el}(f,S) = \sum_{\gamma\in\Gamma^\E_\el(G,S,\zeta)} a^{G,\E}_{r,s,\el}(\gamma,S)f_G(\gamma)
\ee
and
\be
\label{sellS}
S^{G}_{r,s,\el}(f,S) = \sum_{\delta\in\Delta^\E_\el(G,S,\zeta)} b^G_{r,s,\el}(\delta,S)f_G^\E(\delta).
\ee
These formulae represent a stabilization of the term with $M = G$ in \eqref{arsgam}.

Having made these preliminary manipulations, we can now establish the geometric expansion of $I^\E_{r,s}(f)$ and $S^G_{r,s}(f)$. It follows the broad argument of \cite[\S10]{STF1}. We set
\be
\label{ISrs}
I^\E_{r,s}(f) = \sum_{G'\in\E^0_\el(G,V)}\iota(G,G')\hat{S}'_{r,s}(f') + \varepsilon(G)S^G_{r,s}(f),\quad f\in\C(G,V,\zeta)
\ee
for linear forms $\hat{S}'_{r,s} = \hat{S}^{\tilde{G}'}_{r,s}$ on $S\C(\tilde{G}',V,\tilde\zeta')$, defined inductively by the supplementary requirement that $I^\E(f)=I(f)$ in the case that $G$ is quasisplit. We shall also assume inductively that if $G$ is replaced by a quasisplit inner $K$-form $\tilde G'$, the corresponding analogue of $S^G$ is defined and stable. 

\begin{prop}\label{geomprop} Let $f\in \C(G,V,\zeta)$.\\
\textnormal{(a)} If $G$ is arbitrary,
\[
I^\E_{r,s}(f) -I^\E_{r,s,\orb}(f)  = \sum_{M\in\L^0}|W^M_0||W^G_0|^{-1}\sum_{\gamma\in\Gamma^\E(M,V,\zeta)}a^{M,\E}_{r,s}(\gamma)I^\E(\gamma,f).
\]
\textnormal{(b)} If $G$ is quasisplit,
\begin{align*}
&S^G_{r,s}(f) -S^G_{r,s,\orb}(f) \notag\\
& = \sum_{M\in\L^0}|W^M_0||W^G_0|^{-1}\sum_{M'\in\E_\el(M,V)}\iota(M,M')\sum_{\delta'\in\Delta(\tilde M',V,\tilde\zeta')}b^{\tilde M'}_{r,s}(\delta')S^G_M(M',\delta',f).
\end{align*}
\end{prop}

\begin{proof}
The proof follows the same argument as \cite[Theorem 10.1]{STF1}. 
\end{proof}

The next proposition concerns the coefficients $a^{G,\E}_{r,s}(\gamma)$ and $b^{G}_{r,s}(\delta)$, generalising \cite[Proposition 10.3]{STF1}, so we provide a detailed proof here. 

\begin{thm}
\label{propgeom}
\textnormal{(a)} If $G$ is arbitrary and $\gamma\in\Gamma^\E(G,V,\zeta)$, then
\be
\label{aell}
a^{G,\E}_{r,s}(\gamma)- a^{G,\E}_{r,s,\el}(\gamma,S) = \sum_{M\in\L^0}|W^M_0||W^G_0|^{-1}\sum_{k\in\K^{V,\E}_{r,s,\el}(\bar{M},S)}a^{M,\E}_\el(\gamma_M\times k)r^G_M(k,b).
\ee
\textnormal{(a)} If $G$ is quasisplit and $\delta\in\Delta^\E(G,V,\zeta)$, then
\be
\label{bell}
b^{G}_{r,s}(\delta)-b^{G}_{r,s,\el}(\delta,S) = \sum_{M\in\L^0}|W^M_0||W^G_0|^{-1}\sum_{k\in\L^{V}_\el(\bar{M},S)}b^{M}_{r,s,\el}(\delta_M\times \ell)s^G_M(\ell,b)
\ee
if $\delta$ lies in the subset $\Delta(G,V,\zeta)$ of $\Delta^\E(G,V,\zeta)$ and is zero otherwise.
\end{thm}

\begin{proof}
Let us first consider the differences of \eqref{ieorblem} and \eqref{iellS}, 
\[
I^\E_{r,s,\orb}(f) -I^\E_{r,s,\el}(f,S)  = \sum_{\gamma\in\Gamma^\E(G,V,\zeta)}(a^{G,\E}_{r,s}(\gamma)-a^{G,\E}_{r,s,\el}(\gamma,S))f_G(\gamma),
\]
and \eqref{seorblem} and \eqref{sellS},
\[
S^G_{r,s,\orb}(f) - S^G_{r,s,\el}(f,S) = \sum_{\delta\in\Delta^\E(G,V,\zeta)}(b^G_{r,s}(\delta)-b^G_{r,s,\el}(\delta,S))f^\E_G(\delta).
\]
Substituting the right-hand side of \eqref{aell}, we obtain the linear form
\be
\label{IE0}
 \sum_{M\in\L^0}|W^M_0||W^G_0|^{-1}\sum_{\gamma\in\Gamma^\E_\el(M,V,\zeta)} \sum_{k\in\K^{V,\E}_\el(\bar{M},S)}a^{M,\E}_{r,s,\el}(\gamma\times k)f_M(\gamma)r^G_M(k,b),
\ee
which we will denote by $I^{\E,0}_{r,s,\orb}(f,S)$. We may vary the function $f\in \C(G,V,\zeta)$ independent of $S$ as long as the support of $f$ remains $S$-admissible, hence the required identity \eqref{aell} is equivalent to showing that
\be
\label{I0}
I^\E_{r,s,\orb}(f) - I^\E_{r,s,\el}(f,S) = I^{\E,0}_{r,s,\orb}(f,S).
\ee
On the other hand, if $G$ is quasisplit, substituting the right-hand side of \eqref{bell}, we obtain the linear form
\be
\label{SG0}
 \sum_{M\in\L^0}|W^M_0||W^G_0|^{-1}\sum_{\delta\in\Delta_\el(M,V,\zeta)} \sum_{\ell\in\L^{V}_\el(\bar{M},S)}b^{M}_{r,s,\el}(\delta\times \ell)f^M(\delta)s^G_M(\ell,b),
\ee
which we will denote by $S^{G,0}_{r,s,\orb}(f,S)$, if $\delta\in\Delta(G,V,\zeta)$, and is zero otherwise. Then the required identity \eqref{bell} is equivalent to showing that
\be
\label{SG00}
S^G_{r,s,\orb}(f) - S^G_{r,s,\el}(f,S) = S^{G,0}_{r,s,\orb}(f,S).
\ee

On the other hand, from the definitions \eqref{ieorb} and \eqref{iells}, the difference
\be
\label{geomdiff}
(I^\E_{r,s,\orb}(f) - I^\E_{r,s,\el}(f,S)) -  \varepsilon(G)(S^G_{r,s,\orb}(f) - S^G_{r,s,\el}(f,S))
\ee
is equal to
\[
\sum_{G'\in\E^0_\el(G,V)}\iota(G,G')(\hat{S}^{\tilde{G}'}_{r,s,\orb}(f') - \hat{S}^{\tilde{G}'}_{r,s,\el}(f',S)).
\]
We can assume inductively that $\hat{S}^{\tilde{G}'}_{r,s,\orb}(f') - \hat{S}^{\tilde{G}'}_{r,s,\el}(f',S)$ equals $\hat{S}^{\tilde{G}',0}_{r,s,\orb}(f,S)$ for any $G'\in\E^0_\el(G,V)$. Replacing it with \eqref{SG0} above, we may use \cite[Lemma 10.2]{STF1} to rearrange the sum as
\[
\sum_{R\in(\L^{*})^0}|W^R_0||W^{G^*}_0|^{-1}\sum_{R'\in\E_\el(R,V)}\iota(R,R') \sum_{\dot\sigma_S'} b^{\tilde R'}_{r,s,\el}(\dot\sigma_S') B_{R'}(\dot\sigma_S',f)
\]
where $\dot\sigma_S'$ is summed over the product of $\sigma'\in\Delta_\el(\tilde{R}',V,\tilde\zeta')$ and $\ell'\in\L^V_\el(\tilde{R}',S)$, $\L^*=\L^{G^*}$ runs over Levi subgroups of $G^*$ containing a fixed minimal Levi subgroup $M_0^*$, and
\[
B_{R'}(\dot\sigma_S',f) = \sum_{G'\in\E^0_{R'}(G)}\iota(G,G')f^{R'}(\sigma'){s}^{\tilde{G}'}_{\tilde{R}'}(\ell',b'),
\]
which we shall express as the difference
\[
\sum_{G'\in\E_{R'}(G)}\iota(G,G'){s}^{\tilde{G}'}_{\tilde{R}'}(\ell',b')f^{R'}(\sigma') - \varepsilon(G)s^G_R(R',\dot\sigma'_S,b,f),
\]
where $s^G_R(R',\dot\sigma'_S,b,f)$ equals $s^{G^*}_{R}(\dot\sigma_S,f)$ if $R'=R$ and $\dot\sigma'_S=\dot\sigma$, and is zero otherwise. The contribution of the second term to \eqref{geomdiff} is $-\varepsilon(G)$ times \eqref{SG0}, which is $S^{G,0}_{r,s,\orb}(f,S)$. The contribution of the first term is zero if $R\in(\L^*)^0$ does not come from $G$ by definition of $f^{R'}(\sigma')$, whereas if $(R,R',\sigma')$ lies in the $W^{G^*}_0$-orbit of a triplet $(M,M',\delta')$ that comes from $G$, then it follows from Proposition \ref{sgmell} that the first term equals
\[
r^G_M(\ell,b)f^{M'}(\delta') = r^G_M(\ell,b)f^\E_{M}(\delta),
\]
where $\delta\times\ell$ is the image of $\delta'\times\ell'$ in the product of $\Delta^\E_\el(M,V,\zeta)$ and $\L^{V,\E}(\bar{M},S)$. We can thus write the contribution of the first term to \eqref{geomdiff} as
\be
\label{geomdiff2}
\sum_{M\in\L^0}|W^M_0|||W^G_0|^{-1}\sum_{M'\in\E_\el(M,V)}\iota(M,M')\sum_{\delta'}\sum_{\ell'}b^{\tilde{M}'}_{r,s,\el}(\delta'\times\ell')r^G_M(\ell,b)f_M^\E(\delta)
\ee
where $\delta'$ and $\ell'$ are summed over $\Delta_\el(\tilde{M}',V,\tilde\zeta')$ and $\L^{V}_\el(\bar{M}',S)$. The sum over $\E_\el(M,V)$ can actually be taken over $\E_\el(M,S)$, since by Proposition \ref{sgmell} the contribution of $M'$ in the complement of $\E_\el(M,V)$ of $\E_\el(M,S)$ is zero. Moreover, using the definitions $r^G_M(\ell,b)$ in \eqref{rgmell} and of $f_M^\E(\delta)$, we may express their product as
\[
\sum_{\gamma\in\Gamma^\E(M,V,\zeta)}\sum_{k\in\K^{V,\E}_\el(\bar{M},S)}\Delta_M(\delta'\times\ell',\gamma\times k)r^G_M(k,b)f_M(\gamma).
\]
If $G$ is quasisplit, the general induction hypothesis implies that Theorem \ref{globge}(b) holds for any $M\in\L^0$ in \eqref{geomdiff2}. We may then write \eqref{ageell} as 
\[
a^{M,\E}_{r,s,\el}(\dot\gamma_S)= \sum_{M'\in\E_\el(M,S)}\iota(M,M')\sum_{\dot\delta_S}b^{\tilde{M}'}_{r,s}(\delta')\Delta_M(\dot\delta'_S,\dot\gamma_S),
\]
and hence the inner sum on $M'$ in \eqref{geomdiff2} is equal to
\[
\sum_{\dot\gamma_S}a^{M,\E}_{r,s,\el}(\dot\gamma_S)r^G_M(k,b)f_M(\gamma).
\]
Thenrefore we conclude that \eqref{geomdiff2} is equal to $I^{\E,0}_{r,s,\el}(f,S)$,\footnote{The expression (10.13) indicated at the beginning of \cite[p.272]{STF1} should be equal to $I^{\E,0}_\orb(f)$ and not $I^{\E,0}_\el(f)$.} and \eqref{geomdiff} is then equal to the difference
\[
I^{\E,0}_{r,s,\orb}(f,S) - \varepsilon(G)S^{G,0}_{r,s,\orb}(f,S).
\]
Now if $\varepsilon(G)=0$, the required identity \eqref{I0} follows from the preceding assertion. On the other hand, if $\varepsilon(G)=1$, then
\[
I^\E_{r,s,\orb}(f) - I^\E_{r,s,\el}(f,S) = I_{r,s,\orb}(f) - I_{r,s,\el}(f,S)
\]
by deifnition, and $I^{\E,0}_{r,s,\orb}(f,S)$ equals 
\[
 \sum_{M\in\L^0}|W^M_0||W^G_0|^{-1}\sum_{\dot\gamma_S}a^{M}_{r,s,\el}(\dot\gamma_S)r^G_M(k,b)f_M(\gamma). 
\]
This again fulfills \eqref{I0}, and the remaining terms thus imply that the required identity \eqref{SG00} is satisfied as well.
\end{proof}

The preceding proposition provides the first reduction of study of the global coefficients to the basic elliptic coefficients. We record it as the following statement for later use.

\begin{cor}\label{glob1'}Suppose that
\ben
\item[(a)]if $G$ is arbitrary, we have 
\[
a^{G,\E}_{r,s,\el}(\dot\gamma_S) = a^G_{r,s,\el}(\dot\gamma_S), 
\]
for any admissible element $\dot\gamma_S\in\Gamma^\E_\el(G,S,\zeta)$, and

\item[(b)]if $G$ is quasisplit, we have that
\[
b^G_\el(\dot\delta_S), \quad \dot\delta_S\in\Delta^\E_\el(G,S,\zeta),
\]
vanishes for any admissible element $\dot\delta_S$ in the complement of $\Delta_\el(G,S,\zeta)$ of $\Delta_\el^\E(G,S,\zeta)$.
\een
Then Theorem \ref{globge} holds. 
\end{cor}

\section{The unramified spectral terms}

\label{sectionunrs}

\subsection{Endoscopic $L$-functions}

For the moment, let $G$ be a connected reductive group over a number field $F$, and fix an induced central torus $Z$ in $G$ with an automorphic character $\zeta$. Also let $V$ be a finite set of places containing $V_\ram(G,\zeta)$. We consider families $c = \{c_v:v\not\in V\}$ of semisimple conjugacy classes $c_v$ in the local $L$-group $^LG_v$ of $G_v$, whose image in the local Weil group $W_{F_v}$ is a Frobenius element. Then let $\mathscr C(G^V,\zeta^V)$ be the set of families $c$ such that the image of each $c_v$ under the projection $^LG_v\to {^LZ_v}$ gives the unramified Langlands parameter of $\zeta_v$, and for any $\hat{G}$-invariant polynomial $A$ on $^LG$ we have 
\[
|A(c_v)|\le q^{r_A}_v,\qquad v\not\in V
\]
for some $r_A>0$, where again $q_v$ is order of the residue field of $F_v$. Given $c\in \mathscr C(G^V,\zeta^V)$ and a finite-dimensional representation $\rho$ of $^LG$, we form the Euler product
\[
L(s,c,\rho) = \prod_{v\not\in V} \det(1-\rho(c_v)q_v^{-s})^{-1}, \quad s\in{\mathbb C},
\]
which converges to an analytic function of $s$ in some right half plane. There is a natural action of $\a^*_{G,Z}\otimes\mathbb C$ on $\mathscr C(G^V,\zeta^V)$ given by
\[
c \mapsto c_\lambda=\{c_{v,\lambda}= c_vq_v^{-\lambda}:v\not\in V\}, \quad \lambda \in \a^*_{G,Z}\otimes \CC,
\]
whereby $L(s,c,\rho)$ is analytic for Re$(s)\gg \text{Re}(\lambda)$. Our main interest will be in the following case: given a Levi subgroup $M$ of $G$ with dual group $\hat{M}$, there is a bijection $P\to \hat{P}$ from $\P(M)$ to the set $\P(\hat{M})$ of $\Gamma$-stable parabolic subgroups of $\hat{G}$ with Levi component $\hat{M}$. For any $P,Q\in\P(M)$, let $\rho_{Q|P}$ be the adjoint representation of $^LM$ on the Lie algebra of the intersection of unipotent radicals of $\hat{\bar{P}}$ and $\hat{Q}$. The $L$-functions $L(s,c,\rho_{Q|P})$ will be used to construct the unramified spectral terms. 

Let $M'$ stand for an elliptic endoscopic datum $(M',\M',s'_M,\xi'_M)$ that is unramified outside of $V$. Recall that $\M'$ is a split extension of $W_F$ by $\hat{M}'$ that need not be $L$-isomorphic to $^LM'$. We therefore fix a central extension $\tilde{M}'$ of $M'$ by an induced torus $\tilde{C}'$ over $F$, together with an $L$-embedding $\tilde\xi':\M' \to {^L\tilde{M}'}$. If $W_F\to \M'$ is any section, the composition
\[
W_F\to \M' \stackrel{\tilde\xi'}{\to} {^L\tilde{M}'}\to {^L\tilde{C}'}
\]
is a global Langlands parameter that is dual to a character $\tilde{\eta}'$ of $\tilde{C}'(F)\bs\tilde{C}'(\A)$. We may assume that $\tilde{M}'$ and $\tilde\eta'$ are also unramified outside of $V$. 

Given $c'\in \mathscr C((\tilde{M}')^V,(\tilde{\zeta}')^V)$, the projection of $c'_v$ onto $^L\tilde{C}'_v$ for any $v\not\in V$ is the conjugacy class that corresponds to the Langlands parameter of the unramified representation $\tilde\eta'_v$. Thus $c' = \tilde\xi'_v(\mathfrak c'_v)$ for some semisimple conjugacy class $\mathfrak c'_v$ in $\M'_v$. Let us write $c_v = \xi'_{M,v}(\mathfrak c'_v)$ for the image of $\mathfrak c'_v$ in $^LM_v$, so that the family $c=\{c_v:v\not\in V\}$ belongs to $\mathscr C(M^V,\zeta^V)$. This gives a map 
\be
\label{CMV}
\mathscr C((\tilde{M}')^V,(\tilde{\zeta}')^V) \to \mathscr C(M^V,\zeta^V).
\ee
Now let $\mathscr C_\text{aut}^V(G,\zeta)$ be the set of $c$ for which there exists an irreducible representation $\pi_V$ of $G_V$ such that $\pi_V\otimes \pi^V(c)$ is an automorphic representation of $G(\A)$, where $\pi^V(c)$ is the product over all $v\not\in V$ of unramified irreducible representations $\pi_v(c) = \pi(c_v)$  of $G_v$ determined by each $c_v$. By Langlands' functoriality principle, we expect that the map \eqref{CMV} descends to a map of subsets 
\[
\mathscr C_\text{aut}^V(\tilde M',\tilde\zeta')\to \mathscr C^V_\text{aut}(M,\zeta),
\]
and while this is not known, it will be enough to use the result of Arthur \cite[Proposition 1]{endo} that if $c$ is the image of a family $c'\in \mathscr C^V_\text{aut}(M^V,\zeta^V)$, the $L$-functions $L(s,c,\rho_{Q|P})$ have meromorphic continuation.

Suppose that $\psi_\alpha:M\to M_\alpha$ is an inner twist over $F$ that is unramified outside of $V$. Let $(Z_\alpha,\zeta_\alpha)$ be the image of $(Z,\zeta)$ and let $\psi_\alpha^*: {^LM}_\alpha\to {^LM}$ be an $L$-isomorphism dual to $\psi_\alpha$. We assume inductively that for any elliptic endoscopic datum $M'$ for $M$ that is proper in the sense that it is not equal to a quasisplit inner form of $M$, the set $\mathscr C^{V,\E}_\textnormal{aut}(\tilde{M}',\tilde\zeta')$ is defined. We then define $\mathscr C^{V,\E}_\textnormal{aut}(M,\zeta)$ to be the union over all such $M'$ of the images in $\mathscr C(M^V,\zeta^V)$ of $\mathscr C^{V,\E}_\textnormal{aut}(\tilde{M}',\tilde\zeta')$, together with the union over all $M_\alpha$ as above of the images in $\mathscr C(M^V,\zeta^V)$ of the sets $\mathscr C_\text{aut}^V(M_\alpha,\zeta_\alpha)$. 

Let $(Z(\hat{M})^\Gamma)^0$ be the identity component of the $\Gamma$-invariant elements in the center of $\hat{M}$. Let $a$ be a nontrivial character of $(Z(\hat{M})^\Gamma)^0$, and let $\rho_a$ be the representation of $^LM$ on the root space $\hat{\mathfrak g}_a$ of $a$ on the Lie algebra of $\hat{G}$. For any $c\in \mathscr C(M^V,\zeta^V)$, the endoscopic $L$-function
\[
L_G(s,c,a) = L(s,c,\rho_a) 
\]
is an analytic function of $s$ in some right half plane, and is equal to 1 unless $a$ is a root of $(\hat{G},(Z(\hat{M})^\Gamma)^0)$. If $\Sigma(\hat{P})$ denotes the set of roots attached to a parabolic subgroup $P\in\P(M)$, we then have
\[
L(s,c,\rho_{Q|P}) = \prod_{a\in \Sigma(\hat{\bar{P}}))\cap \Sigma(\hat{Q})} L_G(s,c,a).
\]
It has meromophic continuation to the complex plane for any $c\in\mathscr C^{V,\E}_\textnormal{aut}(M,\zeta)$. 

Suppose $a$ is a nontrivial character of $(Z(\hat{M})^\Gamma)^0$. Then the kernel $Z_a$ of $a$ acts by translation on $Z(\hat{M})^\Gamma/Z(\hat{G})^\Gamma$, which is in bijection with $\E_{M'}(G)$ by \cite[Corollary 3]{endo}. Let $\E_{M'}(G)/Z_a$ be the set of orbits, and recall that there is an isomorphism
\[
(Z(\hat{M}')^\Gamma)^0/(Z(\hat{M}')^\Gamma)^0\cap Z(\hat{G}')^\Gamma \stackrel{\sim}{\longrightarrow} (Z(\widehat{\tilde{M}'})^\Gamma)^0/(Z(\widehat{\tilde{M}'})^\Gamma)^0\cap Z(\widehat{\tilde{G}'})^\Gamma.
\]
If $a$ is trivial on $(Z(\hat{M}')^\Gamma)^0\cap Z(\hat{G}')^\Gamma$, let $a'$ be the unique character on $(Z(\widehat{\tilde{M}'})^\Gamma)^0$ that is trivial on $(Z(\widehat{\tilde{M}'})^\Gamma)^0\cap Z(\widehat{\tilde{G}'})^\Gamma$. If not, we take $a'$ to be any character on $(Z(\widehat{\tilde{M}'})^\Gamma)^0$ whose restriction to $(Z(\hat{M}')^\Gamma)^0$ is equal to $a$. Hence the character $a$ determines a family of $L$-functions $L_{\tilde{G}}(s,c',a')$ for each $c'\in \mathscr C((\tilde{M}')^V,(\tilde{\zeta}')^V)$ and $G'\in \E_{M'}(G)$. If $c$ is the image in $\mathscr C(M^V,\zeta^V)$ of some $c'$ in $\mathscr C((\tilde{M}')^V,(\tilde{\zeta}')^V)$, we have the decompostion
\be
\label{lgsca}
L_G(s,c,a) = \prod_{G'\in \E_{M'}(G)/Z_a}L_{\tilde{G}'}(s,c',a')
\ee
by \cite[Lemma 4]{endo}.

The complexification $\a^*_{M,Z}\otimes\CC$ can be identified with a subspace of the Lie algebra of $Z(\hat{M})^\Gamma$. If $da$ is the linear form on $\a^*_M\otimes\CC$ associated with the character $a$ on $(Z(\hat{M}')^\Gamma)^0$, we then have
\[
L_G(s,c_\lambda,a) = L_G(s+(da)(\lambda),c,a).
\]
It follows from this that for any fixed $s\in\CC$ and $c\in \mathscr C^{V,\E}_\textnormal{aut}(M,\zeta)$, the $L$-function $L_G(s,c,a)$ is a meromorphic function of $\lambda$.

\subsection{The unramified terms}

We are now ready to define the unramified normalizing factors. The preceding discussion concerned only connected groups. To extend the constructions to a $K$-group $G$, we simply take the union of the corresponding sets attached to each of the components $G_\alpha$ of $G$. Given $c\in \mathscr C^V(M,\zeta)$, the quotients
\[
r(c_\lambda,a) = L_G(0,c_\lambda,a)L_G(1,c_\lambda,a)^{-1}
\]
and
\begin{align*}
r_{Q|P}(c_\lambda) &= \prod_{a\in \Sigma(\hat{\bar{P}}))\cap \Sigma(\hat{Q})} r(c_\lambda,a) \\
&=  L_G(0,c_\lambda,\rho_{Q|P})L_G(1,c_\lambda,\rho_{Q|P})^{-1}
\end{align*}
are meromorphic functions of $\lambda\in \a^*_{M,Z}\otimes \CC$, for any $P,Q\in\P(M)$. Then
\[
r_{Q}(\Lambda,c_\lambda)=r_{Q|\bar{Q}}(c_\lambda)^{-1}r_{Q|\bar{Q}}(c_{\lambda+\frac12\Lambda}), \quad Q\in\P(M)
\]
is a $(G,M)$-family of functions of $\Lambda\in i\a^*_{M,Z}$. The limit
\[
r^G_M(c_\lambda) = \lim_{\Lambda\to 0} \sum_{Q\in\P(M)}r_{Q}(\Lambda,c_\lambda)\theta_Q(\Lambda)^{-1},
\]
with
\[
\theta_Q(\Lambda)= \vol(\a^G_M/\Z(\Delta^\vee_Q))^{-1}\prod_{\alpha\in\Delta_Q}\Lambda(\alpha^\vee),
\]
is defined as a meromorphic function of $\lambda\in\a^*_{M,Z}\otimes\CC$ as in \cite[\S2]{ITF0}.

Now consider the space $\F(G^Z_V,\zeta_V)$ of finite, complex linear combinations of irreducible characters on $G^Z_V$ with $Z_V$-central character equal to $\zeta_V$. We identify an element $\pi\in \F(G^Z_V,\zeta_V)$ with a linear form $f\to f_G(\pi)$ on $\C(G,V,\zeta)$. There is an additional map $f\to f_M$ factoring through this map, by which we define an induction operation from $\F(M_V,\zeta_V)$ to $\F(G_V,\zeta_V)$, given by $\rho\to\rho^G$ satisfying 
\[
f_G(\rho^G) = f_M(\rho).
\]
In particular, this holds in the case that $V$ contains a single element. In \cite[(5.1)]{witf}, using the basic function $b=b^r_s$ we introduced the unramified character
\[
r^G_M(c,b) = r^G_M(c)b_M(c) 
\]
that takes the place of the unramified spectral terms appearing in the stable trace formula \cite[\S3]{STF1}. Here $b_M(c)$ is the character of the induced representation of $\pi_v(c)$ evaluated at the function $b$.  We note that these unramified spectral terms differ from Arthur's $r^G_M(c)$ in that $r^G_M(c)$ are scalars, whereas $r^G_M(c,b)$ can be understood as scalar multiples of induced characters, which we shall have to stabilize.
\begin{lem}
\label{bmc}
For any elliptic endoscopic datum $M'$ of $M$ unramified outside of $V\supset V_\infty$, and any $c'\in \mathscr C^{V,\E}_\textnormal{aut}(\tilde{M}',\tilde{\zeta}')$ with image $c\in \mathscr C^{V,\E}_\textnormal{aut}(M,\zeta)$, we have
\[
b_M(c) = b'_{\tilde{M}'}(c').
\]
\end{lem}

\begin{proof}
Since we are assuming that $M$ is unramified outside of $V$, there is a canonical class of $L$-embeddings from $^LM'_v$ to $^LM_v$, which we shall also denote by $\xi_{M_v}$ for each $v\not\in V$, and we can take $Z_v$ to be trivial. The transfer $b'$ of $b$ on the level of functions is prescribed by the canonical map on spherical Hecke algebras 
\[
\H_\ac(M_v,K_v\cap M_v)\to \H_\ac({M}_v',{K}_v'\cap {M}'_v),
\]
or rather its $\zeta^{-1}$-equivariant analogue, induced by $\xi'_M$ and the generalized Satake isomorphism. By construction, it commutes with the map sending $c'$ to $c$ given by \eqref{CMV}. 

If $r$ is a finite-dimensional representation of $^LM_v$, we write $r'$ for the corresponding representation of $^LM'_v$. In particular, we have
\[
r(c) = r(\xi'_M(c')) = r'(c'),
\]
and the required identity then follows from the equality unramified local $L$-functions $L_v(s,c,r) = L_v(s,c',r')$ and the definition of $b$.
\end{proof}

We now prove the following statement for triples $(G,M,\zeta)$ as above, for a finite set of places $V$ of $F$.

\begin{prop}
\label{sgmdisc}
For each triple $(G,M,\zeta)$ with $G$ quasisplit and each $c\in \mathscr C^{V,\E}_\textnormal{aut}(M,\zeta)$, there is a function
\[
s^G_M(c_\lambda,b) = s^{G^*}_{M^*}(c_\lambda,b), \qquad \lambda\in \a^*_{M,Z,\mathbb C}
\]
that is analytic in some right half plane, such that for any elliptic endoscopic datum $M'$ of $M$ and any $c'\in \mathscr C^{V,\E}_\textnormal{aut}(\tilde{M}',\tilde{\zeta}')$ with image $c\in \mathscr C^{V,\E}_\textnormal{aut}(M,\zeta)$, the identity
\[
r^G_M(c_\lambda, b) = \sum_{G'\in \E_{M'}(G)}\iota_{M'}(G,G')s^{\tilde{G}'}_{\tilde{M}'}(c'_\lambda,b')
\]
holds.
\end{prop}

\begin{proof}
If $V$ does not contain $V_\ram(G)$, we set $s^G_M(c_\lambda,b)=0$. Otherwise, we  define $s^G_M(c_\lambda,b)$ inductively by setting
\[
s^G_M(c_\lambda,b) = r^G_M(c_\lambda,b) - \sum_{G'\in \E_{M}^0(G)}\iota_{M}(G,G')s^{{G}'}_{{M}}(c_\lambda,b).
\]
The sum is finite since the coefficient $\iota_{M^*}(G,G')$ vanishes unless $G'$ is elliptic. We then have to show that $r^G_M(c_\lambda,b)$ equals the endoscopic expression
\[
r^{G,\E}_M(c_\lambda',b') =\sum_{G'\in \E_{M'}(G)}\iota_{M'}(G,G')s^{\tilde{G}'}_{\tilde{M}'}(c_\lambda',b').
\]
We shall prove this in a slightly more general setting, along the lines of \cite[\S4]{endo}.

Suppose that $A$ is a finite set of continuous characters on $(Z(\hat{M})^\Gamma)^0$. Consider the product of $b_M(c_\lambda)$ with 
\[
r_Q(\Lambda,c_\lambda,A) = \prod_{a\in A\cap \Sigma(\hat{Q})} r(c_\lambda,a)^{-1}r(c_{\lambda+\frac12\Lambda},a), \quad Q\in \P(M),
\]
which is a $(G,M)$-family of functions of $\Lambda\in i\a^*_{M,Z}$, with values in the space of functions of $\lambda$. Then the limit
\[
r^G_M(c_\lambda,b,A) = \lim_{\Lambda\to 0}\sum_{Q\in\P(M)}r_Q(\Lambda,c_\lambda,A)b_M(c_\lambda)\theta_Q(\Lambda)^{-1}
\]
is an analytic function of $\lambda$ in some right half plane. If $A' = \{a' : a\in A\}$, we define inductively
\[
s^G_M(c_\lambda,b,A) = r^G_M(c_\lambda,b,A) - \sum_{G'\in \E_{M}^0(G)}\iota_{M}(G,G')s^{{G}'}_{{M}}(c_\lambda,b,A)
\]
for $(G,M,\zeta)$ quasisplit, and 
\be
\label{rgem}
r^{G,\E}_M(c_\lambda',b',A') =\sum_{G'\in \E_{M'}(G)}\iota_{M'}(G,G')s^{\tilde{G}'}_{\tilde{M}'}(c_\lambda',b',A')
\ee
in general. In fact, if $(G^*,M^*,\zeta^*)$ is a quasisplit inner twist of $(G,M,\zeta),$ there is a bijection $c\to c^*$ from $\mathscr C^{V,\E}_\textnormal{aut}(M,\zeta)$ to $\mathscr C^{V,\E}_\textnormal{aut}(M^*,\zeta^*)$ such that 
\[
r^G_M(c_\lambda,b,A)=r^{G^*}_{M^*}(c_\lambda^*,b^*,A),
\]
where $b^*=b^{G^*}$. Since $r^{G,\E}_M(c_\lambda',b',A')$ is equal to $r^{G^*,\E}_{M^*}(c_\lambda',b',A')$, we can therefore assume that our triple $(G,M,\zeta)$ is quasisplit. We shall show that $r^{G,\E}_M(c_\lambda',b',A') $ equals $r^{G}_M(c_\lambda,b,A)$ by induction on $A$.

Assume first that $A=\{a\}$. The function $r^G_M(c_\lambda,b,a)$ vanishes unless $a$ is a root of $(\hat{G},(Z(\hat{M})^\Gamma)^0)$ and spans the kernel $\a^G_M$ of the natural map $\a_M\to\a_G$. The same assertion holds inductively for the functions $s^{\tilde{G}'}_{\tilde{M}'}(c_\lambda',b',a')$, and hence also for $r^{G,\E}_M(c_\lambda',b',a')$. We can therefore assume that $M$ is a maximal Levi subgroup, in which case $r^G_M(c_\lambda,a)$ is a logarithmic derivative $r(c_\lambda,a)$. Then by \eqref{lgsca} and Lemma \ref{bmc} we have
\[
r^G_M(c_\lambda,b,a) = \sum_{G'\in\E_{M'}(G)/Z_a} r^{\tilde{G}'}_{\tilde{M'}}(c'_\lambda,a') b'_{\tilde{M}'}(c'_\lambda).
\]
On the other hand, define the expression
\[
^*r^{G,\E}_M(c_\lambda',b',a') =\sum_{G'\in \E_{M'}(G)}\iota_{M'}(G,G'){^*s^{\tilde{G}'}_{\tilde{M}'}}(c_\lambda',b',a')
\]
where
\[
{^*s^{\tilde{G}'}_{\tilde{M}'}}(c_\lambda',b',a') = |Z_{a'}/Z_{a'}\cap Z(\widehat{\tilde{G}'})^\Gamma|^{-1} r^{\tilde{G}'}_{\tilde{M'}}(c'_\lambda,b',a').
\]
It suffices then to show that $^*r^{G,\E}_M(c_\lambda',b',a')$ equals $r^G_M(c_\lambda,b,a)$, for if $M'=M$ this establishes inductively that $^*s^G_M(c_\lambda,a) = s^G_M(c_\lambda,a)$, whereby for $M'$ arbitrary we conclude that
\[
r^{G,\E}_M(c'_\lambda,b',a')={^*r^{G,\E}_M}(c'_\lambda,b',a')=r^{G}_M(c_\lambda,b,a)
\]
as required. From \cite[p.1146]{endo}, it follows that the identity
\[
\iota_{M'}(G,G'){^*s^{\tilde{G}'}_{\tilde{M}'}}(c_\lambda',b',a') =  |Z_{a}/Z_{a}\cap Z(\hat{G}')^\Gamma|^{-1} r^{\tilde{G}'}_{\tilde{M'}}(c'_\lambda,b',a')
\]
holds, and substituting into \eqref{rgem} we therefore have
\begin{align*}
^*r^{G,\E}_M(c'_\lambda,b',a') 
&= \sum_{G'\in\E_{M'}(G)/Z_a}r^{\tilde{G}'}_{\tilde{M}'}(c'_\lambda,b',a')\\
& = r^G_M(c_\lambda,b,a),
\end{align*}
since $r^{\tilde{G}'}_{\tilde{M}'}(c'_\lambda,b',a')$ depends only on the orbit of $Z_a$ in $\E_{M'}(G)$, and since the stabilizer of $G'$ in $Z_a$ is $Z_a\cap Z(\hat{G})^\Gamma$.

Now suppose that $A$ is the disjoint union of nonempty proper subsets $A_1$ and $A_2$. Assume inductively that 
\[
r^{L_i,\E}_M(c'_\lambda,b',A_i') = r^{L_i}_M(c_\lambda,b,A_i), \quad L_i\in\L(M)
\]
for $i=1,2$. We would like to use splitting formulas to reduce the case of $A$ to those of $A_1$ and $A_2$. While $r_Q(\Lambda,c_\lambda,A)$ is certainly a product a $(G,M)$-families, we also note that $r^G_M(c_\lambda,b,A)$ can be viewed as a weighted character, whose weight factor is simply the limit of $(G,M)$ families defined by $r^G_M(c_\lambda,A)$, rather than the usual intertwining operators that form the usual weighted characters in the trace formula. We therefore have the splitting formula
\[
r^G_M(c_\lambda,b,A) = \sum_{L_1,L_2\in\L(M)}d^G_M(L_1,L_2)r^{L_1}_M(c_\lambda,b,A_1)r^{L_2}_M(c_\lambda,b,A_2)
\]
where the coefficient $d^G_M(L_1,L_2)$ is defined as in \cite[\S2]{ITF1}, analogous to \cite[Proposition 9.4]{ITF1}. We similarly have a splitting formula for $r^{G,\E}_M(c_\lambda',b',A')$, obtained in the same way as in the proof of \cite[Theorem 5]{endo}. It follows then from our induction assumption that $r^{G,\E}_M(c'_\lambda,b',A')$ equals $r^G_M(c_\lambda,b,A)$, and the proposition follows by taking $A$ to be the set of roots of $(\hat{G},(Z(\hat{M})^\Gamma)^0).$
\end{proof}

For the application to the trace formula, we may work with a slightly smaller subset of $\mathscr C(G^V,\zeta^V)$. Define $\mathscr C^{V,\E}_\disc(G,\zeta)=\mathscr C^{V,\E}_\disc(G^*,\zeta^*)$ inductively as the union over all inner $K$-forms $G_1$ of $G^*$, of the sets $\mathscr C^V_\disc(G_1,\zeta_1)$ defined below, together with the union over all elliptic endoscopic data $G'\in\E^0_\el(G^*,V)$ of the images in $\mathscr C(G^V,\zeta^V)$ of $\mathscr C^{V,\E}_\disc(\tilde{G}',\tilde\zeta')$. It is a proper subset of $\mathscr C^{V,\E}_\textnormal{aut}(G,\zeta)$. The following corollary then follows from the same proof of \cite[Corollary 8.4]{STF1}. 

\begin{cor}
\label{rgcor}
Let $c\in\mathscr C^{V,\E}_\disc(M,\zeta)$. Then $r^G_M(c_\lambda,b)$ is an analytic function of $\lambda\in i\a^*_{M,Z}$ in some right-half plane satisfying
\[
\int_{i\a^*_{M,Z}/i\a^*_{G,Z}}r^G_M(c_\lambda,b)(1+||\lambda||)^{-N}d\lambda <\infty
\]
for some $N$, and similarly for $s^G_M(c_\lambda,b)$ if $G$ is quasisplit.
\end{cor}

By \cite[Corollary 8.4]{STF1}, the functions $r^G_M(c_\lambda)$ and $s^G_M(c_\lambda)$ themselves are analytic in $\lambda$, and as we are assuming that the implied parameter $s$ of $b$ has real part large enough, we can in fact evaluate $r^G_M(c_\lambda,b)$ and $s^G_M(c_\lambda,b)$ at the point $\lambda=0$.

\subsection{Spectral transfer factors}

We now turn to the spectral side of the trace formula. We shall first recall the basic definitions of the objects that appear on the spectral side as in \cite[\S3]{STF1}. Recall the space $\F(G_V^Z,\zeta_V)$ of finite complex linear combinations of irreducible characters on $G^Z_V$. It has a canonical basis $\Pi(G_V^Z,\zeta_V)$ given by irreducible characters with $Z_V$-central character equal to $\zeta_V$. As before, we identify elements $\pi\in \F(G_V^Z,\zeta_V)$ with the linear form 
\[
f\mapsto f_G(\pi) =\tr(\pi(f))
\] 
on $\H(G_V^Z,\zeta_V)$. We write $\Pi_\text{unit}(G_V^Z,\zeta_V)$ for the subset of unitary characters in $\Pi(G_V^Z,\zeta_V)$. The orbits of the action $\pi\mapsto \pi_\lambda$ of $i\a^*_{G,Z}$ on $\pi\in \Pi_\text{unit}(G_V,\zeta_V)$ can be identified with the set $\Pi_\text{unit}(G_V^Z,\zeta_V)$. The induction operation is given by the induced characters 
\[
f_M(\pi)= f_G(\pi^G)=\tr(\I_P(\pi,f)), \quad \pi\in \Pi_\text{unit}(M_V,\zeta_V).
\]
If moreover we restrict to the subset of tempered characters $\Pi_\text{temp}(G_V,\zeta_V)$, we may then take $f \in \C(G_V^Z,\zeta_V)$. 

We regard $f_G$ as a linear form on both $\D(G_V,\zeta_V)$ and $\F(G_V,\zeta_V)$, which are in turn determined by its restriction to the subsets $\Gamma_\reg(G_V,\zeta_V)$ and $\Pi_\temp(G_V,\zeta_V)$ respectively. The functions determine each other, so we may consider the subspace of stable distributions $S\F(G_V,\zeta_V)$ in $\F(G_V,\zeta_V)$. Parallel to the geometric side, we consider a basis $\Phi(G_V^Z,\zeta_V)$ of $S\F(G_V,\zeta_V)$, and the subset $\Phi(G_V,\zeta_V) = \Phi^\E(G_V,\zeta_V) \cap S\F(G_V,\zeta_V)$ that forms a basis of $S\F(G_V,\zeta_V)$. It is defined in terms of the abstract bases $\Phi_\disc(M_v,\zeta_v)$ of a certain cuspidal subspace $S\I_\text{cusp}(M_v,\zeta_v)$ in \cite{LCR}, and their analogues for endoscopic groups $M'$ of $M$. Then if $\phi'_v$ is an element of $\Phi(\tilde{G}'_v,\tilde\zeta'_v)$ with image $\phi'_v\in\Phi_\E(G_v,\zeta_v)$, we have spectral transfer factors
\[
\Delta(\phi_v,\pi_v) = \Delta(\phi_v',\pi_v), \quad \pi_v\in\Pi(G_v,\zeta_v),
\]
such that the linear form $f\to f^\E_G(\phi_v)=f'(\phi'_v)$ has an expansion
\[
f^\E(\phi_v) =\sum_{\pi\in\Pi(G_v,\zeta_v)}\Delta(\phi_v,\pi_v)f_G(\pi_v).
\]
Then we form the extended transfer factor
\[
\Delta(\phi,\pi) = \prod_{v\in V}\Delta(\phi_v,\pi_v),\quad \phi\in\Phi^\E(G_V,\zeta_V),\ \pi\in\Pi(G_V,\zeta_V).
\]
Following \cite[\S5]{STF1}, we can assume that we have fixed quotients $\Phi^\E(G_V^Z,\zeta_V)$ of $\Phi^\E(G_V,\zeta_V)$ and  $\Phi(G_V^Z,\zeta_V)$ of $\Phi(G_V,\zeta_V)$ that form bases of $\F(G_V^Z,\zeta_V)$ and $S\F(G_V^Z,\zeta_V)$ respectively. If $\phi$ and $\pi$ have unitary central characters, we can then identify them with the orbits $\phi_\lambda$ and $\pi_\lambda$ of $i\a^*_{G,Z}$ in $\Phi^\E(G_V,\zeta_V)$ and $\Pi(G_V,\zeta_V)$ respectively. We then define transfer factors on $\Phi^\E(G_V^Z,\zeta_V)\times \Pi(G_V^Z,\zeta_V)$ as the average 
\[
\Delta(\phi,\pi) = \sum_{\lambda\in i\a^*_{G,Z}}\Delta(\phi_\lambda,\pi_\lambda),
\]
where the sum only has one nonzero term.

\subsection{The spectral coefficients}

We shall now construct the global spectral coefficients, which will be the terms on the spectral side that depend on the basic function, following \cite[\S7]{STF1}. We can identify any element $c\in\mathscr C(G^V,\zeta^V)$ with a $K^V$-unramified representation $\pi^V(c)\in\Pi(G^V,\zeta^V)$. Given $\pi\in\Pi(G_V,\zeta_V)$, we write $\pi\times c = \pi\times \pi^V(c)$ for the associated representation in $\Pi(G(\A)^Z,\zeta)$. We shall identify $\pi$ with its representative in $\Pi(G^Z_V,\zeta)$, and $\pi\times c$ with the associated representation in $\Pi(G(\A)^1,\zeta)$. We recall the basic spectral coefficient $a(\dot\pi)$ in \cite[\S3]{STF1} that is supported on the discrete subset $\Pi_\disc(G)$ of unitary representations $\Pi_\unit(G(\A)^1)$. Let $\Pi_\disc(G,\zeta)$ be the subset of those representations whose central character on $Z(\A)^1$ is equal to $\zeta$. We then define $\Pi_\disc(G,V,\zeta)$ to be the set of representations $\pi\in\Pi_\unit(G^Z_V,\zeta)$ such that $\pi\times c$ belongs to $\Pi_\disc(G,\zeta)$ for some $c\in \mathscr C(G^V,\zeta^V)$. We also define $\mathscr C_\disc^V(G,\zeta)$ to be the set of $c\in\mathscr C(G^V,\zeta^V)$ such that $\pi\times c$ belongs to $\Pi_\disc(G,\zeta)$ for some $\pi\in\Pi_\disc(G,V,\zeta)$, which we used to define the sets $\mathscr C^{V,\E}_\disc(G_1,\zeta_1)$ earlier.

We then define for any $\pi\in \Pi(G^Z_V,\zeta_V)$, the spectral coefficient
\be
\label{arspi}
a^G_{r,s}(\pi) =  \sum_{L\in\L}|W^L_0||W^G_0|^{-1}\sum_{c\in \mathscr C^V_\disc(M,\zeta)}a^M_\disc(\pi_M\times c)r^G_M(c,b),
\ee
where $\pi_M\times c$ is a finite sum of representations $\dot\pi$ in $\Pi_\text{unit}(M(\A),\zeta)$, and $a^M_\disc(\pi_M\times c)$ is the sum of corresponding values $a^G_\text{disc}(\dot\pi)$. Let $\tilde\Pi_\disc(M,V,\zeta)$ be the preimage of $\Pi_\disc(M,V,\zeta)$ in $\Pi_\text{unit}(M_V,\zeta_V)$, and let $\Pi^G_\disc(M,V,\zeta)$ be the set of $i\a^*_{G,Z}$-orbits in $\tilde\Pi_\disc(M,V,\zeta)$.  There is a free action $\rho\to \rho_\lambda$ of $i\a^*_{M,Z}/i\a^*_{G,Z}$ on $\Pi^G_\disc(M,V,\zeta)$ whose orbits can be identified with $\Pi_\disc(M,V,\zeta)$. Any element $\rho\in \Pi^G_\disc(M,V,\zeta)$ is an irreducible representation of $M_V\cap G^Z_V$, from which one can form the parabolically induced representation $\rho^G$ of $G^Z_V$. Finally, let $\Pi(G,V,\zeta)$ be the union over $M\in \L$ and $\rho\in \Pi_\disc^G(M,V,\zeta)$ of irreducible constituents of $\rho^G$. It follows then from the definitions that $a^G_{r,s}(\pi)$ is supported on the subset $\Pi(G,V,\zeta)$ of $\Pi(G^Z_V,\zeta_V)$. We again write $a^G_{r,s,\disc}(\pi)$ for the term $M=G$ in the expansion of $a^G_{r,s}(\pi)$. That is,
\be
\label{agrsdisc}
a^G_{r,s,\disc}(\pi) = \sum_{c\in\mathscr C_\disc^V({G},\zeta)}a^G_{\disc}(\pi\times c)r_G(c,b).
\ee
and we note that $r_G(c,b)$ is equal to $L(s,c,r)$.

We next construct parallel families of endoscopic and stable spectral coefficients on certain subsets $\Phi^\E_\disc(G,V,\zeta)$, $\Phi_\disc(G,V,\zeta)$, and $\Pi_\disc^\E(G,V,\zeta)$ of $\Phi^\E(G^Z_V,\zeta_V)$, $\Phi(G^Z_V,\zeta_V)$, and $\Pi(G^Z_V,\zeta_V)$ respectively in the same way as $\Pi_\disc(G,V,\zeta)$. Similarly, from these discrete subsets we construct the larger subsets $\Phi^\E(G,V,\zeta)$, $\Phi(G,V,\zeta)$, and $\Pi^\E(G,V,\zeta)$ again by the same induction process. For any $\pi\in \Pi^\E(G,V,\zeta),$ we set 
\be
\label{agrspi}
a^{G,\E}_{r,s}(\pi) = \sum_{G'}\sum_{\phi'}\iota(G,G')b^{\tilde{G}'}_{r,s}(\phi')\Delta_G(\phi',\pi)+\varepsilon(G)\sum_{\phi}b^G_{r,s}(\phi)\Delta_G(\phi,\pi)
\ee
with $G',\phi',$ and $\phi$ summed over $\E^0_\el(G,V),\Phi(\tilde{G}',V,\tilde{\zeta}')$ and $\Phi^\E(G,V,\zeta)$ respectively, and the coefficients $b^{\tilde{G}'}_{r,s}(\phi')$ are defined inductively by the requirement that
\[
a^{G,\E}_{r,s}(\pi) = a^G_{r,s}(\pi)
\]
in the case that $G$ is quasisplit. Moreover, we set
\[
b^{G^*}_{r,s}(\phi^*) = b_{r,s}^G(\phi), \quad \phi\in\Phi^\E(G,V,\zeta)
\]
where $b_{r,s}^G(\phi)$ is obtained as a function on $\Phi^\E(G,V,\zeta)$ by the local inversion formula \cite[(5.8)]{STF1}. 

We then state the main global theorem concerning the spectral coefficients. It is the analogue of the main Global Theorem 2$'$ of \cite[\S7]{STF1}, and will again be proved by a series of reductions. We state it here in order to use the necessary induction hypotheses.
\begin{thm}
\label{globsp}
\textnormal{(a)} If $G$ is arbitrary, we have 
\[
a^{G,\E}_{r,s}(\pi) = a^G_{r,s}(\pi),\quad \pi\in\Pi^\E(G,V,\zeta).
\]
\textnormal{(b)} If $G$ is quasisplit, we have that
\[
b^G_{r,s}(\phi), \quad \phi\in\Phi^\E(G,V,\zeta),
\]
is supported on the subset $\Phi(G,V,\zeta)$ of $\Phi^\E(G,V,\zeta)$.
\end{thm}

We also define the stable and endoscopic analogues of the discrete coefficient $a^G_{r,s,\disc}(\dot\pi)$. That is, for any elements $\dot\pi\in\Pi_\disc^\E(G,\zeta)$ and $\dot\phi\in\Phi^\E_\disc(G,\zeta)$, we set
\[
a^{G,\E}_{r,s,\disc}(\dot\pi) = \sum_{G'}\sum_{\dot\pi'}\iota(G,G')b^{\tilde{G}'}_{r,s,\disc}(\dot\phi')\Delta_G(\dot\phi',\dot\pi)+\varepsilon(G)\sum_{\dot\phi}b^G_{r,s,\disc}(\dot\phi)\Delta_G(\dot\phi,\dot\pi)
\]
with $G',\dot\phi',$ and $\dot\phi$ summed over $\E^0_\el(G),\Phi_\disc(\tilde{G}',\tilde{\zeta}')$ and $\Phi^\E_\disc(G,\zeta)$ respectively, and the coefficients $b^{\tilde{G}'}_{r,s}(\phi)$ are defined inductively by the requirement that
\[
a^{G,\E}_{r,s,\disc}(\dot\pi) = a^G_{r,s,\disc}(\dot\pi)
\]
and 
\[
b^{G^*}_{r,s,\disc}(\dot\phi^*) = b^G_{r,s,\disc}(\dot\phi)
\]
in the case that $G$ is quasisplit. 

Finally, we define the endoscopic and stable analogues of \eqref{agrsdisc}. For any $c\in\mathscr C(G^V,\zeta^V)$, let $\phi^V(c)$ be the corresponding product of unramified Langlands parameters $\phi_v(c_v)$ over all $v\not\in V$, write $\phi\times c= \phi\times\phi^V(c_V)$ for the associated element in $\Phi^\E(G(\A)^Z,\zeta)$ for any $\phi\in\Phi^\E(G^Z_V,\zeta_V)$. We then set
\[
a^{G,\E}_{r,s,\disc}(\pi) = \sum_{c\in\mathscr C_\disc^{V,\E}({G},\zeta)}a^{G,\E}_{r,s,\disc}(\pi\times c)r_G(c,b)
\]
for $G$ arbitrary and $\pi\in\Pi^\E_\disc(G,V,\zeta)$, and
\[
b^G_{r,s,\el}(\delta,S) = \sum_{c\in\mathscr C_\disc^V({G},\zeta)}b^G_{r,s,\el}(\phi\times c)r_G(c,b)
\]
for $G$ quasisplit and $\phi\in\Phi^\E_\disc(G,V,\zeta)$.  These definitions will allow us to define endoscopic and stable variants of the spectral expansion of the linear form $I^r_s(f)$.

\subsection{The discrete and unitary parts}


Recall that the spectral expansion of $I^r_s(f)$ is given in \cite[Theorem 5.2]{witf} by
\be
\label{Irss}
I^r_s(f) = \sum_{M\in\L}|W^M_0||W^G_0|^{-1}\int_{\Pi(M,V,\zeta)} a^M_{r,s}(\pi)I_M(\pi,f)d\pi
\ee
 for any $f\in \C(G,V,\zeta)$, where $d\pi$ is a natural Borel measure on the set $\Pi(M,V,\zeta)$.  Let us define
 \be
 \label{idisc}
 I_{r,s,\disc}(\dot{f}) = \sum_{\dot\pi\in\Pi_\disc(G,\zeta)}a^G_{r,s,\disc}(\dot\pi)\dot{f}_{G}(\dot\pi), 
 \ee
 for $\dot{f}\in \C(G(\A)^Z,\zeta)$, and
 \be
 \label{iunit}
 I_{r,s,\unit}(f) = \int_{\Pi(G,V,\zeta)}a^G_{r,s}(\pi)f_G(\pi)d\pi, 
 \ee
 for $f \in \C(G,V,\zeta)$, corresponding to the term $M=G$ in \eqref{Irss}. We call it the purely unitary part of $I^r_s(f)$, which is a continuous linear combination of irreducible unitary characters. If we restrict to the discrete coefficients, we obtain the linear form
\be
\label{idiscsum}
I_{r,s,\disc}(f)= \sum_{\pi\in\Pi_\disc(G,V,\zeta)} a^G_{r,s,\disc}(\pi)f_G(\pi)
\ee
which can be regarded as the discrete part of $I^r_s(f)$. It follows from the definitions that  $I_{r,s,\disc}(f) = I_{r,s,\disc}(\dot{f})$ for $\dot f = f\times b^V$.

We define endoscopic and stable analogues of these by setting inductively 
 \be
 \label{iedisc}
 I^{\E}_{r,s,\disc}(\dot{f})=\sum_{G' \in \E^0_{\disc}(G,V)}\iota(G,G')\hat{S}^{\tilde{G}'}_{r,s,\disc}(\dot{f}') + \varepsilon(G)S^G_{r,s,\disc}(\dot{f})
 \ee
 and
 \be
 \label{ieunit}
I^\E_{r,s,\unit}(f)=\sum_{G' \in \E^0_\el(G,V)}\iota(G,G')\hat{S}^{\tilde{G}'}_{r,s,\unit}(f') + \varepsilon(G)S^G_{r,s,\unit}(f)
 \ee
where $\dot{f}_S\in\C(G,S,\zeta)$ and $f\in\C(G,V,\zeta)$ respectively, and the terms $\hat{S}^{\tilde{G}'}_{r,s,\disc}$ and $\hat{S}^{\tilde{G}'}_{r,s,\unit}$ are linear forms on $S\C(\tilde{G}',V,\tilde\zeta)$. We furthermore require that 
\[
I^{\E}_{r,s,\disc}(f) =  I_{r,s,\disc}(f)
\]
and
\[
I^\E_{r,s,\unit}(f)=I_{r,s,\unit}(f)
\]
respectively in the case that $G$ is quasisplit, and the general induction hypothesis that $S^{\tilde G'}_\el$ and $S^{\tilde G'}_\unit$ are stable for $G'$ in $\E^0_\el(G,S)$ and $\E^0_\el(G,V)$ respectively.  If $G$ is arbitrary, it follows by the same argument of \cite[Lemma 7.3]{STF1} that
\be
\label{iedisclem}
I^\E_{r,s,\disc}(\dot{f}) = \sum_{\dot\pi\in\Pi^\E_\disc(G,\zeta)}a^{G,\E}_{r,s,\disc}(\dot\pi)\dot{f}_{G}(\dot\pi)
\ee
and
\be
\label{ieunitlem}
I^\E_{r,s,\unit}(f) = \int_{\Pi^\E(G,V,\zeta)}a^{G,\E}_{r,s}(\pi)f_{G}(\pi)d\pi,
\ee
and if $G$ is quasisplit we have that
\be
\label{sdisclem}
S^G_{r,s,\disc}(\dot{f}) = \sum_{\dot\phi\in\Phi^\E_\disc(G,\zeta)}b^{G}_{r,s,\disc}(\dot\phi){f}_{G}^\E(\dot\phi)
\ee
and
\be
\label{sedisclem}
S^G_{r,s,\unit}(f) = \int_{\Phi^\E(G,V,\zeta)}b^{G}_{r,s}(\phi)f_{G}^\E(\phi)d\phi,
\ee 
where again we have natural Borel measures defined on the sets $\Phi^\E_\disc(G,\zeta)$ and $\Phi^\E_\disc(G,V,\zeta)$.

We also set 
\[
I^{\E}_{r,s,\disc}(f) = \sum_{G'\in\E^0_\el(G)}\iota(G,G')\hat{S}^{\tilde{G}'}_{r,s,\disc}(f') +\varepsilon(G)S^{G}_{r,s,\disc}(f)
\]
for linear forms $\hat{S}^{\tilde{G}'}_{r,s,\disc}(f')$ on $S\C(\tilde{G}',V,\tilde\zeta')$ which are defined inductively by requiring that 
\[
I^{\E}_{r,s,\disc}(f) = I_{r,s,\disc}(f)
\]
in the case that $G$ is quasisplit. If we take $\dot{f}  = f\times b^V$, it then follows inductively from Corollary \ref{rgcor} and  \eqref{iedisc} that 
\be
\label{idiscdot}
I^{\E}_{r,s,\disc}(f) = I^{\E}_{r,s,\disc}(\dot{f})
\ee
and
\be
\label{sdiscdot}
S^{G}_{r,s,\disc}(f) = S^{G}_{r,s,\disc}(\dot{f}),
\ee
moreover from the expansions \eqref{iedisclem} and \eqref{sdisclem} we conclude that 
\be
\label{idiscS}
I^{\E}_{r,s,\disc}(f) = \sum_{\pi\in\Pi^\E_\disc(G,V,\zeta)} a^{G,\E}_{r,s,\disc}(\pi)f_G(\pi)
\ee
and
\be
\label{sdiscS}
S^{G}_{r,s,\disc}(f) = \sum_{\Phi^\E_\disc(G,V,\zeta)} b^G_{r,s,\disc}(\phi)f_G^\E(\phi)
\ee
These formulae represent a stabilization of the term with $M = G$ in \eqref{arspi}.

Having made these preliminary manipulations, we can now establish the spectral expansion of $I^\E_{r,s}(f)$ and $S^G_{r,s}(f)$. We recall from \eqref{ISrs} the definition
\[
I^\E_{r,s}(f) = \sum_{G'\in\E^0_\el(G,V)}\iota(G,G')\hat{S}'_{r,s}(f') + \varepsilon(G)S^G_{r,s}(f),\quad f\in\C(G,V,\zeta)
\]
defined inductively by the supplementary requirement that $I^\E(f)=I(f)$ in the case that $G$ is quasisplit. We are assuming inductively that if $G$ is replaced by a quasisplit inner $K$-form $\tilde G'$, the corresponding analogue of $S^G$ is defined and stable. 

\begin{prop}\label{specprop}Let $f\in \C(G,V,\zeta)$.\\
\textnormal{(a)} If $G$ is arbitrary,
\[
\label{unitspeca}
I^\E_{r,s}(f) -I^\E_{r,s,\unit}(f)  = \sum_{M\in\L^0}|W^M_0||W^G_0|^{-1}\int_{\Pi^\E(M,V,\zeta)}a^{M,\E}_{r,s}(\pi)I^\E_M(\pi,f)d\pi.
\]
\textnormal{(b)} If $G$ is quasisplit,
\begin{align*}
\label{unitspecb}
&S^G_{r,s}(f) -S^G_{r,s,\unit}(f) \notag\\
& = \sum_{M\in\L^0}|W^M_0||W^G_0|^{-1}\sum_{M'\in\E_\el(M,V)}\iota(M,M')\int_{\Phi(\tilde M',V,\tilde\zeta')}b^{\tilde M'}_{r,s}(\phi')S^G_M(M',\phi',f)d\phi.
\end{align*}
\end{prop}

\begin{proof}
The proof follows the same argument as \cite[Theorem 10.6]{STF1}. 
\end{proof}

The next proposition concerns the coefficients $a^{G,\E}_{r,s}(\pi)$ and $b^{G}_{r,s}(\phi)$, so we again provide the details of the proof.

\begin{thm}
\label{specthm}
\textnormal{(a)} If $G$ is arbitrary and $\pi\in\Pi^\E(G,V,\zeta)$, then
\be
\label{adisc}
a^{G,\E}_{r,s}(\pi)- a^{G,\E}_{r,s,\disc}(\pi) = \sum_{M\in\L^0}|W^M_0||W^G_0|^{-1}\sum_{c\in\mathscr C^{V,\E}_\mathrm{disc}({M},\zeta)}a^{M,\E}_\el(\pi_M\times c)r^G_M(c,b).
\ee
\textnormal{(a)} If $G$ is quasisplit and $\phi\in\Phi^\E(G,V,\zeta)$, then
\be
\label{bdisc}
b^{G}_{r,s}(\phi)-b^{G}_{r,s,\disc}(\phi) = \sum_{M\in\L^0}|W^M_0||W^G_0|^{-1}\sum_{c\in\mathscr C^{V,\E}_\mathrm{disc}({M},\zeta)}b^{M}_{r,s,\disc}(\phi_M\times c)s^G_M(c,b)
\ee
if $\phi$ lies in the subset $\Phi(G,V,\zeta)$ of $\Phi^\E(G,V,\zeta)$, and is zero otherwise.
\end{thm}

\begin{proof}
The proof  is parallel to that of Theorem \ref{propgeom}, so we can be brief. Let us consider the differences of \eqref{ieunitlem} and \eqref{idiscS}, 
\[
I^\E_{r,s,\unit}(f) -I^\E_{r,s,\disc}(f)  = \int_{\Pi^\E(G,V,\zeta)}(a^{G,\E}_{r,s}(\pi)-a^{G,\E}_{r,s,\disc}(\pi))f_G(\pi)d\pi,
\]
and \eqref{sedisclem} and \eqref{sdiscS},
\[
S^G_{r,s,\unit}(f) - S^G_{r,s,\disc}(f) = \int_{\Phi^\E(G,V,\zeta)}(b^G_{r,s}(\phi)-b^G_{r,s,\disc}(\phi))f^\E_G(\phi)d\phi.
\]
Substituting the right-hand side of \eqref{adisc}, we obtain the linear form
\[
\sum_{M\in\L^0}|W^M_0||W^G_0|^{-1}\int_{\Pi^\E_\disc(M,V,\zeta)} \sum_{c\in\mathscr C^{V,\E}_{disc}({M},\zeta)}a^{M,\E}_{r,s,\disc}(\pi\times c)f_M(\pi)r^G_M(c,b).
\]
We can rewrite this as 
\[
I^\E_{r,s,\unit}(f)=\sum_{M\in\L^0}|W^M_0||W^G_0|^{-1}\int a^{M,\E}_{r,s,\disc}(\dot\pi)f_M(\pi)r^G_M(c,b)d\dot\pi.
\]
where the integral is taken over elements $\dot\pi=\pi\times c$ in the product of $\Pi^\E_\disc(M,V,\zeta)$ and $\mathscr C^{V,\E}_{disc}({M},\zeta)$ relative to natural measures. which we will denote by $I^{\E,0}_{r,s,\unit}(f)$. Hence the required identity \eqref{adisc} is equivalent to showing that
\[
I^\E_{r,s,\unit}(f) - I^\E_{r,s,\disc}(f) = I^{\E,0}_{r,s,\unit}(f).
\]
On the other hand, if $G$ is quasisplit, substituting the right-hand side of \eqref{bdisc}, we obtain the linear form
\[
\sum_{M\in\L^0}|W^M_0||W^G_0|^{-1}\int_{\Phi_\disc(M,V,\zeta)} \sum_{c\in\mathscr C^{V,\E}_{disc}({M},\zeta)}b^{M}_{r,s,\disc}(\phi\times c)f^M(\phi)s^G_M(c,b),
\]
which again we rewrite as 
\[S^{G,0}_{r,s,\unit}(f) = \sum_{M\in\L^0}|W^M_0||W^G_0|^{-1}\int b^{M}_{r,s,\disc}(\dot\phi)f^M(\phi)s^G_M(c,b)d\dot\phi
\]
where the integral is taken over elements $\dot\phi=\phi\times c$ in the product of $\Phi_\disc(M,V,\zeta)$ and $\mathscr C^{V,\E}_{disc}({M},\zeta)$ relative to natural measures. Then the second required identity \eqref{bdisc} is equivalent to showing that
\[
S^G_{r,s,\unit}(f) - S^G_{r,s,\disc}(f) = S^{G,0}_{r,s,\unit}(f).
\]
Now the rest of the proof follows that of Proposition \ref{propgeom} where instead of Proposition \ref{sgmell} we use Proposition \ref{sgmdisc}. In particular, if $\dot\phi'=\phi'\times c'$ belongs to the product of $\Phi_\disc(\tilde M',V,\tilde\zeta')$ with $\mathscr C^{V,\E}_\disc(\tilde M',\tilde\zeta')$ and has image $\phi\times c$ in the product of $\Phi^\E_\disc(M,V,\zeta)$ with $\mathscr C^{V,\E}(M,\zeta)$,\footnote{The corresponding sets in \cite[p.276]{STF1} are mislabeled.} then it follows that 
\[
\sum_{G'\in\E_{M'}(G)}\iota_{M'}(G,G')f^{M'}(\phi')s^{\tilde{G}'}_{\tilde{M}'}(c',b') = r^G_M(c,b)f^\E_M(\phi).
\]
The remainder of the argument then follows the same structure of Proposition \ref{propgeom}.
\end{proof}

The preceding proposition provides a parallel reduction of study of the global spectral coefficients to the basic discrete coefficients.

\begin{cor}
\label{glob2'}
Suppose that
\ben
\item[(a)]if $G$ is arbitrary, we have 
\[
a^{G,\E}_{r,s,\disc}(\dot\pi) = a^G_{r,s,\el}(\dot\pi), 
\]
for any $\dot\pi\in\Pi^\E_\disc(G,\zeta)$ and

\item[(b)]if $G$ is quasisplit, we have that
\[
b^G_\disc(\dot\phi), \quad \dot\phi\in\Phi^\E_\disc(G,\zeta),
\]
is supported on the subset $\Phi_\disc(G,\zeta)$ of $\Phi^\E(G,\zeta)$. 
\een
Then Theorem \ref{globsp} holds. 
\end{cor}

Let us take stock of what we have attained so far. Propositions \ref{geomprop} and \ref{specprop} together imply the reduction of the stable expansions to expressions
\[
S(f) = \sum_{M\in\L}|W^M_0||W^G_0|^{-1}\sum_{\delta\in\Delta(M,V,\zeta)}b^M_{r,s}(\delta)S_M(\delta,f)
\]
and 
\[
S(f) = \sum_{M\in\L}|W^M_0||W^G_0|^{-1}\int_{\Phi(M,V,\zeta)}b^M_{r,s}(\phi)S_M(\phi,f)
\]
for $f\in\C(G,V,\zeta)$. By the main local theorems of \cite{STF1} and their extension to $\C(G,V,\zeta)$ \cite[Theorem 7.2]{maps2}, the linear forms $S_M(\delta,f)$ and $S_M(\phi,f)$ are stable.  What remains is to show that the global coefficients $b^M_{r,s}(\delta)$ and $b^M_{r,s}(\phi)$ are stable. This will take up the rest of the paper.

\section{Global descent}
\label{sectiondesc}

\subsection{Jordan decompositions}

The next reduction that we shall make concerns the global descent of the elliptic coefficients. This will allow us to reduce to the case where the elements $\dot\gamma_S$ and $\dot\delta_S$ appearing in Corollary \ref{glob1'} are unipotent. The arguments of \cite{STF2} will apply here without much modification. 

The semisimple part of $\dot\gamma_S\in\Gamma(G_S,\zeta_S)$ is defined as a semisimple conjugacy class $c_S\in\Gamma_\ss(\bar{G}_S)$, following \cite[\S1]{STF2}. If $c_S$ is contained in the component $\bar{G}_{\alpha,S}=G_{\alpha,S}/Z_S$, we write
\[
\bar{G}_{c_S,+}=\prod_{v\in S}\bar{G}_{c_v,+}
\]
for the centralizer of $c_S$ in the component $\bar{G}_S$, and we write $\bar{G}_{c_S}$ for the connected component of the identity of $\bar{G}_{c_S,+}$. We shall write $G_{c_S}$ for the preimage of $\bar{G}_{c_S}$ in $G_\alpha$, and also let it stand for its $F_S$-points. The unipotent part $\dot\alpha_S$ of $\dot\gamma_S$ is then defined to be an element in the subset $\Gamma_\unip(G_{c_S},\zeta_S)$ of distributions in $\Gamma(G_{c_S},\zeta_S)$ whose semisimple part is trivial. We note that the $\Gamma_\unip(G_{c_S},\zeta_S)$ is a basis for the distributions $\D_1(G_{c_S},\zeta_S)=\D_\unip(G_{c_S},\zeta_S)$ supported on the conjugacy class of the identity. The elements in $\Gamma(G_S,\zeta_S)$ then have a canonical decomposition $\dot\gamma_S = c_S\dot\alpha_S$. The distribution $\dot\alpha_S$ is determined by $\dot\gamma_S$ only up to the action of the finite group
\[
\bar{G}_{c_S,+}(F_S)/\bar{G}_{c_S}(F_S) = \prod_{v\in S}\bar{G}_{c_v,+}(F_v)/\bar{G}_{c_v}(F_v)
\]
in $\Gamma_\unip(G_{c_S},\zeta_S)$.

Also, we say a semisimple element $c$ is $F$-elliptic in $G$ if $A_{G_c}=A_G.$ We then define $i^G(S,c)$ to be equal to 1 if $c$ is an $F$-elliptic element in $G$ and the $G(\A)^S$-conjugacy class of $c$ meets the maximal compact subgroup $K^S$, and equal to zero otherwise. Then the following provides a descent formula for the elliptic coefficient. 

\begin{lem}
\label{adesclem}
Suppose that $S$ contains $V_\ram(G,\zeta)$, and that $\dot\gamma_S\in\Gamma_\el(G,S,\zeta)$ is admissible. Then
\be
\label{adesc}
a^G_{r,s,\el}(\dot\gamma_S) = \sum_c\sum_{\dot\alpha}\iota^{\bar{G}}(S,c)|\bar{G}_{c,+}(F)/\bar{G}_c(F)|^{-1}a^{G_c}_{r,s,\el}(\dot\alpha),
\ee
where $c$ runs over elements in $\Gamma_\ss(\bar{G})$ whose image in $\Gamma_\ss(\bar{G}_S)$ equals $c_S$, and $\dot\alpha$ runs over the orbit of $\bar{G}_{c,+}(F_S)/\bar{G}_c(F_S)$ in $\Gamma_\unip(G_{c,S},\zeta_S)$ determined by $\dot\alpha_S$.
\end{lem}
\begin{proof}
Let $\o^S = \prod_{v\not\in S}\o_v$ and $Z_{S,\o}=Z(F)\cap Z_SZ(\o^S)$. We recall that two elements $\dot\gamma$ and $\dot\gamma_1$ in $G(F_S)$ with standard Jordan decompositions $\dot\gamma = c\dot\alpha$ and $\dot\gamma_1=c_1\dot\alpha_1$ are said to be $(G,S)$-equivalent if there is an element $\dot\delta\in G(F)$ such that $\dot\delta^{-1}c_1\delta = c$ and $\dot\delta^{-1}\alpha_1\delta$ is conjugate to $\alpha$ in $G_c(F_S)$. The elliptic coefficients were then defined in \cite[(4.10)]{witf} as
\[
a^G_{r,s,\text{ell}}(\dot\gamma_S) = \sum_{\{\dot\gamma\}}|Z(F,\dot\gamma)|^{-1}a^G_{r,s}(S,\dot\gamma)(\dot\gamma/\dot\gamma_S)
\]
where the sum over $\{\dot\gamma\}$ runs over a set of representatives of $Z_{S,\o}$-orbits in the set of $(G,S)$-equivalence classes $(G(F))_{G,S}$ of $G(F)$, and $Z(F,\dot\gamma)$ is the subset of $z\in Z_{S,\o}$ such that $z\dot\gamma = \dot\gamma$. Also, $(\dot\gamma/\dot\gamma_S)$ is the ratio of the signed measure on $\dot\gamma_S$ that comes with $\dot\gamma$ and the invariant measure on $\dot\gamma_S$.

Moreover, for a general element $\dot\gamma = c\dot\alpha$, we have the descent formula in \cite[(4.9)]{witf} for the general coefficient
\[
a^G_{r,s}(S,\dot\gamma) = i^G(S,c)|{G}_{c,+}(F)/{G}_c(F)|^{-1}a^{G_c}_{r,s}(S,\dot\alpha).
\]
The function $i^{\bar{G}}(S,c)$ in this case will be nonzero if $c$ is an $F$-elliptic element in $\bar{G}$ whose $G(\A^S)$-conjugacy class meets the maximal compact subgroup $\bar{K}^S= K^SZ(\A^S)/Z(\A^S)$.
Then using the fact that 
\[
|\bar{G}_{c,+}(F)/\bar{G}_c(F)||{G}_{c,+}(F)/{G}_c(F)|^{-1} = |Z(F,\tilde c)|
\]
for any conjugacy class $\tilde c\in\Gamma_\ss({G})$ in the preimage of $c$, the required formula follows by expressing the sum over $\{\dot\gamma\}$ as a double sum over $c$ and $\dot\alpha$ as above, and comparing the expressions for $a^G_{r,s,\text{ell}}(\dot\gamma_S)$ and $a^{G_c}_{r,s,\text{ell}}(\dot\alpha_S)$.
\end{proof}

We shall also require Jordan decompositions of elements in $\dot\delta_S\in\Delta(G_S,\zeta_S)$, where the semisimple part is given by a semisimple stable conjugacy class $d_S\in\Delta_\ss(\bar{G}^*_S)$ in $\bar{G}^*_S=G^*_S/Z_S$, such that the connected centralizer $G^*_{d_S}$ is quasisplit, and the unipotent part is given by an element $\dot\beta_S$ in the subset $\Delta_\unip(G^*_{d_S},\zeta_S)$ in $\Delta(G^*_{d_S},\zeta_S)$ with trivial semisimple part. The distribution $\dot\beta_S$ is again determined by $\dot\delta_S$ only up to the action of the finite group
\[
\bar{G}^*_{d_S,+}(F_S)/\bar{G}^*_{d_S}(F_S) = \prod_{v\in S}\bar{G}^*_{d_v,+}(F_v)/\bar{G}^*_{d_v}(F_v)
\]
in $\Delta_\unip(G_{d_S},\zeta_S)$.

Suppose there is an element $d$ in the set $\Delta_\ss({G}^*)$ of semisimple stable conjugacy classes in ${G}^*(F)$ whose image in $\Delta_\ss({G}^*_S)$ is $d_S$. We then define $i^{{G}}(S,d)$ to be equal to 1 if $d$ is an $F$-elliptic element in $G$ and bounded at each place $v\not\in S$, and zero otherwise. We also define
\be
\label{jG}
j^{G^*}(S,d) = i^{G^*}(S,d)\tau(G^*)\tau(G^*_d)^{-1},
\ee
where $G^*_d$ is a quasisplit connected centralizer of an appropriate representative of the class $d$, and $\tau(G^*)$ is the absolute Tamagawa number of $G^*$.

\subsection{Descent}

We shall first prove a descent formula for the endoscopic and stable geometric coefficients in the following special case where the derived group $G_\der$ of $G$ is simply connected and $Z$ is trivial, from which we will deduce the general result. We note that the first condition implies that $G_{c,+}=G_c$ for any $c\in \Gamma_\ss(G)$, and the second implies that $\bar{G}=G$. The formula \eqref{adesc} then reduces to
\[
a^G_{r,s,\el}(\dot\gamma_S) = \sum_c\iota^{\bar{G}}(S,c)a^{G_c}_{r,s,\el}(\dot\alpha).
\]
Then we have the following descent formula.

\begin{prop}
\label{propdesc}
Assume that $G_\mathrm{der}$ is simply connected and that $Z=1$.
\ben
\item[(a)] Suppose that $\dot\gamma_S$ is an admissible element in $\Gamma^\E_\mathrm{ell}(G,S)$ with Jordan decomposition $\dot\gamma_S=c_S\dot\alpha_S$. Then
\be
\label{propadesc}
a^{G,\E}_{r,s,\el}(\dot\gamma_S) = \sum_{c}i^G(S,c)a^{G_c,\E}_{r,s,\el}(\dot\alpha)
\ee
where $c$ runs over elements in $\Gamma_\mathrm{ss}(G)$ that map to $c_S$, and $\dot\alpha$ is the image of $\dot\alpha_S$ in $\Gamma_\mathrm{unip}(G_{c,S})$.
\item[(b)]  Suppose that $G$ is quasisplit, and that $\dot\delta_S$ is an admissible element in $\Delta^\E_\mathrm{ell}(G,S)$ with Jordan decomposition $\dot\delta_S=d_S\dot\beta_S$. Then
\be
\label{propbdesc}
b^{G}_{r,s,\el}(\dot\delta_S) = \sum_{d}j^{G^*}(S,d)b^{G_d^*}_{r,s,\el}(\dot\beta)
\ee
where $d$ runs over elements in $\Delta_\mathrm{ss}(G^*)$ that map to $d_S$ in $\Delta_\mathrm{ss}(G^*_S)$, and $\dot\beta$ is the image of $\dot\beta_S$ in $\Delta_\mathrm{unip}(G_{d,S}^*)$. Moreover, $b^G_{r,s,\el}$ vanishes on the complement of $\Delta_\mathrm{ell}(G,S)$ in the set of admissible elements in $\Delta_\mathrm{ell}^\E(G,S)$ whose semisimple part is not central in $G_S$.
\een
\end{prop}

\noindent  Before we prove the proposition, we state a corollary that we shall require from our induction assumption. It is the main result of this section.

\begin{cor}
\label{cordesc}
Suppose that Proposition \ref{propdesc} holds for some $z$-extension $\tilde G$ of $G$.
\ben
\item[(a)] 
If $G$ is arbitrary, and $\dot\gamma_S\in\Gamma^\E_\el(G,S,\zeta)$ is an admissible element with Jordan decomposition $\dot\gamma_S=c_S\dot\alpha_S$. Then
\[
a^{G,\E}_{r,s,\el}(\dot\gamma_S) = \sum_c\sum_{\dot\alpha}i^{\bar{G}}(S,c)|\bar{G}_{c,+}(F)/\bar{G}_c(F)|^{-1}a^{G_c,\E}_{r,s,\el}(\dot\alpha),
\]
where the sums are taken as in \eqref{adesc}.
\item[(b)] 
If $G$ is quasisplit, and $\dot\delta_S\in\Delta^\E_\el(G,S,\zeta)$ is an admissible element with Jordan decomposition $\dot\delta_S=d_S\dot\beta_S$. Then
\be
\label{corexp}
b^G_{r,s,\el}(\dot\delta_S)=\sum_d\sum_{\dot\beta}j^{\bar{G}^*}(S,d)|(\bar{G}^*_{d,+}/\bar{G}^*_d)(F)|^{-1}b^{G_d^*}_{r,s,\el}(\dot\beta),
\ee
where $d$ runs over elements in $\Delta_\ss(\bar{G}^*)$ that map to $d_S$, and $\dot\beta$ runs over the orbit of $(\bar{G}^*_{d,+}/\bar{G}^*_{d})(F)$ in $\Delta_\unip(G_{d,S},\zeta_S)$ determined by $\dot\beta_S$.
\een
\end{cor}

\noindent We begin with the proof of Proposition \ref{propdesc}.

\begin{proof}
We shall take on the induction hypothesis that Corollary \ref{glob1'} holds if $G$ is replaced by any group $H$ over $F$ such that either $\dim(H_\der)<\dim(G_\der)$, or $\varepsilon(G)=0$ and $H=G^*$. Our argument will follow the proof of \cite[Theorem 1.1]{STF2}, where we observe that the major combinatorial problems have already been solved. Since $Z$ is trivial, we can take $\tilde{G}'=G'$ for any $G'\in\E(G)$. From the definition \eqref{ageell}, it follows that the difference
\be
\label{60}
a^{G,\E}_{r,s,\el}(\dot\gamma_S)-\varepsilon(G)\sum_{\dot\delta_S\in \Delta^\E_\el(G,S)}b^G_{r,s,\el}(\dot\delta_S)\Delta_G(\dot\delta_S,\dot\gamma_S)
\ee
equals
\be
\label{61}
\sum_{G'\in\E^0_\el(G,S)}\iota(G,G')\sum_{\dot\delta_S'}b^{\tilde{G}'}_{r,s,\el}(\dot\delta_S')\Delta_G(\dot\delta_S',\dot\gamma_S).
\ee
We also assume inductively that any $G'\in\E^0_\el(G,S)$ has a $z$-extension for which Proposition \ref{propdesc} and hence Corollary \ref{cordesc} holds for $G'$. Suppose $\dot\delta_S'=d'_S\dot\beta_S'\in \Delta_\el(G',S)$ such that $\Delta(\dot\delta_S',\dot\gamma)\neq0$, so that $\dot\delta_S'$ is admissible. Then applying \ref{corexp}, we have
\[
b^{G'}_{r,s,\el}(\dot\delta_S')=\sum_{d'}\sum_{\dot\beta'}j^{\bar{G}'}(S,d')|(\bar{G}'_{d',+}/\bar{G}'_{d'})(F)|^{-1}b^{G_{d'}'}_{r,s,\el}(\dot\beta').
\]
Substituting this into \eqref{61}, and using the identities $\Delta_G(\dot\delta'_S,\dot\gamma_S) =\Delta_G(d_S'\dot\beta'_S,c_S\dot\alpha_S)$ and
\[
\iota(G,G')j^{G'}(S,d') = i^{G'}(S,d)\tau(G)\tau(G'_{d'})^{-1}|\text{Out}_G(G')|^{-1}
\]
by \eqref{jG} and \cite[Theorem 8.3.1]{Kcusp}, we can write \eqref{61} as 
\be
\label{62}
\sum_{G'}\sum_{d'}\sum_{\dot\beta'}|\text{Out}_G(G')|^{-1}{\tau(G)\tau(G'_{d'})^{-1}}{|(\bar{G}'_{d',+}/\bar{G}'_{d'})(F)|^{-1}} b^{G'_{d'}}_{r,s,\el}(\dot\beta')\Delta_G(d_S'\dot\beta'_S,c_S\dot\alpha_S)
\ee
where $G'\in \E^0_\el(G,S)$, $d'$ runs over classes in $\Delta_\ss(G')$ that are $F$-elliptic and bounded at each $v\not\in S$, and $\dot\beta'\in\Delta_\unip(G'_{d',S})$.

Let us assume that $d\in\Delta_\ss(G^*)$ is elliptic, bounded at each $v\not\in S$, and a local image of each component $c_v$ of $c_S$. If such a $d$ does not exist, it follows from the argument at the end of \cite[p.212]{STF2} that the required result is trivial. We can also assume that $d$ does not lie in the center of $G^*$. We add to \eqref{62} the sum
\be
\label{63}
\varepsilon(G)\sum_{\dot\beta\in\Delta_\unip(G^*_{d,S})}\tau(G^*)\tau(G^*_d)^{-1}b^{G^*_d}_{r,s,\el}(\dot\beta)\Delta_G(d_S\dot\beta_S,c_S\dot\alpha_S),
\ee
which in the case $\varepsilon(G)=1$, combines to make the outer sum of \eqref{62} into a full sum over $\E_\el(G,S)$.

Let us write $c_\A$ for the product of $c_S$ with $c_v\in K_v$ for each $v\not\in S$ for which $d_v$ is an image.  Also recall that if $G'_d$ belongs to $\E_\el(G^*_d,S)$, it has a canonical extension $\tilde G'_d$ by $\tilde C'_d$ determined by $\tilde G_d$. We shall fix $\tilde\eta'_d$ to be an automorphic character of $\tilde C'_d$. Examining the arguments of \cite[pp.214--218]{STF2}, we can hence conclude that the sum of \eqref{61} and \eqref{63} is equal to
\be
\label{67}
\sum_{c}\sum_{G'_d\in\E_\el(G^*_d,S)}\iota(G_c,G_d')\sum_{\dot\delta'_{d,S}}b^{\tilde{G}'_d}_{r,s,\el}(\dot\delta'_{d,S})\Delta_{G,c}(\dot\delta'_{d,S},\dot\alpha),
\ee
where $c$ runs over classes in $\Gamma_\ss(G)$ that map to the $G(\A)$-conjugacy class of $c_\A$, and $\dot\delta'_{d,S}\in\Delta_\el(\tilde{G}'_d,S,\tilde\eta'_d)$, and $\Delta_{G,c}(\dot\delta'_{d,S},\dot\alpha)$ is the canonical global transfer factor for $G_c$ and $G_d'$ defined in p.214 and Lemma 4.2 of \cite{STF2}. Then using the definition \eqref{ageell}, the induction hypothesis, and the assumption that $d$ is not central, it follows that the sum is equal to
\[
\sum_{c}a^{G_c,\E}_\el(\dot\alpha).
\]
Moreover, any $c$ that occurs in the sum is $F$-elliptic in $G$ and is $G(\A^S)$-conjugate to an element in $K^S$, so we conclude that \eqref{60} is equal to
\be
\label{69}
\sum_{c}i^G(S,c)a^{G_c,\E}_\el(\dot\alpha).
\ee
where now $c$ runs over the set in \eqref{propadesc}. Hence if $\varepsilon(G)=0$, the first required identity \eqref{propadesc} follows. On the other hand, suppose $\varepsilon(G)=1$. Then $a^{G,\E}_\el(\dot\gamma_S)=a^{G}_\el(\dot\gamma_S)$ by definition. Any $c$ occurring in the sum in \eqref{69} is not central, so $\dim(G_{c,\der})<\dim(G_\der)$, so it follows from the induction hypothesis that $a^{G_c,\E}_\el(\dot\alpha) = a^{G_c}_\el(\dot\alpha)$, and the required identity follows from \eqref{adesc}, proving part (a).

For part (b), we continue to assume that $\varepsilon(G)=1$. Then since the sum of \eqref{61} and \eqref{63} equals \eqref{69}, it follows from part (a) that the coefficient of $\varepsilon(G)$ in the latter sum is equal to zero. Moreover, since the contribution to \eqref{63} is nonzero only if $d$ is an $F$-elliptic element in $G$ and bounded at each place $v\not\in S$, we can replace $\tau(G^*)\tau(G^*_d)^{-1}$ on the lefthand side with $j^{G^*}(S,d)$. It follows then that 
\[
\sum_{\dot\delta_S\in \Delta^\E_\el(G,S)}b^G_{r,s,\el}(\dot\delta_S)\Delta_G(\dot\delta_S,\dot\gamma_S)
\]
is equal to
\[
\sum_{\dot\beta\in\Delta_\unip(G^*_{d,S})}j^{G^*}(S,d)b^{G^*_d}_{r,s,\el}(\dot\beta)\Delta_G(d_S\dot\beta_S,\dot\gamma_S).
\]
Then using the adjoint relation in \cite[(5.5)]{STF1} we may multiply by the adjoint transfer factor and sum over $\Gamma^\E_\el(G,S,\zeta)$ to invert these expressions. It follows that $b^G_{r,s,\el}(\dot\delta_S)$ vanishes unless it has a Jordan decomposition with $\dot\beta_S\in\Delta_\unip(G^*_{d_S})$, in which case
\[
b^G_{r,s,\el}(\dot\delta_S) = j^{G^*}(S,d)b^{G^*_d}_{r,s,\el}(\dot\beta),
\]
where $\dot\beta$ is the image of $\dot\beta_S$ in $\Delta_\unip(G^*_{d,S})$, which proves part (b).
\end{proof}

Consider the commutator quotient $G^\text{ab}_S = G_S/G_{\der,S}$. It acts by conjugation on $\D_\unip(G_{\der,S})$. We define a linear map from $\D_\unip(G_{\der,S})$ to $\D_\unip(G_S,\zeta_S)$ by sending any $D\in \D_\unip(G_{\der,S})$ to the distribution 
\[
f \to \sum_{a\in G^\text{ab}_S}(aD)(f|_{G_{\der,S}}), \quad f\in \C(G_S,\zeta_S),
\]
where $f|_{G_\der}$ denotes the restriction of $f$ to $G_{\der,S}$. We shall write $\D_{\unip,\der}(H_S,\zeta_S)$ for the image of this map, and assume that the basis $\Gamma_\unip(G_{S},\zeta_S)$ has been chosen such that the intersection
\[
\Gamma_{\unip,\der}(G_{S},\zeta_S) = \Gamma_\unip(G_{S},\zeta_S)\cap \D_{\unip,\der}(G_{S},\zeta_S)
\]
is a basis of $\Gamma_{\unip,\der}(G_{S},\zeta_S)$. Similarly, we construct stable subsets $S\D_{\unip,\der}(G_{S},\zeta_S)$ with the basis $\Delta_{\unip,\der}(G_{S},\zeta_S)$. The unipotent global coeffcients $a^G_{r,s,\el}(\dot\alpha)$ and $a^{G,\E}_{r,s,\el}(\dot\alpha)$ will be supported on the subset $\Gamma_{\unip,\der}(G_{S},\zeta_S)$, and $b^G_{r,s,\el}(\dot\beta)$ on the subset $\Delta_{\unip,\der}(G_{S},\zeta_S)$ respectively.\footnote{The term $a^H_\disc$ on \cite[p.169]{STF2} should be $a^H_\el$.} We note that there is a canonical isomorphism between $\D_{\unip,\der}(G_{S},\zeta_S)$ and the space $\D_{\unip,\der}(G_{S})$ with trivial central data given by sending any $D\in \D_{\unip,\der}(G_{S},\zeta_S)$ to the linear form
\be
\label{dfc}
f\to D(f_c), \quad f\in \C(G_S,\zeta_S),
\ee
where $f_c$ is any function in $C^\infty_c(G_S)$ that equals $f$ on an invariant neighbourhood of 1. We now proceed to extend the preceding proposition to the general case. 

\begin{prop}
\label{prop21}
Suppose that Proposition \ref{propdesc} holds for some $z$-extension $\tilde G$ of $G$.
\ben
\item[(a)] 
If $G$ is arbitrary, and $\dot\gamma_S\in\Gamma^\E_\el(G,S,\zeta)$ is an admissible element whose semisimple part is not central in $\bar{G}_S$. Then
\[
a^{G,\E}_{r,s,\el}(\dot\gamma_S) = a^G_{r,s,\el}(\dot\gamma_S).
\]
\item[(b)] 
If $G$ is quasisplit, and $\dot\delta_S\in\Delta^\E_\el(G,S,\zeta)$ is an admissible element whose semisimple part is not central in $\bar{G}^*_S$. Then
\[
b^G_{r,s,\el}(\dot\delta_S)
\]
is supported on the subset $\Delta_\el(G,S,\zeta)$ of $\Delta_\el^\E(G,S,\zeta)$. 
\een
\end{prop}

\begin{proof}
We recall that as a $z$-extension, $\tilde G$ is a central extension of $G$ by a central induced torus $\tilde C$ over $F$ such that $\tilde G_\der$ is simply connected. We assume that $S$ is chosen such that $\tilde G$ and $\tilde\zeta$ are unramified away from $S$. Since $G_S\simeq \tilde G_S/\tilde C_S$, we therefore have a canonical isomorphism from $\C(G,S,\zeta)$ to $\C(\tilde G,S,\tilde\zeta)$. We shall also assume that the bases $\Gamma(\tilde G,\tilde\zeta),$ $\Delta^\E(\tilde G,\tilde\zeta),$ and $\Delta(\tilde G,\tilde\zeta)$ for $\tilde G$ are the images of the corresponding bases for $G$ under the canonical maps $\dot\gamma_S\to \tilde{\dot{\gamma}}_S$ and $\dot\delta_S\to \tilde{\dot{\delta}}_S$. 

Let us first show that the result is valid for the pair $(G,\zeta)$ if it is valid for the pair $(\tilde G,\tilde\zeta)$. We have to show that
\be
\label{21}
a^G_{r,s,\el}(\dot\gamma_S) =a^G_{r,s,\el}(\tilde{\dot{\gamma}}_S)
\ee
for any admissible element $\dot\gamma_S\in\Gamma(G_S,\zeta_S)$. To do so, we consider the expanson
\[
J^r_s(\dot f_S) = \sum_{M\in\L}|W^M_0||W^G_0|^{-1}\sum_{\dot\gamma_S\in\Gamma_\el(M,S,\zeta)}a^M_{r,s,\el}(\dot\gamma_S)J_M(\dot\gamma_S,\dot{f}_S)
\]
in \cite[(4.10)]{witf}, that is valid for any $\dot{f}_S\in \C_\adm(G,S,\zeta)$. From the map of functions $\dot{f}_S\to\tilde{\dot{f}}_S$ induced by the isomorphism $G_S \simeq \tilde{G}_S/\tilde{C}_S$, it follows that $J^r_s(\dot{f}_S)$ is equal to the distribution $J^r_s(\tilde{\dot{f}}_S)$ on $\tilde{G}$, and moreover the weighted orbital integrals satisfy
\[
J^G_M(\dot\gamma,\dot f) = J^{\tilde G}_{\tilde M}(\tilde{\dot\gamma},\tilde{\dot f}),\quad \dot\gamma_S\in\Gamma(G_S,\zeta_S).
\]
Assuming inductively that \eqref{21} holds if $G$ is replaced by any proper Levi subgroup $M$ of $G$, it follows that the terms with $M\neq G$ in the expansions for $J^r_s({\dot{f}}_S)$ and $J^r_s(\tilde{\dot{f}}_S)$ are equal, and by varying $\dot f_S$ we conclude that \eqref{21} holds. 

Since $\tilde G_\der$ is simply connected, there is an $L$-embedding $^L\tilde{G}'\to {^L\tilde{G}}$ which we can assume is unramified outside of $S$ by \cite[Lemma 7.1]{STF1}. The composition of $L$-embeddings $\G'\to {^LG}\to {^L\tilde G}$ maps $\G'$ into the image of $^L\tilde{G}'$ in $^L\tilde{G}$, giving an $L$-emebedding $\tilde\xi':\G'\to{^L\tilde{G}'}$ which makes up an auxiliary datum $(\tilde G',\tilde\xi')$ for $G'$. We thus have a bijection from $\E_\el(G,S)$ to $\E_\el(\tilde{G},S)$, together with an equality of transfer maps 
\[
{\dot f}_S' = (\tilde{\dot f}_S)',\quad \dot f_S\in\C_\adm(G,S,\zeta)
\]
by \cite[\S4.2]{LS}. Then if we consider the elliptic parts \eqref{ieell}, it follows that the difference
\[
I^\E_{r,s,\el}(\dot f_S)  - \varepsilon(G)S^G_{r,s,\el}(\dot f_S) = \sum_{G'\in \E^0_\el(G,S)}\iota(G,G')\hat{S}^{\tilde{G}'}_{r,s,\el}(\dot{f}'_S)
\]
equals
\[
I^\E_{r,s,\el}(\tilde{\dot{f}}_S)  - \varepsilon(\tilde{G})S^{\tilde G}_{r,s,\el}(\tilde{\dot{f}}_S) = \sum_{\tilde G'\in \E^0_\el(\tilde G,S)}\iota(\tilde G,\tilde G')\hat{S}^{\tilde{G}'}_{r,s,\el}(\tilde{\dot{f}}'_S),
\]
where we have used the fact that $\iota(G,G') = \iota(\tilde{G},\tilde{G}')$ \cite[(2.2)]{STF2}. On the other hand, it follows from \eqref{21} that the elliptic part \eqref{Idotell} satisfies \[
I_{r,s,\el}(\dot{f}_S) = I_{r,s,\el}(\tilde{\dot{f}}_S).
\]
Hence if $\varepsilon(G)=1$, we conclude from \eqref{ieelllem} and \eqref{selllem} that 
\[
I_{r,s,\el}^\E(\dot{f}_S) = I_{r,s,\el}^\E(\tilde{\dot{f}}_S),\quad S^{G}_{r,s,\el}({\dot{f}}_S) = S^{\tilde G}_{r,s,\el}(\tilde{\dot{f}}_S).
\]
The general induction hypothesis concerning Corollary \ref{glob1'} then allows us to compare the expansions of each one to deduce that
\[
a^{G.\E}_{r,s,\el}(\dot\gamma_S) =a^{G,\E}_{r,s,\el}(\tilde{\dot{\gamma}}_S)
\]
in general for any admissible element $\dot\gamma_S\in\Gamma^\E_\el(G,S,\zeta)$, and
\be
\label{24}
b^G_{r,s,\el}(\dot\delta_S) =b^G_{r,s,\el}(\tilde{\dot{\delta}}_S)
\ee
if $G$ is quasisplit for any admissible element $\dot\delta_S\in\Delta^\E_\el(G,S,\zeta)$. It follows then that the proposition is valid for the pair $(G,\zeta)$ if it is valid for the pair $(\tilde G,\tilde\zeta)$. This reduces the result to the case where $G=\tilde G$. 

We next want to show that the result is valid for $G$ with central datum $(Z,\zeta)$ if it holds for $(Z,\zeta)$ trivial. The map \eqref{zetav} provides a projection from the subspace of admissible functions $\C_\adm(G,S)$ onto $\C(G,S,\zeta)$. We then have
\[
I_{r,s,\el}(\dot f_S)=\int_{Z_{S,\o}\bs Z^1_S}I_{r,s,\el}(\dot f_{S,z}^1)\zeta(z)dz, \quad \dot f_S\in \C(G,S,\zeta)
\]
where $\dot f^1_S$ is any function in $\C_\adm(G,S)$ that maps to $\dot f_S$, and $\dot f^1_{S,z}$ is its translate by $z$. Moreover, the map \eqref{zetav} commutes with endoscopic transfer by \cite[Lemma 6.1]{witf}. To see this, observe that $(\dot f_S^1)'$ lies in $S\C(\tilde G',S,\tilde\eta')$, where $\tilde\eta'$ is the automorphic character of the central induced torus $\tilde C'$ defining the $z$-extension $\tilde G'$, and \eqref{zetav} induces a map from $(\dot f_S^1)'$ to a function in $S\C(\tilde G',S,\tilde\zeta')$ where $\tilde\zeta'=\tilde\zeta\tilde\eta'$. It follows then from \eqref{transfer} and \cite[Lemma 4.4.A]{LS} that
\[
(\dot f^1_{S,z})' = (\dot f^1_{S})'_z\tilde\eta'(z),\quad z\in Z^1_S,
\]
and integrating against $\zeta(z)$, the claim follows. Then as before, we have that
\[
I^\E_{r,s,\el}(\dot{f}^1_{S,z})  - \varepsilon(G)S^G_{r,s,\el}(\dot{f}^1_{S,z}) = \sum_{G'\in \E^0_\el(G,S)}\iota(G,G')\hat{S}^{\tilde{G}'}_{r,s,\el}((\dot{f}^1_{S})'_z)\tilde\eta'(z)
\]
for any $z\in Z^1_S$. Assume inductively that for any $G'\in\E^0_\el(G,S)$, the integrand in
\[
\int_{Z_{S,\o}\bs Z^1_S}\hat{S}^{\tilde{G}'}_{r,s,\el}((\dot f_{S}^1)'_z)\tilde\zeta'(z)dz
\]
is invariant under translation by $Z_{S,\o}$, and that the integral is equal to $\hat{S}_{r,s,\el}^{\tilde{G}'}(\dot f_S')$. It follows then that the integral
\[
\int_{Z_{S,\o}\bs Z^1_S}(I^\E_{r,s,\el}(\dot{f}^1_{S,z})  - \varepsilon(G)S^G_{r,s,\el}(\dot{f}^1_{S,z}))\zeta(z)dz
\]
equals
\[
\sum_{G'\in \E^0_\el(G,S)}\iota(G,G')\hat{S}^{\tilde{G}'}_{r,s,\el}(\dot{f}_{S}'),
\]
which in turn is equal to $I^\E_{r,s,\el}(\dot{f}_{S})  - \varepsilon(G)S^G_{r,s,\el}(\dot{f}_{S}).$ Then arguing as before, we conclude that 
\[
I^\E_{r,s,\el}(\dot f_S)=\int_{Z_{S,\o}\bs Z^1_S}I^\E_{r,s,\el}(\dot f_{S,z}^1)\zeta(z)dz, 
\]
in general, and
\be
\label{27}
S^G_{r,s,\el}(\dot f_S)=\int_{Z_{S,\o}\bs Z^1_S}S^G_{r,s,\el}(\dot f_{S,z}^1)\zeta(z)dz, 
\ee
if $G$ is quasisplit. Finally, choose $\dot f_S\in\C_\adm(G,S,\zeta)$ such that $\dot{f}_{S,G}(\dot\gamma_S)$ vanishes for any $\dot\gamma_S\in\Gamma(G_S,\zeta_S)$ whose semisimple part is not central, and moreover such that $\dot{f}^{G^*}_S=0$ if $G$ is quasisplit. We choose the associated function $\dot f_S^1$ similarly. Then assuming that the proposition holds for $G$ with $(Z,\zeta)$ trivial, it follows that $I^\E_{r,s,\el}(\dot f^1_{S,z}) = I_{r,s,\el}(\dot f_{S,z}^1)$  and $S^G_{r,s,\el}(\dot f_S)=0$. The expansions \eqref{ieelllem} and \eqref{selllem} then imply that the desired results holds for admissible elements $\dot\gamma_S\in\Gamma^\E_\el(G,S,\zeta)$ and $\dot\delta_S\in\Delta^\E_\el(G,S,\zeta)$ whose semisimple parts are not central, and hence for arbitrary $(G,\zeta)$.
\end{proof}

We can now complete the extension to the general case, which is Corollary \ref{cordesc}.

\begin{proof}[Proof of Corollary \ref{cordesc}]
We shall prove (b), and the proof of (a) will be similar to it. We assume therefore that $G$ is quasisplit. We claim that \eqref{corexp} holds for $(G,\zeta)$ if it holds for $(\tilde G,\tilde\zeta)$. By \eqref{24} it follows that the right hand side of \eqref{corexp} for $(G,\zeta)$ is equal to the corresponding expression for $(\tilde G,\tilde\zeta)$. Then since $\bar{\tilde{G}}=\bar{G}$, the sums over $d$ of the two expansions run over the same set, and there is a canonical bijection from $\dot\beta$ to $\tilde{\dot\beta}$ where by \eqref{24} again we have 
\[
b^{G_d}_{r,s,\el}(\dot\beta_S) =b^{\tilde G_d}_{r,s,\el}(\tilde{\dot{\beta}}_S),
\]
whence the claim follows. It suffices then to treat the case $\tilde G =G$.

We next show that if $G_\der$ is simply connected, and if \eqref{corexp} holds for $(Z,\zeta)$ trivial, then it holds for arbitrary $(Z,\zeta)$ as well. Let $\dot{f}_S\in\C_\adm(G,S,\zeta)$ such that $\dot{f}_{S,G}^\E$ is supported on $\Delta(G_S,\zeta_S)$. Then \eqref{selllem} can be written as
\[
S^G_{r,s,\el}(\dot{f}_S) = \sum_{\dot\delta_S\in\Delta_\el(G,S,\zeta)}b^{G}_{r,s,\el}(\dot\delta_S)\dot{f}_{S}^{G}(\dot\delta_S)
\]
and from \eqref{27} we have also
\[
S^G_{r,s,\el}(\dot f_S)=\int_{Z_{S,\o}\bs Z^1_S}\sum_{\dot\delta^1_S\in\Delta_\el(G,S)}b^{G}_{r,s,\el}(\dot\delta_S^1)\dot{f}_{S}^{1,G}(z\dot\delta_S^1)\zeta(z)dz.
\]
We are assuming the coefficients $b^{G}_{r,s,\el}(\dot\delta_S^1)$ satisfy \eqref{corexp}, which reduce to the \eqref{propbdesc}. Substituting, we obtain 
\be
\label{211}
\int_{Z_{S,\o}\bs Z^1_S}\sum_{\dot\beta}\sum_{d}j^{G^*}(S,d^1)b^{G_{d^1}^*}_{r,s,\el}(\dot\beta^1)\dot{f}_{S}^{1,G^*}(zd^1\dot\beta^1)\zeta(z)dz
\ee
where the sum over $d^1$ runs over classes in $\Delta_\ss(G^*)$ that are bounded at each $v\not\in S$, and $\dot\beta^1$ runs over $\Delta_\unip(G^*_{d^1,S})$.

Our goal then is to compare the expansions \eqref{corexp} and \eqref{211}. First, recall that $b^{G_{d^1}^*}_{r,s,\el}(\dot\beta^1)$ is supported on the subset $\Delta_{\unip,\der}(G^*_{d^1,S})$ of $\Delta_{\unip}(G^*_{d^1,S})$, which we can choose to be in bijection with $\Delta_{\unip,\der}(G^*_{d^1,S},\zeta_S)$ by the map \eqref{dfc}. Second, observe that there is a surjection $d^1\to d$ from $\Delta_\ss(G^*)$ to $\Delta_\ss(\bar{G}^*)$ of classes that are bounded away from $S$, and $Z_{S,\o}$ acts transitively on the fibers of this map. The stabilizer of $d^1$ in $Z_{S,\o}$ is isomorphic to $(\bar{G}^*_{d,+}/\bar{G}^*_{d})(F)$ by the map 
\[
g\to z= g^{-1}d^1g(d^1)^{-1}.
\]
Third, by \cite[p.177]{STF1}, the coefficients $j^{G^*}(S,d^1)$ and $j^{\bar{G}^*}(S,d)$ are equal. Fourth, assume inductively that if $d^1$ is not central in $G$, then
\[
b^{G_{d^1}^*}_{r,s,\el}(\dot\beta^1) =b^{ G_d^*}_{r,s,\el}({\dot{\beta}}).
\]
Then by \eqref{24} it is enough to assume that the analogue holds for some $z$-extension of the group $G^*_{d^1}=G^*_d$. Finally, we note that
\[
\int_{Z^1_S}\dot{f}^1_S(zx)\zeta(z)dz = \dot{f}_S(x),\quad x\in G^1_S,
\]
and assume first that $\dot f_S$ vanishes on an invariant neighborhood of the center of $G_S$, then $\dot{f}_{S}^{1,G^*}(zd^1\dot\beta^1)$ vanishes if $d^1$ is central in $G$, then by varying $\dot{f}_S$ appropriately it follows that \eqref{corexp} holds for any $\dot\delta_S$ whose semisimple part is not central. Next, removing the vanishing condition on $\dot f_S$, by a comparison of terms we conclude that \eqref{corexp} also holds for $\dot\delta$ whose semisimple part is central, and moreover, that
\[
b^{G_{d}^*}_{r,s,\el}(\dot\beta^1) =b^{ G^*}_{r,s,\el}({\dot{\beta}})
\]
holds for any $\dot\beta$ in the subset $\Delta_{\unip,\der}(G^*_{d^1,S},\zeta_S)$ of $\Delta_{\unip}(G^*_{d^1,S},\zeta_S)$ on which $b^{ G^*}_{r,s,\el}$ is supported. This completes the induction argument.
\end{proof}


\subsection{The unipotent part}

We now consider the unipotent analogues of the elliptic part of the trace formula. We shall introduce objects that will be parallel to those considered in Section 3. Let $\Gamma_\unip(G,V,\zeta)$ denote the subset of classes in $\Gamma_\el(G,V,\zeta)$ whose semisimple parts are trivial, and similarly we let $\K_\unip^V(\bar{G},S)$ be the corresponding unipotent subset of $\K_\el^V(\bar{G},S)$. More generally, we shall let the subscript `unip' denote the unipotent variant of the elliptic objects analogously. We then define the unipotent part of the invariant trace formula
\be
\label{iunplem}
I_{r,s,\unip}(f,S)= \sum_{\alpha\in\Gamma_\unip(G,V,\zeta)}a^G_{r,s,\unip}(\alpha,S)f_G(\alpha)
\ee
with unipotent coefficients 
\be
\label{arsunp}
a^G_{r,s,\unip}(\alpha,S) = \sum_{k\in \K_\unip^V(\bar{G},S)}a^G_{r,s,\unip}(\alpha\times k)r_G(k,b),\quad \alpha\in\Gamma_\unip(G,V,\zeta).
\ee
We shall also define endoscopic and stable analogues inductively by the formula
\[
I^{\E}_{r,s,\unip}(f,S) = \sum_{G'\in\E^0_\el(G)}\iota(G,G')\hat{S}^{\tilde{G}'}_{r,s,\unip}(f',S) +\varepsilon(G)S^{G}_{r,s,\unip}(f,S)
\]
with the requirement that $I^{\E}_{r,s,\unip}(f,S) = I_{r,s,\unip}(f,S)$ in the case that $G$ is quasisplit. Using the natural variation of the argument \cite[Lemma 7.2]{STF1} we obtain expansions
\be
\label{ieunplem}
I^\E_{r,s,\unip}(f) = \sum_{\alpha\in\Gamma^\E_\unip(G,V,\zeta)}a^{G,\E}_{r,s,\unip}(\alpha)f_{G}(\alpha),
\ee
if $G$ is arbitrary, and
\be
\label{seunplem}
S^G_{r,s,\unip}(f) = \sum_{\beta\in\Delta^\E_\unip(G,V,\zeta)}b^{G}_{r,s,\unip}(\beta)f_{G}^\E(\beta),
\ee
if $G$ is quasisplit, where
\be
\label{aersunp}
a^{G,\E}_{r,s,\unip}(\alpha,S) = \sum_{k\in\K_\unip^{V,\E}(\bar{G},S)}a^{G,\E}_{r,s,\unip}(\alpha\times k)r_G(k,b), \quad \alpha\in\Gamma^\E_\unip(G,V,\zeta)
\ee
and
\be
\label{brsunp}
b^G_{r,s,\unip}(\beta,S) = \sum_{\ell\in\L_\unip^V(\bar{G},S)}b^G_{r,s,\beta}(\beta\times \ell)r_G(k,b),\quad \beta\in\Delta^\E_\unip(G,V,\zeta)
\ee
respectively.

Let $Z(G)_{V,\o}$ be the subgroup of elements $z$ in $Z(G)(F)$ such that for every $v\not\in V$, $z_v$ is bounded in $Z(G)(F_v)$. It acts discontinuously on $G_V$, and so does the quotient 
$
Z(\bar{G})_{V,\o}=Z(G)_{V,\o}Z_V/Z_V
$
on $\bar{G}_V$. For any $z\in Z(\bar{G})_{V,\o}$, we set
\[
I_{z,r,s,\unip}(f,S) = I_{r,s,\unip}(f_z,S), 
\]
and similarly 
\[
I^\E_{z,r,s,\unip}(f,S) = I^\E_{r,s,\unip}(f_z,S),
\]
and 
\[
S^G_{z,r,s,\unip}(f,S) = S^G_{r,s,\unip}(f_z,S).
\]
Also, following \S\ref{A4} we define $\C^\uns(G_V,\zeta_V)$ be the subset of functions $f\in \C(G_V,\zeta_V)$ whose stable orbital integrals vanish, that is, $f^G=0$. We call such functions unstable.

\begin{cor}\label{parunip}\
\ben
\item[(a)] If $G$ is arbitrary and $f\in \C(G_V,\zeta_V)$, we have  
\[
I^\E_{r,s,\el}(f,S)-I_{r,s,\el}(f,S) = \sum_{z\in Z(\bar{G})_{V,\o}}(I^\E_{z,r,s,\unip}(f,S)-I_{z,r,s,\unip}(f,S)).
\]
\item[(b)] If $G$ is quasisplit and $f\in \C^\uns(G_V,\zeta_V)$, we have
\[
S^G_{r,s,\el}(f,S)= \sum_{z\in Z(\bar{G})_{V,\o}}S^G_{z,r,s,\unip}(f,S).
\]
\een
\end{cor}
\begin{proof}
We begin with (a). Using the expansions \eqref{iellS} and \eqref{ielS}, it follows that
\[
I^\E_{r,s,\el}(f,S)-I_{r,s,\el}(f,S)  = \sum_{\gamma\in\Gamma^\E_\el(G,S,\zeta)} (a^{G,\E}_{r,s,\el}(\gamma,S) - a^G_{r,s,\el}(\gamma,S))f_G(\gamma),
\]
and by \eqref{agersell} and \eqref{agrsell}, we have that
\[
a^{G,\E}_{r,s,\el}(\gamma,S) - a^G_{r,s,\el}(\gamma,S) = \sum_{k\in\K_\el^{V,\E}(\bar{G},S)}(a^{G,\E}_{r,s,\el}(\gamma\times k) - a^{G}_{r,s,\el}(\gamma\times k))r_G(k,b). 
\]
By Proposition \ref{prop21}(a), the coefficients $a^{G,\E}_{r,s,\el}(\gamma\times k)$ and $a^{G}_{r,s,\el}(\gamma\times k)$ are equal if the semisimple part of $\gamma\times k$ is not central in $G$, hence $a^{G,\E}_{r,s,\el}(\gamma,S) = a^G_{r,s,\el}(\gamma,S)$ for such $\gamma$. On the other hand, if the semisimple part of $\gamma$ is central, it has the Jordan decomposition
\[
\gamma = z\alpha, \qquad z\in Z(\bar{G})_{V,\o},\alpha\in\Gamma^\E_\unip(G,V,\zeta),
\]
we can therefore break the sum over $\gamma$ into a double sum over $z$ and $\alpha$. In this case, Corollary \eqref{cordesc}(a) implies that 
\[
a^{G,\E}_{r,s,\el}(\gamma,S) - a^G_{r,s,\el}(\gamma,S) = a^{G,\E}_{r,s,\unip}(\alpha,S) - a^G_{r,s,\unip}(\alpha,S),
\]
and using the definition
\[
I^\E_{z,r,s,\unip}(f,S)-I_{z,r,s,\unip}(f,S) = \sum_{\alpha\in\Gamma^\E_\unip(G,V,\zeta)}(a^{G,\E}_{r,s,\unip}(\alpha)-a^{G,\E}_{r,s,\unip}(\alpha))f_{G}(z\alpha),
\]
the required formula (a) follows.

The case of (b) is similar. Recall the expansions
\[
S^{G}_{r,s,\el}(f,S) = \sum_{\delta\in\Delta^\E_\el(G,S,\zeta)} b^G_{r,s,\el}(\delta,S)f_G^\E(\delta)
\]
and 
\[
b^G_{r,s,\el}(\delta,S) = \sum_{\ell\in\L_\el^V(\bar{G},S)}b^G_{r,s,\el}(\delta\times \ell)r_G(k,b)
\]
from \eqref{sellS} and \eqref{bgrsell} respectively. Since $f$ is unstable, it follows that $f^\E_G(\delta)$ vanishes on the subset $\Delta_\el(G,S,\zeta)$ of $\Delta^\E_\el(G,S,\zeta)$. On the other hand, suppose $\delta$ belongs to the complement of $\Delta_\el(G,S,\zeta)$ in $\Delta^\E_\el(G,S,\zeta)$. If its semisimple part is not central in $G$, Proposition \ref{21}(b) implies that $b^G_{r,s,\el}(\delta,S)$ vanishes. If the semisimple part of $\delta$ is central, it again has the Jordan decomposition
\[
\delta = z\beta, \qquad z\in Z(\bar{G})_{V,\o},\beta\in\Delta^\E_\unip(G,V,\zeta),
\]
we can therefore break the sum over $\delta$ into a double sum over $z$ and $\beta$. In this case, Corollary \eqref{cordesc}(b) implies that 
\[
b^G_{r,s,\el}(\gamma,S)= b^G_{r,s,\unip}(\beta,S),
\]
and by the definition of $S^G_{z,r,s,\unip}(f,S)$ the required formula (b) follows.
\end{proof}

This completes the primary application of the descent formula for the global geometric coefficients, which has allowed us to reduce the study of the elliptic part of the trace formula to the unipotent terms. 

\section{Stabilization}

\label{sectionstab}

\subsection{The parabolic parts} 

We recall that our goal is to prove Theorems \eqref{globge} and \eqref{globsp}. The remainder of the argument, following \cite{STF3}, is now a largely formal process, having the preceding results in place, hence we allow ourselves to be brief. Our induction will be based on an integer $d_\der$. We assume that the theorems hold if $\dim(G_\der)<d_\der$, and if $G$ is not quasisplit we shall assume that $\dim(G_\der)=d_\der$ and the relevant theorems hold for the quasisplit inner $K$-form of $G$. This, together with the assumption of Conjecture \ref{conj} on the geometric side are the hypotheses that we will carry. 

We shall first analyze the parabolic parts of the geometric expansions, which are defined to be the terms associated to $M\neq G$, 
\[
I_{r,s,\pa}(f)= \sum_{M\in\L^0}|W^M_0||W^G_0|^{-1}\sum_{\gamma\in\Gamma(M,V,\zeta)}a^M_{r,s}(\gamma)I_M(\gamma,f),
\]
and similarly
\[
I_{r,s,\pa}^\E(f)= \sum_{M\in\L^0}|W^M_0||W^G_0|^{-1}\sum_{\gamma\in\Gamma^\E(M,V,\zeta)}a^{M,\E}_{r,s}(\gamma)I^\E_M(\gamma,f),
\]
and $S^G_{r,s,\pa}(f)$ defined as
\[
\sum_{M\in\L^0}|W^M_0||W^G_0|^{-1}\sum_{M'\in\E_\el(M,V)}\iota(M,M')\sum_{\delta'\in\Delta(\tilde{M}',V,\tilde\zeta')}b^{\tilde M'}_{r,s}(\delta')S^G_M(M',\delta',f).
\]
We shall use this to remove the terms in the spectral expansion that are not discrete, relating the parabolic parts to the unipotent and discrete parts of the trace formula only. Our goal will then be to show that each these terms vanish in an appropriate way.

\begin{lem} \label{pardiscu} \ 
\ben
\item[(a)] If $G$ is arbitrary and $f\in \C(G_V,\zeta_V)$, 
\begin{align*}
I_{r,s,\pa}^\E(f) &- I_{r,s,\pa}(f) \\
&= (I_{r,s,\disc}^\E(f) - I_{r,s,\disc}(f)) - \sum_{z\in Z(\bar{G})_{V,\o}}(I^\E_{r,s,z,\unip}(f,S)-I_{r,s,z,\unip}(f,S)).
\end{align*}
\item[(b)] If $G$ is quasisplit and $f\in \C^\uns(G_V,\zeta_V)$,
\[
S^G_{r,s,\pa}(f) = S^G_\disc(f) = \sum_{z\in Z(\bar{G})_{V,\o}} S^G_{r,s,z,\unip}(f,S).
\]
\een
\end{lem}

\begin{proof}
By \eqref{Irsg}, \eqref{iorb} and Proposition \ref{geomprop}(a), we have by definition
\[
I^\E_{r,s,\pa}(f) - I_{r,s,\pa}(f) = (I^\E_{r,s}(f) - I^r_s(f)) - (I^\E_{r,s,\orb}(f) - I_{r,s,\orb}(f))
\]
and by \eqref{iorb} and \eqref {ieorblem}, we have
\[
I^\E_{r,s,\orb}(f) - I_{r,s,\orb}(f) = \sum_{\gamma\in\Gamma^\E(G,V,\zeta)}(a^{G,\E}_{r,s}(\gamma)-a^{G}_{r,s}(\gamma))f_G(\gamma).
\]
Then applying the global induction hypothesis to \eqref{arsgam} and Proposition \ref{propgeom}(a), it follows that
\[
a^{G,\E}_{r,s}(\gamma) - a^{G}_{r,s}(\gamma)=  a^{G,\E}_{r,s,\el}(\gamma,S) - a^{G}_{r,s,\el}(\gamma,S),
\]
whence from \eqref{iellS} and \eqref{ielS} we have
\[
I^\E_{r,s,\orb}(f) - I_{r,s,\orb}(f)  = I^\E_{r,s,\el}(f,S) - I_{r,s,\el}(f,S).
\]
Then by Corollary \ref{parunip}(a) we see that $I^\E_{r,s,\pa}(f) - I_{r,s,\pa}(f)$ equals
\[
(I_{r,s}^\E(f) - I_{r,s}(f)) -\sum_{z\in Z(\bar{G})_{V,\o}}(I^\E_{r,s,z,\unip}(f,S)-I_{r,s,z,\unip}(f,S)).
\]
Moreover, by \eqref{Irss} and Proposition \ref{specprop}(a) we can express the spectral expansion of $I_{r,s}^\E(f) - I_{r,s}(f)$ as the sum of $I_{r,s,\unit}^\E(f) - I_{r,s,\unit}(f)$ and
\[
\sum_{M\in\L^0}|W^M_0||W^G_0|^{-1}\int_{\Pi^\E(M,V,\zeta)}(a^{M,\E}_{r,s}(\pi)I^\E_M(\pi,f)- a^{M}_{r,s}(\pi)I_M(\pi,f))d\pi.
\]
By the global induction hypothesis again the coefficients $a^{M,\E}_{r,s}(\pi)$ and $a^{M}_{r,s}(\pi)$ are equal, and by Corollary \ref{corb}(a), so too are the local distributions $I^\E_M(\pi,f)$ and $I_M(\pi,f)$. On the other hand, by \ref{iunit} and \ref{ieunitlem} we have 
\[
I_{r,s,\unit}^\E(f) - I_{r,s,\unit}(f) = \int_{\Pi^\E_\disc(G,V,\zeta)}(a^{G,\E}_{r,s}(\pi)-a^{G}_{r,s}(\pi))f_{G}(\pi)d\pi.
\]
Then by applying the global induction hypothesis to \eqref{arspi} and Theorem \ref{specthm}(a), it follows that 
\[
a^{G,\E}_{r,s}(\pi) - a^{G}_{r,s}(\pi)=  a^{G,\E}_{r,s,\disc}(\pi) - a^{G}_{r,s,\disc}(\pi),
\]
whence from \eqref{idiscsum} and \eqref{idiscS} we have
\[
I^\E_{r,s}(f) - I^r_s(f) = I^\E_{r,s,\disc}(f) - I_{r,s,\disc}(f).
\]
This gives the required formula (a). 

The proof of (b) follows in a similar fashion. In this case, it follows from Proposition \ref{geomprop} that $S^G_{r,s,\pa}(f)$ equals $S^G_{r,s}(f) - S^G_{r,s,\orb}(f)$, then using the expansion \eqref{seorblem} of the orbital part and the fact that $f$ is unstable, it follows from \eqref{sellS} and Proposition \eqref{geomprop}(b) that 
\[
S^G_{r,s,\orb}(f) = \sum_{\delta\in\Delta^\E(G,V,\zeta)}b^G_\el(\delta,S)f^\E_G(\delta) = S^G_\el(f,S).
\]
Then by Corollary \ref{parunip}(b), we see that $S^G_{r,s,\pa}(f)$ equals
\[
S^G_{r,s}(f) -  \sum_{z\in Z(\bar{G})_{V,\o}}S^G_{z,r,s,\unip}(f,S).
\]
Moreover by \eqref{sedisclem} and Proposition \ref{specprop}(b) we can express the spectral expansion of $S^G_{r,s}(f)$ as the sum of 
\[
S^G_{r,s,\unit}(f) = \int_{\Phi^\E(G,V,\zeta)}b^{G}_{r,s}(\phi)f_{G}^\E(\phi)d\phi.
\]
and
\[
\sum_{M\in\L^0}|W^M_0||W^G_0|^{-1}\sum_{M'\in\E_\el(M,V)}\iota(M,M')\int_{\Phi(\tilde M',V,\tilde\zeta')}b^{\tilde M'}_{r,s}(\phi')S^G_M(M',\phi',f)d\phi.
\]
The second term again vanishes by Corollary \ref{corb}(b) and the fact that $f$ is unstable. Finally, from Theorem \ref{specthm}(b) and \eqref{sdiscS} that 
\[
S^G_{r,s,\unit}(f) = \sum_{\phi\in\Phi^\E_\unit(G,V,\zeta)}b^G_{r,s,\disc}(\phi)f^\E_G(\phi)= S^G_{r,s,\disc}(f).
\]
This gives the required formula (b). 
\end{proof}

It follows from Corollary \ref{cora} that $I^\E_{r,s,\pa}(f)- I_{r,s,\pa}(f)$ equals 
\be
\label{ipar0}
\sum_{M\in\L} |W_0^M||W_0^G|^{-1}\sum_{\gamma\in\Gamma(M,V,\zeta)}a^{M}_{r,s}(\gamma)(I^\E_M(\gamma,f)-I_M(\gamma,f)) =0 
\ee
for any $f\in\C(G_V,\zeta_V)$ and $G$ arbitrary, and $S^G_{r,s,\pa}(f)$ equals
\begin{align}
&\sum_{M\in\L^0} |W_0^M||W_0^G|^{-1}\sum_{M'\in\E_{\el}(M,V)}\iota(M,M')\sum_{\delta'\in\Delta(\tilde M',V,\tilde\zeta')}b^{\tilde M'}_{r,s}(\delta')S_M(\tilde M',\delta',f)\notag\\
&= \sum_{M\in\L^0} |W_0^M||W_0^G|^{-1}\sum_{\delta^*\in\Delta(M^*,V,\zeta^*)}b^{M^*}_{r,s}(\delta^*)S_M(M^*,\delta^*,f)=0\label{spar0}
\end{align}
for any $f\in\C^\uns(G_V,\zeta_V)$ and $G$ quasisplit.

\subsection{Proofs of the global theorems}

We have shown that the left hand side of Lemma \ref{pardiscu} vanishes. We shall next establish that the discrete parts on the right hand side also vanish.

\begin{prop}
\label{disc0}\ 
\ben
\item[(a)] If $G$ is arbitrary and $f\in \C(G_V,\zeta_V)$, we have
\[
I^\E_{r,s,\disc}(f) - I_{r,s,\disc}(f)=0.
\]
\item[(b)] If $G$ is quasisplit and $f\in \C^\uns(G_V,\zeta_V)$, we have
\[
S^G_{r,s,\disc}(f)=0.
\]
\een
\end{prop}

\begin{proof}
The proof relies on a difficult local argument of Arthur that culminates in \cite[Corollary 5.2]{STF3}. As in \S\ref{A4}, for any $v\in V$, we recall that $f_v\in \C(G_v,\zeta_v)$ is called $M$-cuspidal if $f_{v,L_v}=0$ for any $L_v\in \L_v$ that does not contain a $G_v$-conjugate of $M_v$, and we define $\C_M(G_V,\zeta_V)$ to be the subspace of $\C(G_V,\zeta_V)$ spanned by functions 
\[
f= f_v f^v  = f_v \prod_{w\neq v}f_w
\]
such that $f_v$ is $M$-cuspidal at two places $v\in V$. If $G$ is quasisplit, we set
\[
\C^\uns_M(G_V,\zeta_V) = \C_M(G_V,\zeta_V)\cap \C^\uns(G_V,\zeta_V).
\]
If $v$ is a nonarchimedean place, we define
$
\C(G_v,\zeta_v)^{0}
$
to be the subspace of functions $f\in \C(G_v,\zeta_v)$ such that $f_{v,G}(z_v\alpha_v)=0$ for any $z_v$ in the center of $\bar{G_v}=G_v/Z_v$ and $\alpha_v$ in the parabolic subset $R_\text{unip,par}(G_v,\zeta_v)$ of the basis $R_\text{unip}(G_v,\zeta_v)$ of unipotent orbital integrals in \cite[\S3]{STF3}. We also write $\C(G,V,\zeta)^{0}$ for the product of functions $f_v \in \C(G_v,\zeta_v)^{0}$ for $v\in V$, and similarly for $\C(G,V,\zeta)^{00}$. We shall denote by the intersections of these various spaces by using overlapping notation, for example, we write  $\C_M(G_v,\zeta_v)^0 = \C_M(G_v,\zeta_v)\cap \C(G_v,\zeta)^0$. We also write $W(M)$ for the Weyl group of $(G,M)$. 

Following the argument of \cite[Lemma 2.3]{STF3}, it follows that if $G$ is arbitrary, then $I^\E_{r,s,\pa}(f) - I_{r,s,\pa}(f)$ is equal to
\[
|W(M)|^{-1}\sum_{v\in V_\mathrm{fin}(G,M)}\sum_{\gamma\in\Gamma(M,V,\zeta)}a^{M}_{r,s}(\gamma)I^\E_M(\gamma_v,f_v) - I_M(\gamma_v,f_v))f^v_M(\gamma^v)
\]
for any $f\in \C_M(G_V,\zeta_V)$, and if $G$ is quasisplit, then $S^G_{r,s,\pa}(f)$ is equal to
\begin{align*}
|W(M)|^{-1}\sum_{M'\in\E_\el(M,V)}&\iota(M,M')\\
&\sum_{v\in V_\mathrm{fin}(G,M)}\sum_{\delta'\in\Delta(\tilde{M}',V,\tilde\zeta')}b^{\tilde M'}_{r,s}(\delta')S^G_M(M'_v,\delta'_v,f_v)(f^v)^{M'}((\delta')^v)
\end{align*}
for any $f\in \C_M^\uns(G_V,\zeta_V)$. A careful examination of the proof of \cite[Corollary 5.2]{STF3} shows that the argument carries over almost in its entirety, except that we must be careful about the construction of two families of mappings $\varepsilon_M$ and $\varepsilon^M$ in \cite[Proposition 3.1]{STF3} whose domains need to be extended to noncompactly supported functions, namely
\[
\varepsilon_M:\C(G_v,\zeta_v)^0\to I\C_\ac(M_v,\zeta_v)
\]
such that 
\[
\varepsilon_M(f_v,\gamma_v)=I^\E_M(\gamma_v,f_v) - I_M(\gamma_v,f_v), \quad \delta_v\in\Gamma(M_v,\zeta_v)
\]
if $G$ is arbitrary, and
\[
\varepsilon^{M}:\C^\uns(G_v,\zeta_v)^0\to S\C_\ac(M_v,\zeta_v), 
\]
such that $\varepsilon^{M}(f_v,\delta_v) = S^G_M(f_v,\delta_v)$ for $\delta_v\in\Delta(M_v,\zeta_v)$, and 
\[
\varepsilon^{M'}:\C^\uns(G_v,\zeta_v)^0\to S\C_\ac(\tilde M'_v,\tilde \zeta_v), \quad M'\in\E_\el^0(M)
\]
such that $\varepsilon^{M'}(f_v,\delta_v') = S^G_M(M'_v,\delta'_v,f_v)$ for $\delta'_v\in\Delta(\tilde M'_v,\tilde\zeta'_v)$ if $G$ is quasisplit. Here $I\C_\ac$ and $S\C_\ac$ are defined as natural analogues of $\I_\ac$ and $\SS_\ac$. These maps extend naturally to the larger spaces, as they are defined in terms of local tempered distributions. While the construction of these functions require endoscopic and stable variants of the functions $\theta$ and $^c\theta$ in the unpublished reference [A12] of \cite{STF3}, they have now also been constructed in the more general twisted setting for smooth compactly supported test functions in \cite[VIII,IX]{MW2}. The constructions also rely on germ expansions of endoscopic and stable orbital integrals in the unpublished reference [A11] of \cite{STF3}, which are also described in \cite[\S4]{stablegerms} and also in \cite{MW2} in a slightly different formulation. More importantly, the variants $^c\varepsilon_M$ and $^c\varepsilon^M$ are no longer compactly supported, but can be shown to be rapidly decreasing, which is sufficient for applications. We also note that the finite partitions of unity required in \cite[pp.793,796]{STF3} need only be replaced by smooth partitions of unity obtained from an open covering of the support of $f$. Then arguing as in \cite[Corollary 3.3]{STF3} it follows that if $G$ is arbitrary,
\[
I^\E_{r,s,\pa}(f) - I_{r,s,\pa}(f) = |W(M)|^{-1}\hat{I}^M_{r,s}(\varepsilon_M(f)),
\]
and if $G$ is quasisplit, 
\[
S^G_{r,s,\pa}(f) = |W(M)|^{-1}\sum_{M'\in\E_\el(M,V)}\iota(M,M')\hat{S}^{\tilde M'}_{r,s}(\varepsilon^{M'}(f)).
\]
The remainder of the proof thus follows in the same way, and is moreover simplified by the absolute convergence of the spectral side \cite{FLM}.	
\end{proof}

We then combine the preceding result with Lemma \ref{pardiscu} and equations \eqref{ipar0} and \eqref{spar0} to deduce the following vanishing result.

\begin{cor}\label{unip0}\
\ben
\item[(a)]  If $G$ is arbitrary and $f\in \C(G_V,\zeta_V)$, we have
\[
\sum_{z\in Z(\bar{G})_{V,\o}}(I^\E_{r,s,z,\unip}(f,S)-I_{r,s,z,\unip}(f,S))=0.
\]
\item[(b)] If $G$ is quasisplit and $f\in \C^\uns(G_V,\zeta_V)$, we have
\[
\sum_{z\in Z(\bar{G})_{V,\o}} S^G_{r,s,z,\unip}(f,S).
\]
\een
\end{cor}

We are now ready to complete the main global Theorems \ref{globge} and \ref{globsp}. Recall that we have reduced them to Corollaries \ref{glob1'} and \ref{glob2'} respectively.

\begin{proof}[Proof of Theorem \ref{globge}]
We recall that by the descent formula Proposition \ref{cordesc}, we can assume that $\dot\gamma_S$ and $\dot\delta_S$ are admissible elements in $\Gamma^\E_\unip(G,S,\zeta)$ and $\Delta^\E_\unip(G,S,\zeta)$ respectively. We shall take $S=V\supset V_\ram(G,\zeta)$ and $f=\dot f_S$ to be an admissible function in $\C(G_S,\zeta_S)$. Using \eqref{ieunplem} and \eqref{iunplem} we express Corollary \ref{unip0}(a) as
\[
\sum_{z\in Z(\bar{G})_{S,\o}} \sum_{\dot\alpha_S\in\Gamma^\E_\unip(G,S,\zeta)}(a^{G,\E}_{r,s,\el}(\dot\alpha_S)-a^{G}_{r,s,\el}(\dot\alpha_S))\dot{f}_{S,G}(z\dot\alpha_S) =0
\]
where we have used the fact that $a^{G,\E}_\el(\dot\alpha_S) = a^{G,\E}_\unip(\dot\alpha_S,S)$ by Corollary \ref{cordesc}(a) and \eqref{aersunp}, and $a^{G}_\el(\dot\alpha_S) = a^{G}_\unip(\dot\alpha_S,S)$ by Lemma \ref{adesclem} and \eqref{arsunp}. On the other hand since, the linear forms 
\[
\dot{f}_S\to \dot f_{S,G}(z\dot\alpha_S),\quad z\in Z(\bar{G})_{S,\o},\dot\alpha_S\in\Gamma_\unip^\E(G,S,\zeta)
\]
are linearly independent on $\C_\text{adm}(G_S,\zeta_S)$, and it follows then that 
\[
a^{G,\E}_{r,s,\el}(\dot\alpha_S)-a^{G}_{r,s,\el}(\dot\alpha_S)=0
\]
for any $\dot\alpha_S\in\Gamma^\E_\unip(G,S,\zeta)$. This proves Theorem \ref{globge}(a) for unipotent $\dot\gamma_S$, hence in general.

Next suppose that $G$ is quasisplit, and write $\Delta^{\E,0}_\unip(G,S,\zeta)$ for the complement of $\Delta_\unip(G,S,\zeta)$ in $\Delta^{\E}_\unip(G,S,\zeta).$ Using \eqref{seunplem} we express Corollary \ref{unip0}(b) as
\[
\sum_{z\in Z(\bar{G})_{S,\o}} \sum_{\dot\beta_S\in\Delta^{\E,0}_\unip(G,S,\zeta)}b^{G}_{r,s,\el}(\dot\beta_S) = 0
\]
for unstable functions $\dot f_S$, and where we have used the fact that $b^{G}_\el(\dot\beta_S) = b^{G}_\unip(\dot\beta_S,S)$ by Corollary \ref{cordesc}(b) and \eqref{brsunp}. On the other hand since, the linear forms 
\[
\dot{f}_S\to \dot f_{S,G}^\E(z\dot\beta_S),\quad z\in Z(\bar{G})_{S,\o},\dot\beta_S\in\Delta_\unip^{\E,0}(G,S,\zeta)
\]
are linearly independent on the intersection of $\C_\text{adm}(G_S,\zeta_S)$ and $\C^\uns(G_S,\zeta_S)$, and it follows then that 
\[
b^{G}_{r,s,\el}(\dot\beta_S)=0
\]
for any $\dot\beta_S\in\Gamma^{\E,0}_\unip(G,S,\zeta)$. This proves Theorem \ref{globge}(a) for unipotent $\dot\beta_S$, hence in general.
\end{proof}

\begin{proof}[Proof of Theorem \ref{globsp}]
We shall take $V\supset V_\ram(G,\zeta)$ and $\dot f = f\times b^V$ for $f\in \C(G_V,\zeta_V)$. Using \eqref{idisc}, \eqref{iedisclem}, and \eqref{idiscdot} we can express Proposition \ref{disc0}(a) as 
\[
\sum_{\dot\pi\in\Pi^\E_\disc(G,\zeta)}(a^{G,\E}_{r,s,\disc}(\dot\pi)-a^G_{r,s,\disc}(\dot\pi))\dot{f}_G(\dot\pi) = 0. 
\]
On the other hand since, the linear forms 
\[
\dot f\to\dot f_G(\dot\pi),\quad \dot\pi\in\Pi_\disc^\E(G,\zeta)
\]
are linearly independent, and it follows then that 
\[
a^{G,\E}_{r,s,\disc}(\dot\pi)-a^{G}_{r,s,\disc}(\dot\pi)=0
\]
for any $\dot\pi\in\Pi_\disc(G,\zeta)$. This proves Theorem \ref{globsp}(a).

Next suppose that $G$ is quasisplit, and write $\Phi^{\E,0}_\disc(G,\zeta)$ for the complement of $\Phi_\disc(G,\zeta)$ in $\Phi^{\E}_\disc(G,\zeta).$ Using \eqref{sdisclem} and \eqref{sdiscdot} we express Proposition \ref{disc0}(b) as 
\[
\sum_{\dot\phi\in\Phi^{\E,0}_\disc(G,\zeta)}b^G_{r,s,\disc}(\dot\phi)\dot{f}^\E_G(\dot\phi)=0
\]
for unstable functions $\dot f$, whereby $\dot f^\E_G$ vanishes on $\Phi_\disc(G,\zeta)$. On the other hand since, the linear forms 
\[
\dot f\to\dot f^\E_G(\dot\phi),\quad \dot\phi\in\Phi_\disc^{\E,0}(G,\zeta)
\]
are linearly independent, and it follows then that 
\[
b^{G}_{r,s,\disc}(\dot\phi)=0
\]
for any $\dot\phi\in\Phi^{\E,0}_\disc(G,\zeta)$. This proves Theorem \ref{globsp}(b).
\end{proof}

Having now established the main global theorems regarding the geometric and spectral coefficients, the derivation of the endoscopic and stable trace formulas in Theorem 1 in the introduction then follows by defining inductively
\[
I^\E_{r,s}(f) = \sum_{G'\in\E^0_\el(G,V)}\iota(G,G')\hat{S}'_{r,s}(f) + \varepsilon(G)S^r_s(f)
\]
with the requirement that $I^\E_{r,s}(f) = I^r_s(f)$ in the case that $G$ is quasisplit. We see that the form $S^r_s(f)$ is stable, since the terms defining it in the spectral and geometric expansions are stable, and $I^\E_{r,s}(f) = I(f)$ in general. We therefore have the decomposition
\[
I_{r,s}(f) = \sum_{G'\in\E(G,V)}\iota(G,G')\hat{S}'_{r,s}(f) 
\]
of the weighted invariant trace formula into weighted stable trace formulas for endoscopic groups.



\appendix

\section{Continuity of the stable trace formula}
\label{conts}

The goal of this appendix is twofold: the first is to prove a mild extension of the Langlands-Shelstad transfer mapping to Schwartz functions in the nonarchimedean case. In the archimedean case, the transfer is known by work of Shelstad. The second goal is to prove the continuity of the stable trace formula for the test function $\dot f =f\times u^V$. This builds upon the continuity of the invariant trace formula $I$ proved in \cite[Theorem 2]{witf}. Our strategy is to stabilise the linear form $I$ that we have shown to be valid on $\C^\circ(G,V,\zeta)$.

\subsection{Transfer} We refer to \cite{STF1} and the references therein for precise definitions of the objects that we consider here. In particular, we will introduce noncompactly supported variants of the many objects involved in the stabilization of the trace formula. Let $\E(G_v)$ be the set of endoscopic data $(G'_v,\mathcal G'_v, s'_v,\xi'_v)$ for $G$ over $F_v$, represented by $G'_v$. We shall aslo assume that the auxiliary data $\tilde{G}'\to G'$ and $\tilde\xi':\mathcal G'\to {^L\tilde{G}'}$ to be chosen according to \cite[Lemma 7.1]{STF1}. 
The Langlands-Shelstad transfer conjecture states that for any $G_v'\in\E(G_v)$, the map that sends $f\in \H(G_v,\zeta_v)$ to the function 
\[
f'(\delta') = f^{G'}(\delta') = \sum_{\gamma\in\Gamma_\text{reg}(G_v,\zeta_v)}\Delta(\delta',\gamma)f_G(\gamma)
\]
on $\Delta_\text{reg}(\tilde{G}'_v,\tilde\zeta'_v)$ exists, and maps $\H(G_v,\zeta_v)$ continuously to the space $\SI(\tilde{G}'_v,\tilde\zeta'_v)$. In the nonarchimedean case, the Langlands-Shelstad transfer was proved for smooth functions $f\in C_c^\infty(G_v)$ as a consequence of \cite{lemme} and the solution of the Fundamental Lemma \cite{Ngo}, and thus holds also for $\zeta^{-1}$-equivariant space $C_c^\infty(G_v,\zeta_v)$. Moreover, since the orbital integrals are tempered distributions, it makes sense to formulate the smooth transfer for the larger Schwartz space $\C(G_v,\zeta_v)$, in which case the transfer would lie in the corresponding space of stable orbital integrals $S\C(\tilde{G}'_v,\tilde{\zeta}'_v)$ of functions in $\C(\tilde{G}'_v,\tilde{\zeta}'_v)$ \cite[\S3]{ArtTW}. Recall that we are taking $G$ to be a $K$-group, so if $f$ equals $\oplus_\alpha f_\alpha$, then
\[
f' = \sum_{\alpha\in\pi_0(G)} f'_\alpha. 
\]
The Langlands-Shelstad transfer for Schwartz functions is then a simple consequence of the smooth transfer. We note that in the archimedean case, the result for Schwartz functions follows from work of Shelstad (c.f. \cite{shels}). The proof relies on the fact that $C_c^\infty(G_v,\zeta_v)$ is a dense subspace of $\C(G_v,\zeta_v)$, topologized by the family of seminorms used to define the Harish-Chandra Schwartz space. Moreover, the spaces $I\C(G_v,\zeta_v)$ and $S\C(G_v,\zeta_v)$ are topologized in a manner such that the maps $f\to f_G$ and $f\to f^G$ respectively are continuous. That is, in  \cite[p.187]{STF1} Arthur defines the spaces
\[
I\mathcal H(G^H_V,\zeta_V) = \{f_G\in \mathcal H(G^H_V,\zeta_V)\}
\]
and
\[
I\mathcal C(G^H_V,\zeta_V) = \{f_G\in \mathcal C(G^H_V,\zeta_V)\},
\]
both topologized such that the surjective map $f\to f_G$ is open and continuous. This describes the topologies of the source and targets, and in particular it follows that $I\mathcal H(G^H_V,\zeta_V)$ is dense in $I\mathcal C(G^H_V,\zeta_V)$. Similarly we follow p.188 of [STF1] to define the stably invariant subspaces
\[
SI\mathcal H(G^H_V,\zeta_V) = \{f^G\in \mathcal H(G^H_V,\zeta_V)\}
\]
and
\[
S\mathcal C(G^H_V,\zeta_V)=SI\mathcal C(G^H_V,\zeta_V) = \{f^G\in \mathcal C(G^H_V,\zeta_V)\},
\]
spanned by strongly regular, stable orbital integrals.

\begin{lem}
\label{FL}
Let $F_v$ be a nonarchimedean local field. Then for $f\in \C(G_v,\zeta_v)$, the map from $f$ to the function 
\[
f'(\delta') = f^{G'}(\delta') = \sum_{\gamma\in\Gamma_\mathrm{reg}(G_v,\zeta_v)}\Delta(\delta',\gamma)f_G(\gamma)
\]
on $\Delta_\text{reg}(\tilde{G}'_v,\tilde\zeta'_v)$ exists, and maps $\C(G_v,\zeta_v)$ continuously to $S\C(G_v,\zeta_v)$. 
\end{lem}

\begin{proof}
Given $f$ in $\C(G_v,\zeta_v)$, we may choose a sequence $(f_n)$ in $C_c^\infty(G_v,\zeta_v)$ converging to $f$ as $n$ tends to infinity, namely,
\[
\nu(f - f_n) \to 0,
\]
where $\nu$ is any seminorm used to define the Schwartz space. Applying the Langlands-Shelstad transfer, it follows then that there is a family of transfers $(f'_n)$ in $I\C(\tilde{G}'_v,\tilde\zeta'_v)$ such that for any $\delta'\in\Delta_\text{reg}(\tilde{G}'_v,\tilde\zeta'_v)$, we have
\[
f'_n(\delta') = \sum_{\gamma\in\Gamma_\mathrm{reg}(G_v,\zeta_v)}\Delta(\delta',\gamma)f_{n,G}(\gamma)
\]
in $\SI(\tilde{G}'_v,\tilde\zeta'_v)$. 

Estimating then the difference 
\[
| f'_n(\delta') - f_{n+1}'(\delta')| \le  \sum_{\gamma\in\Gamma_\mathrm{reg}(G_v,\zeta_v)}|\Delta(\delta',\gamma)||f_{n,G}(\gamma) - f_{n+1,G}(\gamma)|
\]
for any fixed $\delta'$, where we note that the sums are finite since the orbital integral of $f_n$ is compactly supported on the regular set for any $n$, we see that the difference 
\[
|f_{G}(\gamma) - f_{n,G}(\gamma)| \ll_\gamma \nu(f- f_{n})
\]
vanishes for $f_G$ in  $I\mathcal C(G_v,\zeta_v)$ and $f_{n,G}$ in  $I\H(G_v,\zeta_v)$. Here we have used the standard estimate on orbital integrals, 
\[
|f_G(\gamma)|\le \nu(f)(1+|\log|D(\gamma)||)^p(1+||H_G(\gamma)||)^{-n}
\]
for any integer $n>0$, $p\in\R$, and $f\in \C(G_v,\zeta_v)$ (c.f. \cite[(5.7)]{fourier}). It follows that $f_n(\delta')$ converges in $S\I(\tilde{G}'_v,\tilde\zeta'_v)$, and by continuity in $S\C(\tilde{G}'_v,\tilde\zeta'_v)$. By completeness, we denote by $f'$ the function in $\C(\tilde{G}'_v,\tilde\zeta'_v)$ such that $f'_n$ converges to $f'$. We note that the choice of $f'$ is unique only up to stable conjugacy, and satisfies the identity 
\[
f'(\delta')= \sum_{\gamma\in\Gamma_\mathrm{reg}(G_v,\zeta_v)}\Delta(\delta',\gamma)f_G(\gamma).
\] 
as required.
\end{proof}

The stabilization of the trace formula relies on the local results of Arthur on orbital integrals such as in \cite{LCR,ArtTW,asymp,parabolic, realgerms}. In order to stabilize the invariant linear form $I(f)$ for $f\times u^V$ with $f\in\C^\circ(G,V,\zeta)$, we note that the local results of Arthur above hold for general Schwartz functions $f\in \C(G,V,\zeta)$ either as explicitly stated, or otherwise can be shown using the fact that the linear forms $I_M(\gamma,f)$ extend to tempered distributions on $G$. These local transfer mappings are required to construct the stable basis $\Delta(G_V^Z,\zeta_V)$ that is used to index the geometric side of the stable trace formula. 

We have to show that this construction holds in our case also. While Arthur's stabilization is carried out for functions $f$ belonging to $\H(G,V,\zeta)$, his construction of these spaces holds generally for functions in $\C^\circ(G,V,\zeta)$, as long as the transfer mappings exist. We shall summarize Athur's construction here, extended to the slightly more general setting of $\C^\circ(G,V,\zeta)$. In the local setting, we can often work with the full Schwartz space $\C(G,V,\zeta)$.

\subsection{Geometric transfer factors} 
\label{A2}

As usual, if $S'$ is a stable, tempered $\tilde{\zeta}'$-equivariant distribution on $\tilde{G}'(F_v)$, then we write $\hat{S}'$ for the corresponding continuous linear form on $S\C(\tilde{G}'_v,\tilde\zeta'_v)$. Applying the transfer to each of the components $G_{\alpha_v}$ of $G_v$, we have a mapping 
\[
f_v \to f_v' = f^{\tilde{G}'}_v
\]
from $\C(G_v,\zeta_v)$ to $S\C(\tilde{G}'_v,\tilde\zeta'_v)$, which can be identified with a mapping
\[
a_v\to a'_v
\]
from $I\C(G_v,\zeta_v)$ to $S\C(\tilde{G}'_v,\tilde\zeta'_v)$. It follows that the product mapping from $\prod_v a_v$ to $\prod_v a'_v$ gives a linear mapping from $I\C(G_V,\zeta_V)$ to $S\C(\tilde{G}'_V,\tilde\zeta'_V)$. This mapping is attached to the product $G'_V$ of the data $G'_v$, which we can think of as the endoscopic data of $G$ over $F_V$. Letting $G'_V$ vary, we obtain a mapping
\[
I\C(G_V,\zeta_V) \to \prod_{G'_V}S\C(\tilde{G}'_V,\tilde{\zeta}'_V) 
\]
by putting together the individual images of $a'$. The image $I\C^\E(G_V,\zeta_V)$ of $I\C(G_V,\zeta_V)$ fits into a sequence of inclusions
\be
\label{inclusions}
I\C^\E(G_V,\zeta_V) \subset \bigoplus_{\{G'_V\}} I\C^\E(G'_V,G_V,\zeta_V) \subset \prod_{\Delta_V} S\C(\tilde{G}'_V,\tilde{\zeta}'_V) 
\ee
in which the summand $ I\C^\E(G'_V,G_V,\zeta_V)$ is a vector space of families of functions on $\tilde{G}'$ parametrized by transfer factors for $G$ and $\tilde{G}'$, depending only on the $F_V$-isomorphism class of $G'_V$. 

The mappings of functions have dual analogues for distributions. Given $G'_V$ with auxiliary data $\tilde{G}'_V$ and $\tilde{\xi}'_V$, assume that $\delta'$ belongs to the space of stable distributions $S\D((\tilde{G}'_V)^{\tilde{Z}'_V},\tilde{\zeta}_V')$. By Lemma \ref{FL}, we may evaluate the transfer $f'$ of any function $f$ in $\C^\circ(G,V,\zeta)$ at $\delta'$. Since the distribution $f\to f'(\delta')$ belongs to $\D(G^Z_V,\zeta_V)$, we can construct the extended geometric transfer factors at each local place
\[
\Delta(\delta',\gamma), \qquad G'\in\E(G), \delta'\in \Delta(\tilde{G}',\tilde\zeta'), \gamma\in\Gamma(G,\zeta).
\]
defined for fixed bases $\Delta(\tilde{G}',\tilde\zeta')$ of the spaces $S\D(\tilde{G}',\tilde\zeta')$ such that 
\be
\label{fdelta}
f'(\delta')  = \sum_{\gamma\in\Gamma(G,\zeta)}\Delta(\delta',\gamma)f_G(\gamma)
\ee
holds for $ \delta'\in \Delta(\tilde{G}',\tilde\zeta')$ and $f\in \C(G,\zeta)$. We can see that the extended local transfer factor, as a function on $\Delta(\tilde{G}',\tilde\zeta')\times \Gamma(G,\zeta)$ is defined in the exact same manner as \cite[\S4]{STF1}, and depends linearly on $\delta'$. We can then define the global transfer factor as the corresponding product
\[
\Delta(\delta,\gamma) = \prod_{v\in V} \Delta(\delta_v,\gamma_v)
\]
for $\delta\in\Delta^\E(G_V,\zeta_V)$ and $\gamma\in\Gamma(G_V,\zeta_V)$. The sequence of inclusions \eqref{inclusions} is dual to a sequence of surjective linear mappings
\be
\label{Dmaps}
\prod_{G'_V}S\D((\tilde{G}'_V)^{\tilde{Z}'},\tilde{\zeta}_V') \to \bigoplus_{\{G'_V\}}\D^\E(G'_V,G^Z_V,\zeta_V)\to \D^\E(G^Z_V,\zeta_V)
\ee
between spaces of distributions. Since $f'$ is the image of the function $f_G$ in $I\C(G,V,\zeta)$, it follows that $f'(\delta')$ depends only on the image $\delta$ of $\delta'$ in $\D^\E(G^Z_V,\zeta_V)$. In other words,
\[
f'(\delta') = f^\E_G(\delta),
\]
where $f^\E_G$ is the image of $f_G$ in $I\C^\E(G,V,\zeta)$, so that by the adjoint relations satisfied by the geometric transfer factor \cite[\S5]{STF1} the map $f_G\to f_G^\E$ is an isomorphism. The same is true therefore of the coefficients $\Delta_G(\delta',\gamma)$, so we may write
\[
\Delta_G(\delta,\gamma)=\Delta_G(\delta',\gamma)
\]
for $\gamma\in \Gamma(G^Z_V,\zeta_V)$ and complex numbers $\Delta_G(\delta,\gamma)$ that depend linearly on $\delta\in \D^\E(G^Z_V,\zeta_V)$. The image in $\Delta^\E(G^Z_V,\zeta_V)$ of the subspace 
\[
S\D((G^*_V)^{Z^*_V},\zeta^*_V)\simeq S\D(G^*_V,G^Z_V,\zeta_V)
\]
can be identified with the subspace $S\D(G^Z_V,\zeta_V)$ of stable distributions in $\D(G^Z_V,\zeta_V)$.

The coefficients in the geometric expansion should really be regarded as elements in the appropriate completion of $\D(M^Z_V,\zeta_V)$ and $S\D(M^Z_V,\zeta_V)$, which we shall identify with the dual space of $\D(M^Z_V,\zeta_V)$ by fixing suitable bases $\Gamma(M^Z_V,\zeta_V)$ and $\Delta(M^Z_V,\zeta_V)$ of the relevant spaces of distributions. In particular, we shall fix a basis $\Delta((\tilde{G}'_V)^{\tilde{Z}'_V},\tilde{\zeta}_V')$ of $S\D((\tilde{G}'_V)^{\tilde{Z}'_V},\tilde{\zeta}_V')$ for any $F_V$-endoscopic datum $G'_V$ with auxiliary data $\tilde{G}'_V$ and $\tilde{\xi}_V'$. We also fix a basis $\Delta^\E(G_V^Z,\zeta_V)$ of $\D^\E(G_V^Z,\zeta_V)$ such that 
\[
\Delta(G_V^Z,\zeta_V) = \Delta^\E(G_V^Z,\zeta_V)\cap S\D(G_V^Z,\zeta_V)
\]
forms a basis of $S\D(G_V^Z,\zeta_V)$, and in the case that $G$ is quasisplit, that $\Delta(G_V^Z,\zeta_V)$ is isomorphic to the image of the basis $\Delta((G_V^*)^{Z^*},\zeta_V)$. 

\subsection{Spectral transfer factors} 

The construction on the spectral side is parallel. In place of the spaces of distributions described by \eqref{Dmaps}, we have the spectral analogue $\F(G^Z_V,\zeta_V)$ of $\D(G^Z_V,\zeta_V)$, and the sequence of maps
\[
\prod_{G'_V}S\F((\tilde{G}'_V)^{\tilde{Z}'},\tilde{\zeta}_V') \to \bigoplus_{\{G'_V\}}\F^\E(G'_V,G^Z_V,\zeta_V)\to \F^\E(G^Z_V,\zeta_V).
\]
In place of the basis $\Gamma(G^Z_V,\zeta_V)$ of $\D(G^Z_V,\zeta_V)$, we have the basis 
\[
\Pi(G^Z_V,\zeta_V) = \coprod_{t\ge 0 }\Pi_t (G^Z_V,\zeta_V)
\]
of $\F(G^Z_V,\zeta_V)$ consisting of irreducible characters. If $\phi'$ belongs to $S\F((\tilde{G}'_V)^{\tilde{Z}'},\tilde{\zeta}_V')$, then the distribution $f \to f'(\phi')$ belongs to $\F(G_V^Z,\zeta_V)$,  we can construct the spectral transfer factors at each local place
\[
\Delta(\phi',\pi), \qquad G'\in\E(G), \phi'\in \Phi(\tilde{G}',\tilde\zeta'), \pi\in\Pi(G,\zeta)
\]
defined for fixed bases $\Phi(\tilde{G}',\tilde\zeta')$ of the spaces $S\F(\tilde{G}',\tilde\zeta')$ such that 
\[
f'(\phi')  = \sum_{\pi\in\Pi(G_v,\zeta_v)}\Delta(\phi',\pi)f_G(\pi)
\]
holds for $ \phi'\in \Phi(\tilde{G}',\tilde\zeta')$ and $f\in \C(G,\zeta)$, parallel to \eqref{fdelta}. We also define the corresponding product
\[
\Delta(\phi,\pi) = \prod_{v\in V} \Delta(\phi_v,\pi_v)
\]
for $\phi\in\Phi^\E(G_V,\zeta_V)$ and $\pi\in\Pi(G_V,\zeta_V)$. Given an element $\phi'$ in $S\F((\tilde{G}'_V)^{\tilde{Z}'},\tilde{\zeta}_V')$, We have that $f'(\phi')$ depends only on the image $\phi$ of $\phi'$ in $\F^\E(G_V^Z,\zeta_V)$, that is,
\[
f'(\phi') = f^\E_G(\phi),
\]
and the spectral coefficients satisfy the relation
\[
\Delta_G(\phi,\pi) = \Delta_G(\phi',\pi)
\]
for $\pi \in \Pi(G^Z_V,\zeta_V)$ and  complex numbers $\Delta_G(\phi,\pi)$ that depend linearly on $\phi \in \F^\E(G_V^Z,\zeta_V)$. They satisfy adjoint relations parallel to the geometric transfer factors. Here as in \cite[\S5]{STF1} we shall fix an endoscopic basis $\Phi^\E(G_V^Z,\zeta_V)$ of $\F(G_V^Z,\zeta_V)$, and a subset
\[
\Phi(G_V^Z,\zeta_V) = \Phi^\E(G_V^Z,\zeta_V)  \cap S\F(G_V^Z,\zeta_V)
\]
that forms a basis of $S\F(G_V^Z,\zeta_V)$, and in the case that $G$ is quasisplit, such that $\Phi(G_V^Z,\zeta_V)$ is isomorphic to the image of the basis $\Phi((G_V^*)^{Z^*},\zeta_V)$. If $v$ is archimedean, we can identify $\Phi(\tilde{G}_v',\tilde\zeta'_v)$ with the relevant set of Langlands parameters. If $v$ is nonarchimedean, we construct $\Phi(\tilde{G}_v',\tilde\zeta'_v)$ in terms of abstract bases $\Phi_\text{ell}(G_v,\zeta_v)$ of the cuspidal subspaces $\SI_\text{cusp}(M_v,\zeta_v)$, and similar objects for endoscopic groups $M'$ of $M$, where we observe that the relevant constructions of \cite{LCR} extend readily to $\C(G_v,\zeta_v)$ (see \cite[p.825]{STF3}).

\subsection{The stable and endoscopic expansions}
\label{A4}

Having defined the relevant objects, we now turn to the continuity of the stable trace formula. As before, our attention will be on extending the arguments in \cite{STF1,STF2,STF3}, which will essentially follow from properly constructing the natural generalizations of the required objects. As the stabilization of the trace formula involves a much more intricate argument than that needed for the invariant trace formula, we are forced to follow the same path here. We note that a similar argument is provided in \cite{MW1,MW2} for the stabilization of the twisted trace formula.  

\begin{thm}
\label{main1}
The linear forms $I^\E$ and $S$ extend continuously from $\H(G,V,\zeta)$ to $\C(G,V,\zeta)$.
\end{thm}
\begin{proof}
We first observe that Global Theorem 1$'$ in \cite[\S7]{STF1} states that the global geometric coefficients satisfy 
\[
a^{G,\E}(\gamma)= a^G(\gamma), \qquad \gamma\in \Gamma^\E(G,V,\zeta)
\]
for any $G$, and that
\[
b^G(\delta), \qquad \delta\in \Delta^\E(G,V,\zeta)
\]
vanishes on the complement of $\Delta(G,V,\zeta)$ if $G$ is quasisplit.  Notice that $\Gamma^\E(G,V,\zeta)$ and $\Delta(G,V,\zeta)$ are constructed as subsets of bases $\Gamma(G_V^Z,\zeta_V)$ and $\Delta(G^Z_V,\zeta_V)$ of the spaces $\D(G_V^Z,\zeta_V)$ and $S\D(G_V^Z,\zeta_V)$ respectively. In particular, we see that this space contains the orbital integrals, and also derivatives of orbital integrals in the archimedean cases, of functions $f$ in $\C(G,V,\zeta)$. Similarly, Global Theorem 2$'$ states that the global geometric coefficients satisfy 
\[
a^{G,\E}(\pi)= a^G(\pi), \qquad \pi\in \Pi_t^\E(G,V,\zeta)
\]
for any $G$, and that 
\[
b^G(\phi), \qquad \delta\in \Phi^\E_t(G,V,\zeta)
\]
vanishes on the complement of $\Phi_t(G,V,\zeta)$ if $G$ is quasisplit. Here the spaces 
\[
\Pi_t^\E(G,V,\zeta),\quad \Phi_t^\E(G,V,\zeta),\quad \Phi_t(G,V,\zeta)
\]
are the subset of elements in 
\[
\Pi^\E(G,V,\zeta),\quad \Phi^\E(G,V,\zeta),\quad \Phi(G,V,\zeta)
\]
respectively whose archimedean infinitesimal characters $\nu$ have norms $t = ||\text{Im}(\nu)||$. Notice that $\Phi_t(G,V,\zeta)$ and $\Pi_t^\E(G,V,\zeta)$ are constructed as discrete subsets of the bases $\Pi^\E_t(G_V^Z,\zeta_V)$ and $\Phi_t(G^Z_V,\zeta_V)$ of the spaces $\F(G_V^Z,\zeta_V)$ and $S\F(G_V^Z,\zeta_V)$ respectively. As we have indicated above, in both the geometric and spectral cases, the construction of the endoscopic spaces implicitly rely on the Langlands-Shelstad transfer, hence by Lemma \ref{FL} these spaces exist unconditionally.

Let $S$ be a finite set of valuations containing $V$. There is a natural map
\[
f\mapsto \dot{f}_S=f\times u^V_S
\]
from $\C(G,V,\zeta)$ to $\C(G,S,\zeta)$. We shall define an admissible subspace $\C_\text{adm}(G,S,\zeta)$ of $\C(G,S,\zeta)$, using the same notion of admissibility in \cite[\S1]{STF1}. The polynomial $\det(1+t-\text{Ad}(x)) = \sum_{k}D_k(x)t^k$ for $x\in G$ defines a morphism
\[
\mathcal D = (D_0,\dots,D_d): G \to {\bf G}_a^{d+1}
\]
where $d=\dim G$. If $X$ is a nonzero point in ${\bf G}_a^{d+1}$, we shall denote $X_\text{min} = X_k$ where $k$ is the smallest integer such that $X_k$ is nonzero. Let $\O^S$ be product of all $v$ not in $S$ of $\O_v$, the ring of integers of $F_v$. We call a subset $C_S$ of $F^{d+1}_S\backslash\{0\}$ admissible if any point $X$ in the intersection
\[
F^{d+1}\cap (C_S\times (\O^S)^{d+1})
\]
satisfies $|X_\text{min}|_v=1$ for all $v\not\in S$. Assume moreover that $S$ contains the places over which $G$ and $\zeta$ are ramified, and that $Z(A) = Z(F) Z_SZ(\O^S)$. Then we call a subset $\Delta_S$ of $G_S$ admissible if $\mathcal D(\Delta_S)$ is admissible in $F^{d+1}_S$. This implies that
\[
|D(\dot\gamma)|_v = 1
\]
for all $\gamma\in G(F)\cap (\Delta_S\times K^S)$ and $v\not\in S$. Also, $\Delta_S$ is admissible if and only if its projection onto $\bar{G}_S= G_S/Z_S$ is admissible. Finally, we define  $\C_\text{adm}(G,S,\zeta)$ to be the subspace of functions in $\C(G,S,\zeta)$ whose support is admissible. Also, we shall say a subset $\Delta$ of $G(\A)$ is $S$-admissible if for some finite set $S$ there is an admissible subset $C_S$ of $F^{d+1}_S$ such that $\mathcal D(\Delta)$ is contained in $C_S\times (\O^S)^{d+1}$. We note that it is this condition of admissibility and $S$-admissibility that the reductions of \cite{STF1,STF2,STF3} are based upon, rather than the compact support of the test functions $f$.

Having made these preliminary remarks, we now proceed as follows. Let $I$ be the invariant linear form on $\C^\circ(G,V,\zeta)$ obtained in \cite[Theorem 2]{witf}. If $G$ is arbitrary, we define an endoscopic linear form inductively by setting
\[
I^\E(f) = \sum_{G'\in\E_\text{ell}(G,V)}\iota(G,G')\hat{S}'(f')
\]
for stable linear forms $\hat{S}'=\hat{S}^{\tilde{G}'}$ on $S\C(G,V,\zeta)$. In the case that $G$ is quasisplit, we define a linear form 
\[
S^G(f) = I(f) - \sum_{G'\in\E_\text{ell}^0(G,V)}\iota(G,G')\hat{S}'(f')
\]
and also the endoscopic linear form by the trivial relation 
$$
I^\E(f) = I(f).
$$
We assume inductively that if $G$ is replaced by a quasisplit inner $K$-form of $\tilde{G}'$, the corresponding analogue of $S^{G}$ is defined and stable. At this stage, the reductions of \cite{STF1,STF2} can now be applied without difficulty. In particular, if on the geometric side, we define $I^\E_\text{orb}(f)$ and $S_\text{orb}^G(f)$ to be the summands corresponding to $M=G$ in $I^\E(f)$ and $S^G(f)$ respectively, we see from the proof of \cite[Theorem 10.1]{STF1} that if $G$ is arbitrary,
\[
I^\E(f) - I^\E_\text{orb}(f) = \sum_{M\in\L^0}|W^M_0||W^G_0|^{-1}\sum_{\gamma\in\Gamma(M,V,\zeta)} a^{M,\E}(\gamma)I^\E_M(\gamma,f)
\]
and if $G$ is quasisplit, we have that $S^G(f) - S^G_\text{orb}(f)$ is equal to
\[
\sum_{M\in\L^0}|W^M_0||W^G_0|^{-1}\sum_{M'\in\E'_\text{ell}(M,V)}\iota(M,M')\sum_{\delta'\in\Delta(\tilde{M}',V,\tilde\zeta')} b^{\tilde{M}'}(\delta')S^G_M(M',\delta',f).
\]
Here $I^\E_M(\gamma,f)$ and $S^G_M(M',\delta',f)$ are the local geometric distributions defined in \cite[\S6]{STF1}. While on the spectral side, we define $I^\E_{t,\text{unit}}(f)$ and $S^G_{t,\text{unit}}$ using the decomposition according to the norm of the archimedean infinitesimal character,
\[
I^\E(f) = \sum_{t\ge0} I_t^\E(f) 
\]
and 
\[
S^G(f) = \sum_{t\ge0} S_t^G(f).
\]
It follows then from the proof of \cite[Theorem 10.6]{STF1} that if $G$ is arbitrary,
\[
I^\E_t(f) - I^\E_{t,\text{unit}}(f) = \sum_{M\in\L^0}|W^M_0||W^G_0|^{-1}\int_{\Pi^\E_t(M,V,\zeta)} a^{M,\E}(\pi)I^\E_M(\pi,f)d\pi
\]
and if $G$ is quasisplit, we have that $S^G_t(f) - S^G_{t,\text{unit}}(f)$ is equal to
\[
\sum_{M\in\L^0}|W^M_0||W^G_0|^{-1}\sum_{M'\in\E'_\text{ell}(M,V)}\iota(M,M')\int_{\Phi_{t'}(\tilde{M}',V,\tilde\zeta')} b^{\tilde{M}'}(\phi')S^G_M(M',\phi',f)d\phi'.
\]
Here again $I^\E_M(\pi,f)$ and $S^G_M(M',\phi',f)$ are the local spectral distributions defined in \cite[\S6]{STF1}.  These identities reduce the study of the global geometric coefficients $a^{G,\E}(\gamma), b^G(\delta)$ and global spectral coefficients $a^{G,\E}(\pi),b^G(\phi)$ to the terms $M=G$ in their expansion, namely $a^{G,\E}_\text{ell}(\gamma), b^G_\text{ell}(\delta)$ and $a^{G,\E}_\text{disc}(\pi),b^G_\text{disc}(\phi)$ respectively by the arguments of Propositions 10.3 and 10.7 of \cite{STF1}. Moreover, the global descent formula of \cite[Corollary 2.2]{STF2} further reduces the study of the global geometric coefficients to unipotent elements. (Our extension of the notion of admissibility is crucial here for the extension of this result, which is a long but straightforward verification.) More precisely, given an admissible element $\dot\gamma_S$ in $\Gamma^\E_\text{ell}(G,S,\zeta)$ with Jordan decomposition $\dot\gamma_S=c_S\dot\alpha_S$, we have
\[
a^{G,\E}_\text{ell}(\dot\gamma_S) = \sum_c\sum_{\dot\alpha}i^{\bar{G}}(S,c)|\bar{G}_{c,+}(F)/\bar{G}_c(F)|^{-1} a^{G_c,\E}_\text{ell}(\dot\alpha)
\]
and if $G$ is quasisplit, given an  admissible element $\dot\delta_S$ in $\Delta_\text{ell}(G,S,\zeta)$ with Jordan decomposition $\dot\delta_S=d_S\dot\beta_S$, we have
\[
b^{G}_\text{ell}(\dot\delta_S) = \sum_d\sum_{\dot\beta}j^{\bar{G}^*}(S,d)|\bar{G}^*_{d,+}(F)/\bar{G}^*_d(F)|^{-1} b^{G_d^*}_\text{ell}(\dot\beta),
\]
where $G_{c,+}$ denotes the centralizer of $c$ in $G$, and $G_c$ is the identity component of $G_{c,+}$. We refer the reader to \cite{STF2} for complete definitions of these expressions.

Turning to the local setting, the analogues of Local Theorems 1 and 2 of \cite[\S6]{STF1} for $f\in\C(G,V,\zeta)$ follow from the analogues of Local Theorems 1$'$ and 2$'$, which concern the compound linear forms $I^\E_M(\gamma,f),$ $S^G_M(M',\delta',f)$ and $I^\E_M(\pi,f),$ $S^G_M(M',\phi',f)$ as a consequence of the geometric and spectral splitting and descent formulae respectively \cite[Propositions 6.1 and 6.3]{STF1}. We recall that the required  geometric formulae are given in \cite[\S3]{stablegerms}, whereas the spectral formulae can be deduced from \cite[X.4]{MW2}.

To apply the arguments of \cite{STF3}, we require analogous constructions of various subspaces of the Hecke space $\H(G,V,\zeta)$ used therein. If $G$ is quasisplit, we define the unstable subspace 
\[
\C^{\circ, \text{uns}}(G,V,\zeta)
\]
of functions $f\in\C(G,V,\zeta)$ such that $f^G=0$. It is spanned by functions $f = \prod_{v\in V} f_v$ such that for some $v\in V$, $f_v$ satisfies the property that $f^G_v=0$. We shall also define the subspace 
\[
\C_M(G,V,\zeta)
\]
of functions $f \in \C(G,V,\zeta)$ such that $f_v$ is $M$-cuspidal at two places $v\in V$. Recall that $f_v\in \C(G_v,\zeta_v)$ is said to be $M$-cuspidal if $f_{v,L_v}=0$ for any element $L_v\in \L_v$ that does not contain a $G_v$-conjugate of $M_v$. If $v$ is a nonarchimedean place, we define
\[
\C(G_v,\zeta_v)^{00}
\]
to be the subspace of functions $f\in \C(G_v,\zeta_v)$ such that $f_{v,G}(z_v\alpha_v)=0$ for any $z_v$ in the center of $\bar{G_v}=G_v/Z_v$ and $\alpha_v$ in the basis $R_\text{unip}(G_v,\zeta_v)$ of unipotent orbital integrals in \cite[\S3]{STF3}. We lastly define 
\[
\C(G_v,\zeta_v)^{0}
\]
analogously, with $\alpha_v$ ranging over the parabolic subset $R_\text{unip,par}(G_v,\zeta_v).$ We also write $\C(G,V,\zeta)^{0}$ for the product of functions $f_v \in \C(G_v,\zeta_v)^{0}$ for $v\in V$, and similarly for $\C(G,V,\zeta)^{00}$. We shall denote by the intersections of these various spaces by using overlapping notation, for example, we write  $\C^{\circ}_M(G_v,\zeta_v)^0 = \C^{\circ}_M(G_v,\zeta_v)\cap \C(G_v,\zeta)^0$. 

The remainder of the proof proceeds by a double induction on integers $r_\mathrm{der}$ and $d_\mathrm{der}$ such that 
\[
0 < r_\mathrm{der} < d_\mathrm{der}.
\]
Namely, we assume inductively that Local Theorem 1 holds if $\dim(G_\text{der}) < d_\mathrm{der}$ and if 
\[
\dim(G_\mathrm{der}) = d_\mathrm{der}, \quad \dim(A_M\cap G_\text{der}) < d_\mathrm{der}
\]
for a local non-archimedean field; the archimedean transfer for $f\in \C(G,\zeta)$ follows from \cite[Theorem 1.1]{parabolic}. We assume inductively that Global Theorems 1 and 2 hold if $\dim(G_\text{der}) < r_\text{der}$. In both local and global cases, we assume that if $G$ is not quasisplit and $\dim(G_\mathrm{der}) = d_\mathrm{der}$, the relevant theorems hold for the quasisplit inner $K$-form of $G$. 

We will be content with recapitulating the broad strokes of the arguments in \cite{STF3}, indicating the points in which we use the new spaces that we have defined above instead of the compactly supported ones. We shall use the subscripts `unip' to denote the unipotent variant of objects with subscript `ell,' and `par' with the objects corresponding to terms $M\neq G$. For example, we write
\[
I_\text{unip}(f,S) = \sum_{\alpha\in\Gamma_\text{unip}(G,V,\zeta)}a^G_\text{unip}(\alpha,S)f_G(\alpha)
\]
where
\[
a^G_\text{unip}(\alpha,S) = \sum_{k\in \K^V_\text{unip}((\bar{G},S)}a^G_\text{ell}(\alpha\times k) r_G(k)
\]
for $\alpha\in \Gamma_\text{unip}(G,V,\zeta)$. By the inductive definitions we obtain $I^\E_{\text{unip}}$ and $S^G_\text{unip}$ analogously. The global induction hypothesis then implies that 
\[
I^\E_\text{par}(f) -I_\text{par}(f)  = \sum_{t}\left(I^\E_{t,\text{disc}}(f) - I_{t,\text{disc}}(f)\right) - \sum_{z}\left(I^\E_{t,\text{unip}}(f,S) - I_{t,\text{unip}}(f,S)\right)
\]
and if $G$ is quasisplit and $f$ belongs to $\C^{\circ, \text{uns}}(G,V,\zeta)$, then
\[
S^G_\text{par}(f) = \sum_{t}S^G_{t,\text{disc}}(f) - \sum_{z}S^G_{z,\text{unip}}(f,S),
\]
where $z$ belongs to the quotient $Z(G)_{V,\o}Z_V/Z_V$, and $Z(G)_{V,\o}$ is the subgroup of elements in $(Z(G))(F)$ such that for every $v\not\in V$, the element $z_v$ is bounded in $(Z(G))(F_v)$. The induction hypotheses further lead to a cancellation of $p$-adic singularities, allowing us to express
\[
I^\E_\text{par}(f) -I_\text{par}(f)  = |W(M)|^{-1}\hat{I}^M(\varepsilon_M(f))
\]
for $f$ in the intersection $\C_M(G,V,\zeta)^{0}$. and if $G$ is quasisplit,
\[
S^G_\text{par}(f) = |W(M)|^{-1}\sum_{M'\in\E_\text{ell}(M,v)} \iota(M,M')\hat{S}^{\tilde{M}'}(\varepsilon^{M'}(f))
\]
for $f$ in the intersection $\C^{\circ,\text{uns}}_M(G,V,\zeta)^0$ as in \cite[Corollary 3.3]{STF3}. Here $\varepsilon_M$ is a map from $\C(G_v,\zeta_v)^0$ to the subspace of cuspidal functions in $\I_\text{ac}(M_v,\zeta_v)$ such that
\[
\varepsilon_M(f_v,\gamma_v) = I^\E_M(\gamma_v,f_v) - I_M(\gamma_v,f_v) 
\]
for any $\gamma_v \in \Gamma(M_v,\zeta_v)$, and in the case that $G_v$ is quasisplit, $\varepsilon^M$ is a map from $\C^{\circ,\text{uns}}(G_v,\zeta_v)^0$ to $\SI_\text{ac}(M_v,\zeta_v)$ such that 
\[
\varepsilon^M(f_v,\delta_v) = S^G_M(\delta_v,f_v)
\]
for any $\delta_v\in \Delta(M_v,\zeta_v)$. These maps are given in \cite[Proposition 3.1]{STF3}, also studied in Chapters VIII and IX of \cite{MW2}, and can be seen as generalizing the mapping $\phi_M$ in a direction different from the maps $\iota_M, \iota^\E_M,$ and $\tau_M$ that we have constructed earlier. 

The separation of the spectral sides according to infinitesimal character follows from \cite[\S4--5]{STF3} and the properties of the function spaces we have defined, but is not strictly necessary given the absolute convergence of the spectral side. On the other hand, the stabilization of the invariant local trace formula in \cite[\S10]{STF1} and \cite[\S6]{STF3} extends to our setting following Lemma \ref{FL}, and together with the global results above lead to the proof of Local Theorem 1 in the nonarchimedean case, again using the local and global induction hypotheses. We note that this implies Local Theorem 2 according to an unpublished work of Arthur, and we may also refer to sections X.5 and X.7 of \cite{MW2} for a variant argument that can be used instead. 

To complete the global theorems, we apply the local theorems to conclude that 
\[
I^\E_\text{par}(f) -I_\text{par}(f)= \sum_{M\in\L^0}|W^M_0||W_0^G|^{-1}\sum_{\gamma\in\Gamma(M,V,\zeta)}a^M(\gamma)(I^\E_M(\gamma,f) - I_M(\gamma,f))
\]
vanishes for $\C(G,V,\zeta)$, and if $G$ is quasisplit, that 
 \[
 S^G_\text{par}(f) = \sum_{M\in\L^0}|W^M_0||W_0^G|^{-1}\sum_{\delta^*\in\Delta(M^*,V,\zeta^*)}b^{M^*}(\delta^*)S^G_M(M^*,\delta^*,f)
 \]
 vanishes for $f\in \C^{\circ,\text{uns}}(G,V,\zeta)$. The induction argument on $r_\text{der}$ implies that the terms $I^\E_{t,\text{disc}}(f) - I_{t,\text{disc}}(f)$ and $S^G_{t,\text{disc}}(f)$ vanish for $f$ in $\C(G,V,\zeta)$ and $\C^{\circ,\text{uns}}(G,V,\zeta)$ respectively, so that
\[
\sum_{z}\left(I^\E_{t,\text{unip}}(f,S) - I_{t,\text{unip}}(f,S)\right)=0, \qquad f\in\C^{\circ}(G,V,\zeta)
\]
and
\[
\sum_{z}S^G_{z,\text{unip}}(f,S)=0,\qquad f\in\C^{\circ,\text{uns}}(G,V,\zeta)
\]
in the case that $G$ is quasisplit. Choosing $V=S$, and using the property that the linear forms 
\[
\dot{f}_S\to \dot{f}_{S,G}(z\dot\alpha_S), \qquad z\in Z(\bar{G})_{S,\o}, \ \dot\alpha_S\in\Gamma^\E_\text{unip}(G,S,\zeta)
\]
on the subspace of admissible functions in $\C(G,S,\zeta)$ are linearly independent, we conclude from the definitions of $I_{t,\text{unip}}$ and $I_{t,\text{unip}}^\E$ that 
\[
a^{G,\E}_\text{ell}(\dot\alpha_S)-  a^{G}_\text{ell}(\dot\alpha_S) = 0
\]
for $\dot\alpha_S\in\Gamma^\E_\text{unip}(G,S,\zeta)$, and similarly 
\[
\dot{f}_S\to \dot{f}_{S,G}^\E(z\dot\beta_S), \qquad z\in Z(\bar{G})_{S,\o}, \ \dot\beta_S\in\Delta^\E_\text{unip}(G,S,\zeta)\backslash \Delta_\text{unip}(G,S,\zeta)
\]
on the subspace of admissible functions in $\C^{\circ,\text{uns}}(G,S,\zeta)$ are linearly independent, whence we conclude that 
\[
b^{G}_\text{ell}(\dot\beta_S) = 0
\]
for $\dot\alpha_S$ in the complement of $\Delta_\text{unip}(G,S,\zeta)$ in $\Delta^\E_\text{unip}(G,S,\zeta)$. Applying the global descent formula to the coefficients then yields the geometric Global Theorem 1. The spectral Global Theorem 2 follows similarly, using the vanishing of
\[
\sum_{t}\left(I^\E_{t,\text{disc}}(f) - I_{t,\text{disc}}(f)\right) =0, \qquad \dot{f}\in\C(G,\zeta)
\]
and
\[
\sum_{t}S^G_{t,\text{disc}}(f)=0, \qquad \dot{f}\in\C^{\circ,\text{uns}}(G,\zeta).
\]
Arguing as in the geometric case we conclude that 
\[
a^{G,\E}_\text{disc}(\dot\pi)  - a^{G,\E}_\text{disc}(\dot\pi) =0
\]
for any $\dot\pi\in \Pi_{t,\text{disc}}(G,\zeta)$, and in the case that $G$ is quasisplit,
\[
b^G_\text{disc}(\dot\phi) = 0 
\]
for any $\dot\phi$ in the complement of $\Phi_{t,\text{disc}}(G,\zeta)$ in $\Phi_{t,\text{disc}}^\E(G,\zeta)$, and the desired result follows.

Finally, we can conclude from these general remarks the extension of the endoscopic and stable trace formulae, with the required expansions
\begin{align*}
I^\E(f) &= \sum_{M\in\L} |W_0^M||W_0^G|^{-1}\sum_{\gamma\in\Gamma^\E(M,V,\zeta)}a^{M,\E}(\gamma)I^\E_M(\gamma,f)\\
& =\sum_t \sum_{M\in\L} |W_0^M||W_0^G|^{-1}\int_{\Pi^\E_t(M,V,\zeta)} a^{M,\E}(\pi)I_M^\E(\pi,f)d\pi 
\end{align*}
and
\begin{align*}
S(f) &= \sum_{M\in\L} |W_0^M||W_0^G|^{-1}\sum_{\delta\in\Delta(M,V,\zeta)}b^{M}(\delta)S_M(\delta,f)\\
& =\sum_t \sum_{M\in\L} |W_0^M||W_0^G|^{-1}\int_{\Phi_t(M,V,\zeta)} b^{M}(\phi)S_M(\phi,f)d\phi 
\end{align*}
for $f\in \C(G,V,\zeta)$.
\end{proof}

We state the local consequences explicitly, which we shall require in the stabilization of the form $I^r_s(f)$. We record them separately according to the geometric and spectral sides, respectively. 
\begin{cor}
\label{cora}
Let $F$ be a number field, and let $V$ be a finite set of valuations $V_\text{ram}(G,\zeta)$.
\begin{enumerate}
\item[(a)]
If $G$ is arbitrary, 
\[
I^\E_M(\gamma,f) = I_M(\gamma,f), \qquad \gamma\in\Gamma(M^Z_V,\zeta_V), f\in \C(G,V,\zeta)
\]
\item[(b)]
Suppose that $G$ is quasisplit, and that $\delta'$ belong to $\Delta((\tilde{M}'_V)^{\tilde{Z}'},\tilde{\zeta}'_V)$ for some $M'\in \E_\mathrm{ell}(M,V)$. Then the linear form
\[
f \to S^G_M(M',\delta',f), \qquad f\in \C(G,V,\zeta)
\]
vanishes unless $M=M^*$, in which case it is stable.
\end{enumerate}
\end{cor}
\begin{cor}
\label{corb}
Let $F$ be a number field, and let $V$ be a finite set of valuations $V_\text{ram}(G,\zeta)$.
\begin{enumerate}
\item[(a)]
If $G$ is arbitrary, 
\[
I^\E_M(\pi,f) = I_M(\pi,f), \qquad \gamma\in\Pi(M^Z_V,\zeta_V), f\in \C(G,V,\zeta)
\]
\item[(b)]
Suppose that $G$ is quasisplit, and that $\phi'$ belong to $\Phi((\tilde{M}'_V)^{\tilde{Z}'},\tilde{\zeta}'_V)$ for some $M'\in \E_\mathrm{ell}(M,V)$. Then the linear form
\[
f \to S^G_M(M',\phi',f), \qquad f\in \C(G,V,\zeta)
\]
vanishes unless $M=M^*$, in which case it is stable.
\end{enumerate}
\end{cor}

\begin{proof}
The statements concern the compound linear forms that arise on either side of the endoscopic and stable trace formulas. In the proof of Theorem \ref{main1}, we established the analogues of the local theorems in \cite{STF1} which can be stated for a local field $F$ and test functions $f\in \C(G,\zeta)$. The result will follow from the necessary splitting and descent formulas in the same manner as for functions $f\in\H(G,V,\zeta)$ described in \cite[\S6]{STF1}. On the geometric side, we conclude that if $G$ is arbitrary, then
\[
I^\E_M(\gamma,f) = I_M(\gamma,f), \qquad \gamma\in \Gamma_{G\text{-reg,ell}}(M,\zeta), f\in \C(G,\zeta)
\]
whereas if $G$ is quasisplit and $\delta'$ belongs to $\Delta_{G\text{-reg}}(\tilde{M}',\tilde{\zeta}')$ for some $M'\in\E_\text{ell}(M)$, then the linear form
\[
f \to S^G_M(M',\delta',f), \qquad f\in \C(G,\zeta)
\]
vanishes unless $M'=M^*$, in which case it is stable. The desired result follows from the splitting and descent formulas for these distributions, which are proved in this setting in Propositions 3.3 and 3.4 of \cite{stablegerms}, and are also stated in a different form in \cite{MW2}. 

Similarly, on the spectral side, we conclude from the proof of Theorem \ref{main1} that if $G$ is arbitrary, then
\[
I^\E_M(\pi,f) = I_M(\pi,f), \qquad \pi\in \Pi_{G\text{-reg,ell}}(M,\zeta), f\in \C(G,\zeta)
\]
whereas if $G$ is quasisplit and $\phi'$ belongs to $\Phi_{G\text{-reg}}(\tilde{M}',\tilde{\zeta}')$ for some $M'\in\E_\text{ell}(M)$, then the linear form
\[
f \to S^G_M(M',\phi',f),\qquad f\in \C(G,\zeta)
\]
vanishes unless $M'=M^*$, in which case it is stable. The desired result follows from the splitting and descent formulas for these distributions, which second of which is stated in this setting in the proof of Propositions 4.1 of \cite{maps2}, and are also stated in a different form in \cite{MW2}.

\end{proof}

The modifications of the endoscopic and stable linear forms in \cite{maps2} then applies to this setting also. We recall that in \cite{maps2} we constructed modified distributions $\tilde{I}^\E$ and $\tilde{S}$ which lead to purely geometric expansions for the unitary part of the spectral expansion in terms of modified linear forms $\tilde{I}^\E_M(\gamma)$ and $\tilde{S}_M(\delta)$ defined in \cite[\S3]{maps2}. The modifications apply to our extension of the endoscopic and stable linear forms on $\C(G,V,\zeta)$ also. 

\begin{cor}
The modified distributions $\tilde{I}^\E$ and $\tilde{S}$ extend to continuous linear forms on $\C(G,V,\zeta)$. 
\begin{enumerate} 
\item[(a)]If $G$ is arbitrary, 
\[
I^\E_\mathrm{unit}(f) = \sum_{M\in\L}|W^M_0||W^G_0|^{-1}\sum_{\gamma\in \Gamma(M,V,\zeta)}a^{M,\E}(\gamma)\tilde{I}^\E_M(\gamma,f).
\]
\item[(b)]If $G$ is quasisplit,
\[
S_\mathrm{unit}(f) = \sum_{M\in\L}|W^M_0||W^G_0|^{-1}\sum_{\delta\in \Delta(M,V,\zeta)}b^M(\delta)\tilde{S}_M(\delta,f).
\]
\end{enumerate}
\end{cor}

\begin{proof}
This follows  from the fact that the maps $\tau_M$ and $\iota^\E_M$ used to construct the modified distributions are defined for functions in $\C(G,V,\zeta)$, whence we may apply Theorem \ref{main1} and argue as in the first part of the proof of \cite[Theorem 2]{witf}.
\end{proof}

\bibliography{BESL2}
\bibliographystyle{alpha}
\end{document}